\documentclass[reqno]{amsart}
\usepackage{amssymb,amsmath,hyperref}
\usepackage{amsrefs, float}
\usepackage[foot]{amsaddr}
\usepackage{bbold,stackrel, multicol}
 
\newcommand{\Z}{{\textsf{\textup{Z}}}}

\newtheorem{thm}{Theorem}

\newtheorem{cor}[thm]{Corollary}
\newtheorem{defi}[thm]{Definition}
\newtheorem{rem}[thm]{Remark}
\newtheorem{nota}[thm]{Notation}

\newtheorem{princ}[thm]{Principle}

\newtheorem{ack}[thm]{Acknowledgement}

\newtheorem*{tempo*}{Template}

\newcommand\be{\begin{equation}}
\newcommand\ee{\end{equation}} 

\usepackage{amsmath,amsfonts} 

\usepackage[applemac]{inputenc}

\def\bdefi{\begin{defi}\rm}
\def\edefi{\end{defi}}
\def\bnota{\begin{nota}\rm}
\def\enota{\end{nota}}
\def\FIVE{\Pi_{1}^{1}\text{-\textup{\textsf{CA}}}_{0}}
\def\SIX{\Pi_{2}^{1}\text{-\textsf{\textup{CA}}}_{0}}

\def\SIXK{\Pi_{k}^{1}\text{-\textsf{\textup{CA}}}_{0}^{\omega}}

\def\ATR{\textup{\textsf{ATR}}}

\def\ZFC{\textup{\textsf{ZFC}}}
\def\ZF{\textup{\textsf{ZF}}}

\def\L{\textsf{\textup{L}}}

 \def\r{\mathbb{r}}

\def\RCA{\textup{\textsf{RCA}}}
\def\({\textup{(}}
\def\){\textup{)}}

\def\RCAo{\textup{\textsf{RCA}}_{0}^{\omega}}
\def\ACAo{\textup{\textsf{ACA}}_{0}^{\omega}}
\def\ATRo{\textup{\textsf{ATR}}_{0}^{\omega}}
\def\WKL{\textup{\textsf{WKL}}}

\def\bye{\end{document}}
\def\N{{\mathbb  N}}
\def\Q{{\mathbb  Q}}
\def\R{{\mathbb  R}}

\def\SS{\textup{\textsf{S}}}

\def\di{\rightarrow}

\def\CBN{\textup{\textsf{CB}}\N}

\def\asa{\leftrightarrow}

\def\ACA{\textup{\textsf{ACA}}}

\def\QFAC{\textup{\textsf{QF-AC}}}

\def\FIN{\textup{\textsf{FIN}}}
\def\alt{\textup{\textsf{alt}}}
\def\QRQ{\textup{\textsf{QRQ}}}

\def\reg{\textup{\textsf{reg}}}
\def\FUT{\textup{\textsf{FUT}}}
\def\NCC{\textup{\textsf{NCC}}}

\def\INDX{\textup{\textsf{IND}}_{1}}
\def\INDY{\textup{\textsf{IND}}_{0}}

\def\boco{\textup{\textsf{Boco}}}
\def\PHM{\textup{\textsf{Pohm}}}
\def\Borel{\textup{\textsf{Borel}}}
\def\RUC{\textup{\textsf{RUC}}}
\def\CUC{\textup{\textsf{CUC}}}
\def\BUC{\textup{\textsf{BUC}}}

\def\fun{\textup{\textsf{fun}}}
\def\DCAA{\textup{\textsf{DCA}}}

\def\accu{\textup{\textsf{accu}}}

\def\HAR{\textup{\textsf{Harnack}}}

\def\cocode{\textup{\textsf{cocode}}}

\def\NIN{\textup{\textsf{NIN}}}
\def\NCC{\textup{\textsf{NCC}}}
\def\NBI{\textup{\textsf{NBI}}}

\def\closed{\textup{\textsf{closed}}}

\def\BOOT{\textup{\textsf{BOOT}}}

\def\IND{\textup{\textsf{IND}}}

\def\HBU{\textup{\textsf{HBU}}}

\def\WHBU{\textup{\textsf{WHBU}}}

\def\BW{\textup{\textsf{BW}}}
\def\BWC{\textup{\textsf{BWC}}}

\def\BWL{\textup{\textsf{BWL}}}

\def\seq{\textup{\textsf{seq}}}
\def\fin{\textup{\textsf{fin}}}

\def\LIN{\textup{\textsf{LIN}}}

\def\DCA{\Delta\textup{\textsf{-CA}}}

\def\MCT{\textup{\textsf{MCT}}}

\def\eps{\varepsilon}

\def\X{\textup{\textsf{X}}}

\def\ECF{\textup{\textsf{ECF}}}

\newcommand{\F}{{\bf F}}

\usepackage{graphicx}
\usepackage{tikz}
\usetikzlibrary{matrix, shapes.misc}
\usepackage{comment,tikz-cd}

\numberwithin{equation}{section}
\numberwithin{thm}{section}

\begin{document}
\title{Big in Reverse Mathematics: the uncountability of the reals}
\author{Sam Sanders}
\address{Institute for Philosophy II, RUB Bochum, Germany}
\email{sasander@me.com}
\begin{abstract}
The uncountability of $\mathbb{R}$ is one of its most basic properties, known far outside of mathematics. 
Cantor's 1874 proof of the uncountability of $\mathbb{R}$ even appears in the very first paper on set theory, i.e.\ a historical milestone. 
In this paper, we study the uncountability of $\R$ in Kohlenbach's \emph{higher-order} Reverse Mathematics (RM for short), in the guise of the following principle:
\[
\text{\emph{for a countable set $A\subset \mathbb{R}$, there exists $y\in \mathbb{R}\setminus A$.}}
\]
An important conceptual observation is that the usual definition of countable set -based on injections or bijections to $\N$- does not seem suitable  for the RM-study of mainstream mathematics; we also propose a suitable (equivalent over strong systems) alternative definition of countable set, namely \emph{union over $\N$ of finite sets}; the latter is known from the literature and closer to how countable sets occur `in the wild'.  
We identify a considerable number of theorems that are equivalent to the centred theorem based on our alternative definition.  
Perhaps surprisingly, our equivalent theorems involve most basic properties of the {Riemann integral}, regulated or bounded variation functions, Blumberg's theorem, and Volterra's early work circa 1881. 
Our equivalences are also \emph{robust}, promoting the uncountability of $\R$ to the status of `big' system in RM. 
\end{abstract}

\setcounter{page}{0}
\tableofcontents
\thispagestyle{empty}
\newpage

\maketitle
\thispagestyle{empty}







\section{Introduction}
\subsection{Summary}\label{intro}
Like Hilbert (\cite{hilbertendlich}), we believe the infinite to be a central object of study in mathematics. 
That the infinite comes in `different sizes' is a relatively new insight, due to Cantor around 1874 (\cite{cantor1}), in the guise of the \emph{uncountability of $\R$}, also known simply as \emph{Cantor's theorem}. 
We have previously studied the uncountability of $\R$ in the guise of the following \emph{third-order}\footnote{By definition, the uncountability of $\R$ is formulated in terms of arbitrary (third-order) mappings from $\R$ to $\N$.  For this reason, we think it best to study this principle in a framework that directly includes such mappings, as opposed to representing them via second-order objects.} principles.  
\begin{itemize}
\item $\NIN$: there is no injection from $[0,1]$ to $\N$.
\item $\NBI$: there is no bijection from $[0,1]$ to $\N$
\end{itemize}
In particular, as shown in \cite{dagsamX, dagsamXII, dagsamXIII}, the principles $\NBI$ and $\NIN$ are \emph{hard to prove} in terms of conventional\footnote{We discuss the notion of `conventional' comprehension in Section \ref{lll} where we introduce $\Z_{2}^{\omega}$: a (conservative) higher-order extension of second-order arithmetic $\Z_{2}$ involving `comprehension functionals' $\SS_{k}^{2}$ that decide arbitrary $\Pi_{k}^{1}$-formulas. Since $\Z_{2}^{\omega}$ cannot prove $\NIN$ (see \cite{dagsamX}) \emph{and} both are essentially third-order in nature, our claim `$\NIN$ is hard to prove' seems justified.\label{kabal}} comprehension, while the objects claimed to exist are \emph{hard to compute} in terms of the other data, in the sense of Kleene's computability theory based on S1-S9 (\cite{longmann, kleeneS1S9}).  
As shown in \cite{samcie22}, this hardness remains if we restrict the mappings in $\NIN$ and $\NBI$ to well-known function classes, e.g.\ based on bounded variation, Borel, upper semi-continuity, and quasi-continuity. 
Morover, \emph{many} basic third-order theorems imply $\NIN$ or $\NBI$, and the same at the computational level (see \cite{samcount, dagsamXIII, dagsamXII, dagsamXI,dagsamX, samNEO2, samcie22,samwollic22}).
Finally, $\NIN$ and $\NBI$ seem to be the weakest \emph{natural} third-order principles that boast all the aforementioned properties.  

\smallskip

For all these reasons, the study of the uncountability of $\R$ in \emph{Reverse Mathematics} (RM for short; see Section \ref{prelim}) seems like a natural enterprise.  
However, \emph{try as we might} we have not managed to obtain elegant equivalences for $\NIN$ and $\NBI$, working in Kohlenbach's higher-order \emph{Reverse Mathematics} (see Section \ref{prelim}).  

\smallskip

As argued in detail in Section \ref{othername}, our main problem is that countable sets that occur `in the wild' do not have injections (let alone bijections) to $\N$ that can be defined in weak logical systems.  By contrast, the (equivalent over $\ZF$ and weaker systems)  definition of countable set as in Definitions \ref{hoogzalielevenint} and \ref{deaddint}  is much more suitable for the development of higher-order RM and is central to this paper.  
\begin{defi}\label{hoogzalielevenint}
A set $A\subset \R$ is \emph{height countable} if there is a \emph{height function} $H:\R\di \N$ for $A$, i.e.\ for all $n\in \N$, $A_{n}:= \{ x\in A: H(x)<n\}$ is finite.
\end{defi}
\begin{defi}[Finite set]\label{deaddint}
Any $X\subset \R$ is \emph{finite} if there is $N\in \N$ such that for any finite sequence $(x_{0}, \dots, x_{N})$ of distinct reals, there is $i\leq N$ such that $x_{i}\not \in X$.
\edefi
The notion of `height function' can be found in the literature in connection to countability (\cites{demol, vadsiger, royco,komig, hux}), while `height countable' essentially amounts to \emph{union over $\N$ of finite sets}.
By contrast, we believe Definition \ref{deaddint} has not been studied in the literature.
Our move away from injections/bijections towards height functions and finite sets constitutes a `shift of definition' which has ample historical precedent in RM and constructive mathematics, as discussed in Remark~\ref{rightdef}.  We note that Kleene's quantifier $(\exists^{2})$ from Section \ref{lll} is needed to make Definition~\ref{deaddint} well-behaved, as discussed in more detail in Remark \ref{LEM}.

\smallskip

In more detail, we shall establish a large number of equivalences for the following principle, which is based on Definition \ref{hoogzalielevenint} and expresses that $[0,1]$ is uncountable:
\begin{itemize}
\item $\NIN_{\alt}$: the unit interval is not height countable.
\end{itemize}
In particular, we show in Section \ref{main} that $\NIN_{\alt}$ is equivalent to the following natural principles, working in Kohlenbach's {higher-order Reverse Mathematics}, introduced in Section~\ref{prelim}.  
Recall that a \emph{regulated} function has left and right limits everywhere, as studied by Bourbaki for Riemann integration (see Section \ref{cdef}).
\begin{enumerate}
\renewcommand{\theenumi}{\roman{enumi}}
\item For regulated $f:[0,1]\di \R$, there is a point $x\in [0,1]$ where $f$ is continuous (or quasi-continuous, or lower semi-continuous, or Darboux).\label{fil1}
\item For regulated $f:[0,1]\di \R$, the set of continuity points is dense in $[0,1]$.\label{fil2}
\item For regulated $f:[0,1]\di [0,1]$ with Riemann integral $\int_{0}^{1}f(x)dx=0$, there is $x\in [0,1]$ with $f(x)=0$ (Bourbaki, \cite{boereng}*{p.\ 61, Cor.\ 1}).
\item (Volterra \cite{volaarde2}) For regulated $f,g:[0,1]\di \R$, there is $x\in [0,1]$ such that $f$ and $g$ are both continuous or both discontinuous at $x$. \label{volkert1}
\item (Volterra \cite{volaarde2}) For regulated $f:[0,1]\di \R$, there is either $q\in \Q\cap [0,1]$ where $f$ is discontinuous, or $x\in [0,1]\setminus \Q$ where $f$ is continuous.\label{volkert2}
\item For regulated $f:[0,1]\di \R$, there is $y\in (0,1)$ where $F(x):=\lambda x.\int_{0}^{x}f(t)dt$ is differentiable with derivative equal to $f(y)$.
\item For regulated $f:[0,1]\di \R$, there are $a, b\in  [0,1]$ such that $\{ x\in [0,1]:f(a)\leq f(x)\leq f(b)\}$ is infinite.
\item Blumberg's theorem (\cite{bloemeken}) restricted to regulated functions. \label{fil8}
\end{enumerate}
A full list of equivalences may be found in Section \ref{mainreg} while we obtain similar (but also very different) results for functions of \emph{bounded variation} in Section \ref{mainbv}.  
We introduce all required definitions in Section \ref{helim}.
Some of the above theorems, including items \eqref{volkert1} and \eqref{volkert2} in the list, stem from Volterra's early work (1881) in the spirit of -but predating- the Baire category theorem, as discussed in Section~\ref{vintro}.  

\smallskip
\noindent
Now, comparing items \eqref{fil1} and \eqref{fil2} suggests that our results are \emph{robust} as follows:
\begin{quote}
A system is \emph{robust} if it is equivalent to small perturbations of itself. (\cite{montahue}*{p.\ 432}; emphasis in original)
\end{quote}
Most of our results shall be seen to exhibit a similar (or stronger) level of robustness. 
In this light, we feel that the uncountability of $\R$ deserves the moniker `big' system in the way this notion is used in second-order RM, namely as boasting many equivalences from various different fields of mathematics. 

\smallskip

Next, items \eqref{fil1}-\eqref{fil8} above imply $\NIN$ and are therefore \emph{hard to prove} in the sense of Footnote \ref{kabal}.  By contrast, we show in Section \ref{plif} that adding the extra condition `Baire 1', makes these items provable from (essentially) arithmetical comprehension.
While regulated functions are of course Baire 1, say over $\ZF$, there is no contradiction here as the statement \emph{a regulated function on the unit interval is Baire 1} already implies $\NIN$ (see \cite{dagsamXIV}*{\S2.8}).
Other restrictions of items \eqref{fil1}-\eqref{fil8}, e.g.\ involving {semi-continuity} or {Baire~2}, are still equivalent to $\NIN_{\alt}$, as shown in Section \ref{restri}. 

\smallskip

Finally, this paper deals with the RM of the uncountability of $\R$ while stronger `completeness' properties of the reals, namely related to measure and category, are studied in \cite{samRMBCT}.  
In particular, the latter paper develops the higher-order RM of the \emph{Baire category theorem} and Tao's \emph{pigeon hole principle} for measure spaces (\cite{taoeps}).  
We do not currently know of a principle weaker than the uncountability of $\R$ that yields (interesting) RM-equivalences.  

\subsection{Volterra's early work and related results}\label{vintro}
We introduce Volterra's early work from \cite{volaarde2} as it pertains to this paper, as well as related results. 

\smallskip

First of all, the Riemann integral was groundbreaking for a number of reasons, including its ability to integrate functions with infinitely many points of discontinuity, as shown by Riemann himself (\cite{riehabi}). 
A natural question is then `how discontinuous' a Riemann integrable function can be.  In this context, Thomae introduced the function $T:\R\di\R$ around 1875 in \cite{thomeke}*{p.\ 14, \S20}):
\be\label{thomae}
T(x):=
\begin{cases} 
0 & \textup{if } x\in \R\setminus\Q\\
\frac{1}{q} & \textup{if $x=\frac{p}{q}$ and $p, q$ are co-prime} 
\end{cases}.
\ee
Thomae's function $T$ is integrable on any interval, but has a dense set of points of discontinuity, namely $\Q$, and a dense set of points of continuity, namely $\R\setminus \Q$. 

\smallskip

The perceptive student, upon seeing Thomae's function as in \eqref{thomae}, will ask for a function continuous at each rational point and discontinuous at each irrational one.
Such a function cannot exist, as is generally proved using the Baire category theorem.  
However, Volterra in \cite{volaarde2} already established this negative result about twenty years before the publication of the Baire category theorem.

\smallskip

Secondly, as to the content of Volterra's paper \cite{volaarde2}, we find the following theorem on the first page, where a function is \emph{pointwise discontinuous} if it has a dense set of continuity points.
\begin{thm}[Volterra, 1881]\label{VOL}
There do not exist pointwise discontinuous functions defined on an interval for which the continuity points of one are the discontinuity points of the other, and vice versa.
\end{thm}
Volterra then states two corollaries, of which the following is perhaps well-known in `popular mathematics' and constitutes the aforementioned negative result. 
\begin{cor}[Volterra, 1881]\label{VOLcor}
There is no $\R\di\R$ function that is continuous on $\Q$ and discontinuous on $\R\setminus\Q$. 
\end{cor}
Thirdly, we shall study Volterra's theorem and corollary restricted to regulated functions (see Section \ref{cdef}). 
The latter kind of functions are automatically `pointwise discontinuous' in the sense of Volterra.

\smallskip

Fourth,  Volterra's results from \cite{volaarde2} are generalised in \cite{volterraplus,gaud}.  
The following theorem is immediate from these generalisations. 
\begin{thm}\label{dorki}
For any countable dense set $D\subset [0,1]$ and $f:[0,1]\di \R$, either $f$ is discontinuous at some point in $D$ or continuous at some point in $[0,1]\setminus D$. 
\end{thm}
Perhaps surprisingly, this generalisation (restricted to bounded variation or regulated functions) is still equivalent to the uncountability of $\R$.  The same holds for the related \emph{Blumberg's theorem} with the same restrictions.
\begin{thm}[Blumberg's theorem, \cite{bloemeken}]
For any $f:\R\di \R$, there is a dense subset $D\subset \R$ such that the restriction of $f$ to $D$, usually denoted $f_{\upharpoonright D}$, is continuous.  
\end{thm}
\noindent
To be absolutely clear, the conclusion of Blumberg's theorem means that 
\[
(\forall x\in D, \eps>0)(\exists \delta>0)\underline{(\forall y\in D)}(|x-y|<\delta\di |f(x)-f(y)|<\eps)), 
\]
where the underlined quantifier marks the difference with `usual' continuity.

%

\subsection{Preliminaries and definitions}\label{helim}
We briefly introduce \emph{Reverse Mathematics} in Section \ref{prelim}.
We introduce some essential axioms (Section \ref{lll}) and definitions (Section~\ref{cdef}).  A full introduction may be found in e.g.\ \cite{dagsamX}*{\S2}.

\subsubsection{Reverse Mathematics}\label{prelim}
Reverse Mathematics (RM hereafter) is a program in the foundations of mathematics initiated around 1975 by Friedman (\cites{fried,fried2}) and developed extensively by Simpson (\cite{simpson2}).  
The aim of RM is to identify the minimal axioms needed to prove theorems of ordinary, i.e.\ non-set theoretical, mathematics. 

\smallskip

We refer to \cite{stillebron} for a basic introduction to RM and to \cite{simpson2, simpson1,damurm} for an overview of RM.  We expect basic familiarity with RM, in particular Kohlenbach's \emph{higher-order} RM (\cite{kohlenbach2}) essential to this paper, including the base theory $\RCAo$.   An extensive introduction can be found in e.g.\ \cites{dagsamIII, dagsamV, dagsamX, dagsamXI}.  
All undefined notions may be found in \cite{dagsamX, dagsamXI}, while we do point out here that we shall sometimes use common notations from the type theory.  For instance, the natural numbers are type $0$ objects, denoted $n^{0}$ or $n\in \N$.  
Similarly, elements of Baire space are type $1$ objects, denoted $f\in \N^{\N}$ or $f^{1}$.  Mappings from Baire space $\N^{\N}$ to $\N$ are denoted $Y:\N^{\N}\di \N$ or $Y^{2}$.

\subsubsection{Some comprehension functionals}\label{lll}
In second-order RM, the logical hardness of a theorem is measured via what fragment of the comprehension axiom is needed for a proof.  
For this reason, we introduce some axioms and functionals related to \emph{higher-order comprehension} in this section.
We are mostly dealing with \emph{conventional} comprehension here, i.e.\ only parameters over $\N$ and $\N^{\N}$ are allowed in formula classes like $\Pi_{k}^{1}$ and $\Sigma_{k}^{1}$.  

\smallskip
\noindent
First of all, the functional $\varphi$ in $(\exists^{2})$ is also \emph{Kleene's quantifier $\exists^{2}$} and is clearly discontinuous at $f=11\dots$ in Cantor space:
\be\label{muk}\tag{$\exists^{2}$}
(\exists \varphi^{2}\leq_{2}1)(\forall f^{1})\big[(\exists n^{0})(f(n)=0) \asa \varphi(f)=0    \big]. 
\ee
In fact, $(\exists^{2})$ is equivalent to the existence of $F:\R\di\R$ such that $F(x)=1$ if $x>_{\R}0$, and $0$ otherwise (see \cite{kohlenbach2}*{Prop.\ 3.12}).  
Related to $(\exists^{2})$, the functional $\mu^{2}$ in $(\mu^{2})$ is called \emph{Feferman's $\mu$} (see \cite{avi2}) and may be found -with the same symbol- in Hilbert-Bernays' Grundlagen (\cite{hillebilly2}*{Supplement IV}):
\begin{align}\label{mu}\tag{$\mu^{2}$}
(\exists \mu^{2})(\forall f^{1})\big[ (\exists n)(f(n)=0) \di [f(\mu(f))=0&\wedge (\forall i<\mu(f))(f(i)\ne 0) ]\\
& \wedge [ (\forall n)(f(n)\ne0)\di   \mu(f)=0]    \big]. \notag
\end{align}
We have $(\exists^{2})\asa (\mu^{2})$ over $\RCAo$ (see \cite{kohlenbach2}*{\S3}) and $\ACAo\equiv\RCAo+(\exists^{2})$ proves the same sentences as $\ACA_{0}$ by \cite{hunterphd}*{Theorem~2.5}. 

\smallskip

Secondly, the functional $\SS^{2}$ in $(\SS^{2})$ is called \emph{the Suslin functional} (\cite{kohlenbach2}):
\be\tag{$\SS^{2}$}
(\exists\SS^{2}\leq_{2}1)(\forall f^{1})\big[  (\exists g^{1})(\forall n^{0})(f(\overline{g}n)=0)\asa \SS(f)=0  \big].
\ee
The system $\FIVE^{\omega}\equiv \RCAo+(\SS^{2})$ proves the same $\Pi_{3}^{1}$-sentences as $\FIVE$ by \cite{yamayamaharehare}*{Theorem 2.2}.   
By definition, the Suslin functional $\SS^{2}$ can decide whether a $\Sigma_{1}^{1}$-formula as in the left-hand side of $(\SS^{2})$ is true or false.   We similarly define the functional $\SS_{k}^{2}$ which decides the truth or falsity of $\Sigma_{k}^{1}$-formulas from $\L_{2}$; we also define 
the system $\SIXK$ as $\RCAo+(\SS_{k}^{2})$, where  $(\SS_{k}^{2})$ expresses that $\SS_{k}^{2}$ exists.  
We note that the operators $\nu_{n}$ from \cite{boekskeopendoen}*{p.\ 129} are essentially $\SS_{n}^{2}$ strengthened to return a witness (if existant) to the $\Sigma_{n}^{1}$-formula at hand.  

\smallskip

\noindent
Thirdly, full second-order arithmetic $\Z_{2}$ is readily derived from $\cup_{k}\SIXK$, or from:
\be\tag{$\exists^{3}$}
(\exists E^{3}\leq_{3}1)(\forall Y^{2})\big[  (\exists f^{1})(Y(f)=0)\asa E(Y)=0  \big], 
\ee
and we therefore define $\Z_{2}^{\Omega}\equiv \RCAo+(\exists^{3})$ and $\Z_{2}^\omega\equiv \cup_{k}\SIXK$, which are conservative over $\Z_{2}$ by \cite{hunterphd}*{Cor.\ 2.6}. 
Despite this close connection, $\Z_{2}^{\omega}$ and $\Z_{2}^{\Omega}$ can behave quite differently, as discussed in e.g.\ \cite{dagsamIII}*{\S2.2}.   
The functional from $(\exists^{3})$ is also called `$\exists^{3}$', and we use the same convention for other functionals.

\subsubsection{Some basic definitions}\label{cdef}
We introduce some definitions needed in the below, mostly stemming from mainstream mathematics.
We note that subsets of $\R$ are given by their characteristic functions as in Definition \ref{char}, well-known from measure and probability theory.

\smallskip

Zeroth of all, we make use the usual definition of (open) set, where $B(x, r)$ is the open ball with radius $r>0$ centred at $x\in \R$.
\bdefi[Sets]\label{char}~
\begin{itemize}
\item A subset $A\subset \R$ is given by its characteristic function $F_{A}:\R\di \{0,1\}$, i.e.\ we write $x\in A$ for $ F_{A}(x)=1$, for any $x\in \R$.
\item A subset $O\subset \R$ is \emph{open} in case $x\in O$ implies that there is $k\in \N$ such that $B(x, \frac{1}{2^{k}})\subset O$.
\item A subset $C\subset \R$ is \emph{closed} if the complement $\R\setminus C$ is open. 
\end{itemize}
\edefi
\noindent
As discussed in Remark \ref{dichtbij}, the study of functions of bounded variation already gives rise to open sets that 
do not come with additional representation beyond Definition \ref{char}.

\smallskip

First of all, we shall study the following notions of weak continuity, all of which hark back to the days of Baire, Darboux, and Volterra (\cites{beren,beren2,darb, volaarde2}).
\bdefi\label{flung} For $f:[0,1]\di \R$, we have the following definitions:
\begin{itemize}
\item $f$ is \emph{upper semi-continuous} at $x_{0}\in [0,1]$ if $f(x_{0})\geq_{\R}\lim\sup_{x\di x_{0}} f(x)$,
\item $f$ is \emph{lower semi-continuous} at $x_{0}\in [0,1]$ if $f(x_{0})\leq_{\R}\lim\inf_{x\di x_{0}} f(x)$,
\item $f$ is \emph{quasi-continuous} at $x_{0}\in [0, 1]$ if for $ \epsilon > 0$ and an open neighbourhood $U$ of $x_{0}$, 
there is a non-empty open ${ G\subset U}$ with $(\forall x\in G) (|f(x_{0})-f(x)|<\eps)$.
\item $f:\R\di \R$ \emph{symmetrically continuous} at $x\in \R$ if
\[
(\forall \eps> 0)(\exists \delta>0)(\forall z\in \R )(|z|<\delta\di  |f(x+z)-f(x-z)|<\eps ).
\]

\item $f$ is \emph{Baire 1} if it is the pointwise limit of a sequence of continuous functions. 
\item $f$ is \emph{Baire $2$} if it is the pointwise limit of a sequence of Baire $1$ functions.
\item $f$ is \emph{Baire 1$^{*}$} if\footnote{The notion of Baire 1$^{*}$ goes back to \cite{ellis} and equivalent definitions may be found in \cite{kerkje}.  
In particular,  Baire 1$^{*}$ is equivalent to the Jayne-Rogers notion of \emph{piecewise continuity} from \cite{JR}.} there is a sequence of closed sets $(C_{n})_{n\in \N}$ such $[0,1]=\cup_{n\in \N}C_{n}$ and $f_{\upharpoonright C_{m}}$ is continuous for all $m\in \N$.
\end{itemize}
The first two items are often abbreviated as `usco' and `lsco'.
\edefi
\noindent
Secondly, we also need the notion of `intermediate value property', also called the `Darboux property' in light of Darboux's work in \cite{darb}.
\bdefi[Darboux property] Let $f:[0,1]\di \R$ be given. 
\begin{itemize}
\item A real $y\in \R$ is a left (resp.\ right) \emph{cluster value} of $f$ at $x\in [0,1]$ if there is $(x_{n})_{n\in \N}$ such that $y=\lim_{n\di \infty} f(x_{n})$ and $x=\lim_{n\di \infty}x_{n}$ and $(\forall n\in \N)(x_{n}\leq x)$ (resp.\ $(\forall n\in \N)(x_{n}\geq x)$).  
\item A point $x\in [0,1]$ is a \emph{Darboux point} of $f:[0,1]\di \R$ if for any $\delta>0$ and any left (resp.\ right) cluster value $y$ of $f$ at $x$ and $z\in \R$ strictly between $y$ and $f(x)$, there is $w\in (x-\delta, x)$ (resp.\ $w\in ( x, x+\delta)$) such that $f(w)=y$.   
\end{itemize}
\edefi
By definition, a point of continuity is also a Darboux point, but not vice versa.  

\smallskip
\noindent
Thirdly, we introduce the `usual' definitions of countable set (Def.\ \ref{eni} and \ref{standard}).  
\bdefi[Enumerable sets of reals]\label{eni}
A set $A\subset \R$ is \emph{enumerable} if there exists a sequence $(x_{n})_{n\in \N}$ such that $(\forall x\in \R)(x\in A\di (\exists n\in \N)(x=_{\R}x_{n}))$.  
\edefi
This definition reflects the RM-notion of `countable set' from \cite{simpson2}*{V.4.2}.  
We note that given $\mu^{2}$ from Section \ref{lll}, we may replace the final implication in Definition~\ref{eni} by an equivalence. 
\bdefi[Countable subset of $\R$]\label{standard}~
A set $A\subset \R$ is \emph{countable} if there exists $Y:\R\di \N$ such that $(\forall x, y\in A)(Y(x)=_{0}Y(y)\di x=_{\R}y)$.  
The mapping $Y:\R\di \N$ is called an \emph{injection} from $A$ to $\N$ or \emph{injective on $A$}. 
If $Y:\R\di \N$ is also \emph{surjective}, i.e.\ $(\forall n\in \N)(\exists x\in A)(Y(x)=n)$, we call $A$ \emph{strongly countable}.
\edefi
The first part of Definition \ref{standard} is from Kunen's set theory textbook (\cite{kunen}*{p.~63}) and the second part is taken from Hrbacek-Jech's set theory textbook \cite{hrbacekjech} (where the term `countable' is used instead of `strongly countable').  
For the rest of this paper, `strongly countable' and `countable' shall exclusively refer to Definition~\ref{standard}, \emph{except when explicitly stated otherwise}. 

\smallskip

Finally, the uncountability of $\R$ can be studied in numerous guises in higher-order RM.  
For instance, the following are from \cite{dagsamX, dagsamXI}, where it is also shown that many extremely basic theorems imply these principles, while $\Z_{2}^{\omega}$ cannot prove them. 
\begin{itemize}
\item For a countable set $A\subset [0,1]$, there is $y\in [0,1]\setminus A$. 
\item $\NIN$: there is no injection from $[0,1]$ to $\N$.
\item For a \textbf{strongly} countable set $A\subset [0,1]$, there is $y\in [0,1]\setminus A$. 
\item $\NBI$: there is no \textbf{bijection} from $[0,1]$ to $\N$.
\end{itemize}
The reader will verify that the first two and last two items are (trivially) equivalent.   
Besides these and similar variations in \cite{samcie22}, we have not been able to obtain elegant or natural equivalences involving the uncountability of $\R$ \emph{try as we might}.  
As discussed in Section \ref{othername}, this is because the above items are formulated using the `set theoretic' definition of countability as in Definition \ref{standard}.
In Section \ref{main}, we obtain many equivalences involving the uncountability of $\R$, based on the alternative (but equivalent over $\ZF$) notion of `height countable' introduced in Section \ref{intro}.  



\subsubsection{Some advanced definitions: bounded variation and around}\label{deffer}
We formulate the definitions of bounded variation and regulated functions, and some background. 

\smallskip

Firstly, the notion of \emph{bounded variation} (often abbreviated $BV$ below) was first explicitly\footnote{Lakatos in \cite{laktose}*{p.\ 148} claims that Jordan did not invent or introduce the notion of bounded variation in \cite{jordel}, but rather discovered it in Dirichlet's 1829 paper \cite{didi3}.} introduced by Jordan around 1881 (\cite{jordel}) yielding a generalisation of Dirichlet's convergence theorems for Fourier series.  
Indeed, Dirichlet's convergence results are restricted to functions that are continuous except at a finite number of points, while $BV$-functions can have infinitely many points of discontinuity, as already studied by Jordan, namely in \cite{jordel}*{p.\ 230}.
Nowadays, the \emph{total variation} of a function $f:[a, b]\di \R$ is defined as follows:
\be\label{tomb}\textstyle
V_{a}^{b}(f):=\sup_{a\leq x_{0}< \dots< x_{n}\leq b}\sum_{i=0}^{n} |f(x_{i})-f(x_{i+1})|.
\ee
If this quantity exists and is finite, one says that $f$ has bounded variation on $[a,b]$.
Now, the notion of bounded variation is defined in \cite{nieyo} \emph{without} mentioning the supremum in \eqref{tomb}; this approach can also be found in \cites{kreupel, briva, brima}.  
Hence, we shall distinguish between the two notions in Definition \ref{varvar}.  As it happens, Jordan seems to use item \eqref{donp} of Definition \ref{varvar} in \cite{jordel}*{p.\ 228-229}.
This definition suggests a two-fold variation for any result on functions of bounded variation, namely depending on whether the supremum \eqref{tomb} is given, or only an upper bound on the latter.  
\bdefi[Variations on variation]\label{varvar}
\begin{enumerate}  
\renewcommand{\theenumi}{\alph{enumi}}
\item The function $f:[a,b]\di \R$ \emph{has bounded variation} on $[a,b]$ if there is $k_{0}\in \N$ such that $k_{0}\geq \sum_{i=0}^{n} |f(x_{i})-f(x_{i+1})|$ 
for any partition $x_{0}=a <x_{1}< \dots< x_{n-1}<x_{n}=b  $.\label{donp}
\item The function $f:[a,b]\di \R$ \emph{has {a} variation} on $[a,b]$ if the supremum in \eqref{tomb} exists and is finite.\label{donp2}
\end{enumerate}
\edefi
Secondly, the fundamental theorem about $BV$-functions is formulated as follows. proved already by Jordan in \cite{jordel}.
\begin{thm}[Jordan decomposition theorem, \cite{jordel}*{p.\ 229}]\label{drd}
A $BV$-function $f : [0, 1] \di \R$ is the difference of  two non-decreasing functions $g, h:[0,1]\di \R$.
\end{thm}
Theorem \ref{drd} has been studied via second-order representations in \cites{groeneberg, kreupel, nieyo, verzengend}.
The same holds for constructive analysis by \cites{briva, varijo,brima, baathetniet}, involving different (but related) constructive enrichments.  
Now, $\ACA_{0}$ suffices to derive Theorem \ref{drd} for various kinds of second-order \emph{representations} of $BV$-functions in \cite{kreupel, nieyo}.  
By contrast, our results in \cite{dagsamXI} imply that the third-order version of Theorem \ref{drd} is hard to prove in terms of conventional comprehension.  

\smallskip

Thirdly, Jordan proves in \cite{jordel3}*{\S105} that $BV$-functions are exactly those for which the notion of `length of the graph of the function' makes sense.  In particular, $f\in BV$ if and only if the `length of the graph of $f$', defined as follows:
\be\label{puhe}\textstyle
L(f, [0,1]):=\sup_{0=t_{0}<t_{1}<\dots <t_{m}=1} \sum_{i=0}^{m-1} \sqrt{(t_{i}-t_{i+1})^{2}+(f(t_{i})-f(t_{i+1}))^{2}  }
\ee
exists and is finite by \cite{voordedorst}*{Thm.\ 3.28.(c)}.  In case the supremum in \eqref{puhe} exists (and is finite), $f$ is also called \emph{rectifiable}.  
Rectifiable curves predate $BV$-functions: in \cite{scheeffer}*{\S1-2}, it is claimed that \eqref{puhe} is essentially equivalent to Duhamel's 1866 approach from \cite{duhamel}*{Ch.\ VI}.  Around 1833, Dirksen, the PhD supervisor of Jacobi and Heine, already provides a definition of arc length that is (very) similar to \eqref{puhe} (see \cite{dirksen}*{\S2, p.\ 128}), but with some conceptual problems as discussed in \cite{coolitman}*{\S3}.

\smallskip

Fourth, a function is \emph{regulated} (called `regular' in \cite{voordedorst}) if for every $x_{0}$ in the domain, the `left' and `right' limit $f(x_{0}-)=\lim_{x\di x_{0}-}f(x)$ and $f(x_{0}+)=\lim_{x\di x_{0}+}f(x)$ exist.  
Scheeffer studies discontinuous regulated functions in \cite{scheeffer} (without using the term `regulated'), while Bourbaki develops Riemann integration based on regulated functions in \cite{boerbakies}.  
We note that $BV$-functions are regulated, while Weierstrass' `monster' function is a natural example of a regulated function not in $BV$.  

\smallskip

Finally, an interesting observation about regular functions is as follows.
\begin{rem}[Continuity and regulatedness]\label{atleast}\rm
First of all, as discussed in \cite{kohlenbach2}*{\S3}, the \emph{local} equivalence for functions on Baire space between sequential and `epsilon-delta' continuity cannot be proved in $\ZF$.  
By \cite{dagsamXI}*{Theorem 3.32}, this equivalence for \emph{regulated} functions is provable in $\ZF$ (and actually just $\ACAo$).  

\smallskip

Secondly, $\mu^{2}$ readily computes the left and right limits of regulated $f:[0,1]\di \R$.  In this way, the formula `$f$ is continuous at $x\in [0,1]$' is decidable using $\mu^{2}$, namely equivalent to the formula `$f(x+)=f(x)=f(x-)$'.  
The usual `epsilon-delta' definition of continuity involves quantifiers over $\R$, i.e.\ the previous equality is much simpler and more elementary. 
\end{rem}
By the previous remark, the basic notions needed for the study of regulated and $BV$-functions make sense in $\ACAo$.

\section{Countability by any other name}\label{othername}
We show that the `standard' set-theoretic definitions of countability -from Section \ref{cdef} and based on injections and bijections to $\N$-
are not suitable for the RM-study of regulated functions (see Section \ref{argsec}) and $BV$-functions (Section \ref{argsec2}).  
We also formulate an alternative -more suitable for RM- notion of countability (Definitions~\ref{hoogzalieleven} and~\ref{hoogzalieleven2}), which amounts to `unions over $\N$ of finite sets' and which can also be found in the mathematical literature.
This kind of `shift of definition' has historical precedent as follows. 
\begin{rem}\label{rightdef}\rm
First of all, the correct choice of definition for a given mathematical notion is crucial to the development of RM, as can be gleaned from the following quote from \cite{earlybs}*{p.\ 129}.  
\begin{quote}
Under the old definition [of real number from \cite{simpson3}], it would be consistent with $\RCA_{0}$ that there exists a sequence of real numbers $(x_{n})_{n\in \N}$ such that $(x_{n}+\pi)_{n\in \N}$ is not a sequence of real numbers. We thank Ian Richards for pointing out this defect of the old definition. Our new definition [of real number from \cite{earlybs}], given above, is adopted in order to remove this defect. All of the arguments and results of \cite{simpson3}
remain correct under the new definition.
\end{quote}
In short, the early definition of `real number' from \cite{simpson3} was not suitable for the development of RM, highlighting the importance of the `right' choice of definition.  

\smallskip

Secondly, we stress that RM is not unique in this regard: the early definition of `continuous function' in Bishop's constructive analysis (\cite{bish1}) was also deemed problematic and changed to a new definition to be found in \cite{bridge2}; the (substantial) problems with both definitions are discussed in some detail in \cite{waaldijk, vandebrug}, including elementary properties such as the concatenation of continuous functions and the continuity of $\frac{1}{x}$ for $x>0$.
\end{rem}
In short, the development of mathematics in logical systems with `restricted' resources, like RM or constructive mathematics, seems to hinge on the 
`right' choice of definition.   In this section, we argue that the `right' definition of countability for higher-order RM is given by height functions as in Section \ref{intro}.
To be absolutely clear, the background theory for this section is $\ZFC$, i.e.\ a statement like `$A\subset\R$ is countable' means that the latter is provable in the former; most arguments (should) go through in $\Z_{2}^{\Omega}$.  

%
%

\subsection{Regulated functions and countability}\label{argsec}
As suggested in Section \ref{intro}, the set-theoretic definition of countable set is not suitable for the RM-study of regulated functions.  
We first provide some motivation for this claim in Remark~\ref{diunk}.  Inspired by the latter, we can then present our alternative notion in Definition~\ref{hoogzalieleven}, which amounts to `unions over $\N$ of finite sets'.
\begin{rem}[Countable sets by any other name]\label{diunk}\rm
First of all, we have previously investigated the RM of regulated functions in \cites{dagsamXI}.  
As part of this study, the following sets -definable via $\exists^{2}$- present themselves, where $f:[0,1]\di \R$ is regulated:
\begin{eqnarray}\label{lagel2}\textstyle
&A:= \big\{x\in (0,1):  f(x+)\ne f(x) \vee f(x-)\ne f(x)\big\}, \notag \\
&A_{n}:=\big\{x\in (0,1): |f(x+)- f(x)|>\frac1{2^{n}} \vee |f(x-)- f(x)|>\frac1{2^{n}}\big\}.
\end{eqnarray}
Clearly, $A=\cup_{n\in\N}A_{n}$ collects all points in $(0,1)$ where $f$ is discontinuous; this set is central to many proofs involving regulated functions (see e.g.\ \cite{voordedorst}*{Thm.\ 0.36}).  
Now, that $A_{n}$ is finite follows by a standard\footnote{If $A_{n}$ were infinite, the Bolzano-Weierstrass theorem implies the existence of a limit point $y\in [0,1]$ for $A_{n}$.  One readily shows that $f(y+)$ or $f(y-)$ does not exist, a contradiction as $f$ is assumed to be regulated.\label{fkluk}} compactness argument.  
However, while $A$ is then countable, we are unable to construct an injection from $A$ to $\N$ (let alone a bijection), even working in $\Z_{2}^{\omega}$ (see Remark \ref{dichtbij} for details).

\smallskip

In short, one readily finds countable sets `in the wild', namely pertaining to regulated functions, for which the associated injections to $\N$ cannot be constructed in reasonably weak logical systems.  

\smallskip

Secondly, in light of \eqref{lagel2}, regulated functions give rise to countable sets given \emph{only} as the union over $\N$ of finite sets (i.e.\ without information about an injection to $\N$).
To see that the `reverse' is also true, consider the following function:
\be\label{mopi}
h(x):=
\begin{cases}
0 & x\not \in \cup_{m\in \N}X_{m} \\
\frac{1}{2^{n+1}} &  x\in X_{n} \textup{ and $n$ is the least such number}
\end{cases},
\ee
where $(X_{n})_{n\in \N}$ is a sequence of finite sets in $[0,1]$.   
One readily shows that $h$ is regulated using $\exists^{2}$.
For general closed sets, \eqref{mopi} is crucial to the study of Baire 1 functions (see \cite{myerson}*{p.\ 238}). 
Hence, regulated functions yield countable sets given (only) as unions over $\N$ of finite sets, namely via $A=\cup_{n\in \N}A_{n}$ from \eqref{lagel2}, \emph{and vice versa}, namely via  $h:[0,1]\di \R$ as in \eqref{mopi}.

\smallskip

In summary, we observe that the usual definition of countable set (involving injections/bijections to $\N$) is not suitable for the RM-study of regulated functions.   
Luckily, \eqref{lagel2} and \eqref{mopi} suggest an alternative approach via the fundamental connection between regulated functions on one hand, and countable sets given as 
\begin{center}
\emph{the union over $\N$ of finite sets}
\end{center}  
on the other hand.
In conclusion, the RM-study of regulated functions should be based on the centred notion of countability and \textbf{not} injections/bijections to $\N$.
%
%
\end{rem}
Motivated by Remark \ref{diunk}, we introduce our alternative definition of countability, which is exactly the same as Definition \ref{hoogzalielevenint} in Section \ref{intro}.  
\begin{defi}\label{hoogzalieleven}
A set $A\subset \R$ is \emph{height countable} if there is a \emph{height} function $H:\R\di \N$ for $A$, i.e.\ for all $n\in \N$, $A_{n}:= \{ x\in A: H(x)<n\}$ is finite.  
\end{defi}
The previous notion of `height' is mentioned in the context of countability in e.g.\ \cite{demol,vadsiger, royco,komig,hux}.   
Definition~\ref{hoogzalieleven} amounts to `union over $\N$ of finite sets', as is readily shown in $\ACAo$.

\smallskip

Finally, the observations from Remark \ref{diunk} regarding countable sets also apply \emph{mutatis mutandis} to finite sets.  
Indeed, finite as each $A_{n}$ from \eqref{lagel2} may be, we are unable to construct an injection to a finite subset of $\N$, even assuming $\Z_{2}^{\omega}$ (see Remark \ref{dichtbij} for details).  
By contrast, the definition of finite set from Section \ref{intro} is more suitable: 
one readily\footnote{The proof of Theorem \ref{flonk} shows that $A_{n}$ is finite, working in $\ACAo+\QFAC^{0,1}$.} shows that $A_{n}$ from \eqref{lagel2} is finite as in Definition \ref{deadd}, which is exactly the same as Definition \ref{deaddint} in Section \ref{intro}.
\begin{defi}[Finite set]\label{deadd}
Any $X\subset \R$ is \emph{finite} if there is $N\in \N$ such that for any finite sequence $(x_{0}, \dots, x_{N})$ of distinct reals, there is $i\leq N$ such that $x_{i}\not \in X$.
\edefi
The number $N$ from Definition \ref{deadd} is call a \emph{size bound} for the finite set $X\subset \R$.
Analogous to countable sets, the RM-study of regulated functions should be based on Definition~\ref{deadd} and \textbf{not} on the set-theoretic definition based on injections/bijections to finite subsets of $\N$ or similar constructs.

\subsection{Bounded variation functions and countability}\label{argsec2}
We discuss the observations from Section \ref{argsec} for the particular case of functions of bounded variation (which are regulated by Theorem \ref{flima}).  
In particular, while the same observations apply, they have to be refined to yield elegant equivalences.  
\begin{rem}[Countable by another name]\label{diunk2}\rm
First of all, we consider \eqref{lagel2}, but formulated for a $BV$-function $g:[0,1]\di \R$, as follows:
\begin{eqnarray}\label{lagel3}\textstyle
&B:= \big\{x\in (0,1):  g(x+)\ne g(x) \vee g(x-)\ne g(x)\big\}, \notag \\
&B_{n}:=\big\{x\in (0,1): |g(x+)- g(x)|>\frac1{2^{n}} \vee |g(x-)- g(x)|>\frac1{2^{n}}\big\}.
\end{eqnarray}
Similar to $A=\cup_{n\in }A_{n}$ as in \eqref{lagel3}, $B=\cup_{n\in\N}B_{n}$ collects all points in $(0,1)$ where $g$ is discontinuous and this set is central to many proofs involving $BV$-functions (see \cite{voordedorst}).  
Similar to $A$ from \eqref{lagel3}, $B$ is countable but we are unable to construct an injection from $B$ to $\N$ (let alone a bijection), even assuming $\Z_{2}^{\omega}$ (see Remark \ref{dichtbij}).

\smallskip

Secondly, there is a crucial difference between \eqref{lagel2} and \eqref{lagel3}: we know that the set $B_{n}$ is finite \textbf{and} has at most $2^{n}V_{0}^{1}(f)$ elements; indeed, each element 
of $B_{n}$ contributes at least $1/2^{n}$ to the total variation $V_{0}^{1}(f)$ as in \eqref{tomb}.  By contrast, we have no extra information about the size of $A_{n}$ from \eqref{lagel2}.   
However, this extra information is crucial if we wish to deal with $BV$-functions (only).  
Indeed, the function $h$ from \eqref{mopi} is not in $BV$, e.g.\ in the trivial case where each $X_{n}$ has at least $2^{n+1}$ elements.    
By contrast, consider the following nicer function
\be\label{mopi2}
k(x):=
\begin{cases}
0 & x\not \in \cup_{m\in \N}Y_{m} \\
\frac{1}{2^{n+1}} \frac{1}{g(n)+1}&  x\in Y_{n} \textup{ and $n$ is the least such number}
\end{cases},
\ee
where $g\in \N^{\N}$ is a \emph{width function}\footnote{The function $g\in \N^{\N}$ is a \emph{width function} for the sequence of sets $(Y_{n})_{n\in \N}$ in $\R$ in case $Y_{n}$ has at most $ g(n)$ elements, for all $n\in \N$.} for $(Y_{n})_{n\in \N}$.
One readily verifies that $k:[0,1]\di \R$ is in $BV$ with total variation bounded by $1$.
 Hence, $BV$-functions yield countable sets given (only) as in the following description:
 \begin{center}
 \emph{unions over $\N$ of finite sets with a width function,}
 \end{center}
 namely via $B=\cup_{n\in \N}B_{n}$ from \eqref{lagel3}, \emph{and vice versa}, namely via  $k:[0,1]\di \R$ as in \eqref{mopi2}.  The generalisations of bounded variation from Remark \ref{essenti} have a similar property, as evidenced by the final part of the proof of Theorem \ref{flunk2}.  

%
%
\end{rem}
Motivated by Remark \ref{diunk2}, we introduce our alternative (equivalent over $\ZF$) definition of countability for the RM-study of $BV$-functions.  
\begin{defi}\label{hoogzalieleven2}
A set $B\subset \R$ is \emph{height-width countable} if there is a height function $H:\R\di \N$ and width function $g:\N\di \N$, i.e.\ for all $n\in \N$, the set $B_{n}:= \{ x\in B: H(x)<n\}$ is finite with size bound $g(n)$.  
\end{defi}
Finally, the following technical remark makes the claims in Remarks \ref{diunk}~and~\ref{diunk2} more precise in terms of logical systems.   
\begin{rem}\label{dichtbij}\rm
As discussed above, the sets $A_{n}$ from \eqref{lagel2} and $B_{n}$ from \eqref{lagel3} are finite, while the unions $A=\cup_{n\in \N}A_{n}$ and $B=\cup_{n\in \N}B_{n}$ are countable. 
Hence, working in $\ZF$ (or even $\Z_{2}^{\Omega}$ from Section \ref{lll}), the following objects can be constructed:
\begin{itemize}
\item for $n\in \N$, an injection $Y_{n}$ from $A_{n}$ to some $\{0, 1, \dots, k\}$ with $k\in \N$,
\item for $m\in \N$, an RM-code $C_{m}$ (see \cite{simpson2}*{II.5.6}) for the closed sets $A_{m}$ or $B_{m}$.
\end{itemize}
However, it is shown in \cite{samwollic22post, samcsl23} that neither $Y_{n}$ nor $C_{n}$ are computable (in the sense of Kleene S1-S9; see \cite{longmann}) in terms of any $\SS_{m}^{2}$ and the other data.
As a result, even $\Z_{2}^{\omega}$ cannot prove the general existence of $Y_{n}$ and $C_{n}$ as in the previous items.
By contrast, the system $\ACAo+\QFAC^{0,1}$ (and even fragments) suffice to show that Definitions \ref{hoogzalieleven}, \ref{hoogzalieleven2}, and \ref{deadd} apply to $A,A_{n}$ from \eqref{lagel2} and $B,B_{n}$ from \eqref{lagel3}.
\end{rem}
In conclusion, we have introduced `new' -but equivalent over $\ZF$ and the weaker $\Z_{2}^{\Omega}$- definitions of finite and countable set with the following properties.
\begin{itemize}
\item Our `new' definitions capture the notion of finite and countable set as it occurs `in the wild', namely in the study of $BV$ or regulated functions.  
This holds over relatively weak systems by Remark \ref{dichtbij}.
\item One finds our `new' definitions, in particular the notion of `height', in the literature (see \cites{vadsiger, royco, demol,komig,hux}).
\item These `new' definitions shall be seen to yield many equivalences in the RM of the uncountability of $\R$ (Sections \ref{mainreg} and \ref{mainbv}).
\end{itemize}
We believe that the previous items justify our adoption of our `new' definitions of finite and countable set.  Moreover, Remark \ref{rightdef} creates some historical precedent based on second-order RM and constructive mathematics. 

\section{Main results: regulated and $BV$-functions}\label{main}
\subsection{Introduction}
In this section, we establish the equivalences sketched in Section \ref{intro} pertaining to the uncountability of $\R$ and properties of regulated functions (Section \ref{mainreg}) and $BV$-functions (Section \ref{mainbv}).
In Section~\ref{basics}, we establish some basic properties of $BV$ and regulated functions in weak systems.  

\smallskip

As noted in Section \ref{intro}, we shall show that the uncountability of $\R$ as in $\NIN_{\alt}$ is \emph{robust}, i.e.\ equivalent to small perturbations of itself (\cite{montahue}*{p.\ 432}).
Striking examples of this claimed robustness may be found in Theorem \ref{duck}, where the perturbations are given by considering either one point of continuity, a dense set of such points, or various uncountability criteria for $C_{f}$, the set of continuity points.    

\smallskip

Finally,  the content of Section \ref{plif} is explained in the following remark.
\begin{rem}[When more is less]\label{LESS}\rm
As noted in Section \ref{intro}, $\NIN_{\alt}$ is equivalent to well-known theorems from analysis restricted to regulated functions, with similar results for $BV$-functions. 
These (restricted) theorems thus imply $\NIN$ and are not provable in $\Z_{2}^{\omega}+\QFAC^{0,1}$ as a result (see \cite{dagsamX}).
We show in Section \ref{plif} that adding the extra condition `Baire 1' to these theorems makes them provable from (essentially) arithmetical comprehension.
While regulated and $BV$-functions \emph{are} Baire 1, say over $\ZF$ or $\Z_{2}^{\Omega}$, there is no contradiction here as the statement 
\begin{center}
\emph{a $BV$-function on the unit interval is Baire 1} 
\end{center}
already implies $\NIN$ by \cite{dagsamXIV}*{Theorem 2.34}.
We stress that `Baire 1' is special in this regard: other restrictions, e.g.\ involving {semi-continuity} or {Baire~2}, yield theorems that are still equivalent to $\NIN_{\alt}$, as shown in Section \ref{restri}. 
We have no explanation for this phenomenon.   
\end{rem}

\subsection{Preliminary results}\label{basics}
We collect some preliminary results pertaining to regulated and $BV$-functions, and $\NIN_{\alt}$ from Section \ref{intro}.

\smallskip

First of all, to allow for a smooth treatment of finite sets, we shall adopt the following principle that collects the most basic properties of finite sets. 
\vspace{1mm}
\begin{princ}[$\FIN$]~
\begin{itemize}
\item \emph{Finite union theorem}: for a sequence of finite sets $(X_{n})_{n\in \N}$ and any $k\in \N$, $\cup_{n\leq k}X_{n}$ is finite. 
\item For any finite $X\subset \R$, there is a finite sequence of reals $(x_{0}, \dots, x_{k})$ that includes all elements of $X$. 
\item \emph{Finite Axiom of Choice}: for $Y^{2}, k^{0}$ with $(\forall n\leq k)(\exists f\in 2^{\N})(Y(f, n)=0)$, there is a finite sequence $(f_{0}, \dots, f_{k})$ in $2^{\N}$ with $(\forall n\leq k)(Y(f_{n}, n)=0)$.
\end{itemize}
\end{princ}
One can readily derive $\FIN$ from a sufficiently general fragment of the induction axiom; the RM of the latter is well-known (see e.g.\ \cite{simpson2}*{X.4.4}) and the RM of (fragments of) $\FIN$ is therefore 
a matter of future research.  We note that in \cite{dagsamXI}, we could derive (fragments of) $\FIN$ from the principles under study, like the fact that (height) countable sets of reals can be enumerated.  
Hence, we could mostly avoid the use of fragments of $\FIN$ in the base theory in \cite{dagsamXI}, which does not seem possible for this paper. 

\smallskip

Secondly, we need some some basic properties of $BV$ and regulated functions, all of which have been established in \cite{dagsamXI} already. 
%
\begin{thm}[$\ACAo$]\label{flima}~
\begin{itemize}
\item Assuming $\FIN$, any $BV$-function $f:[0,1]\di \R$ is regulated.
\item Any monotone function $f:[0,1]\di \R$ has bounded variation.   
\item For any monotone function $f:[0,1]\di \R$, there is a sequence $(x_{n})_{n\in \N}$ that enumerates all $x\in [0,1]$ such that $f$ is discontinuous at $x$.  
\item For regulated $f:[0,1]\di \R$ and $x\in [0,1]$, $f$ is sequentially continuous at $x$ if and only if $f$ is epsilon-delta continuous at $x$.
\item For finite $X\subset [0,1]$, the function $\mathbb{1}_{X}$ has bounded variation. 
\end{itemize}
\end{thm}
\begin{proof}
Proofs may be found in \cite{dagsamXI}*{\S3.3}.
\end{proof}
The fourth item of Theorem \ref{flima} is particularly interesting as the local equivalence between sequential and epsilon-delta continuity for \emph{general} $\R\di \R$ functions
is not provable in $\ZF$, while $\RCAo+\QFAC^{0,1}$ suffices, as discussed in Remark \ref{atleast}.

\smallskip

Thirdly, we discuss some `obvious' equivalences for $\NIN$ and $\NBI$.  
\begin{rem}\label{kimu}\rm
Now, $\NIN$ and $\NBI$ are formulated for mappings from $[0,1]$ to $\N$, but we can equivalently replace the unit interval by e.g.\ $\R$, $2^{\N}$, and $\N^{\N}$, as shown in \cite{samcie22}*{\S2.1}.  
An important observation in this context, and readily formalised in $\ACAo$, is that the (rescaled) tangent function provides a bijection from any open interval to $\R$; the inverse of tangent, called \emph{arctangent}, yields a bijection in the other direction (also with rescaling).
Moreover, using these bijections, one readily shows that $\NIN_{\alt}$ is equivalent to the following:
\begin{itemize}
\item there is no height function from $\R$ to $\N$.  
\end{itemize}
Similarly, if we can show that there is no height function from {some} fixed open interval to $\N$, then $\NIN_{\alt}$ follows.  We will tacitly make use of this fact in the proof of Theorems \ref{flonk} and \ref{duck}.
\end{rem}
Fourth, while we choose to use (at least) the system $\ACAo$ as our base theory, one can replace the latter by $\RCAo$ using the following trick.
\begin{rem}[Excluded middle trick]\label{LEM}\rm
The law of excluded middle as in $(\exists^{2})\vee \neg(\exists^{2})$ is quite useful as follows:  suppose we are proving $T\di \NIN_{\alt}$ over $\RCAo$.  
Now, in case $\neg(\exists^{2})$, all functions on $\R$ are continuous by \cite{kohlenbach2}*{Prop.\ 3.12} and $\NIN_{\alt}$ then trivially\footnote{In case $H:\R\di \N$ is continuous on $\R$, the set $A_{n}:= \{ x\in A: H(x)<n\}$ for $A=[0,1]$ in Definition \ref{hoogzalieleven} cannot be finite for any $n\in \N$ for which it is non-empty.}
holds.  Hence, what remains is to establish $T\di \NIN_{\alt}$ 
\emph{in case we have} $(\exists^{2})$.  However, the latter axiom e.g.\ implies $\ACA_{0}$ and can uniformly convert reals to their binary representations.  
In this way, finding a proof in $\RCAo+(\exists^{2})$ is `much easier' than finding a proof in $\RCAo$.
In a nutshell, we may \emph{without loss of generality} assume $(\exists^{2})$ when proving theorems that are trivial (or readily proved) when all functions (on $\R$ or $\N^{\N }$) are continuous, like $\NIN_{\alt}$.   
Moreover, we can replace $2^{\N}$ by $[0,1]$ at will, which is convenient sometimes.   
\end{rem}
While the previous trick is useful, it should be used sparingly: the axiom $(\exists^{2})$ is required to guarantee that basic 
sets like the unit interval are sets in our sense (Definition \ref{char}) or that finite sets (Definition \ref{deaddint}) are well-behaved.  For this reason, we only mention Remark \ref{LEM} in passing and
shall generally work over $\ACAo$.

\subsection{Regulated functions and the uncountability of $\R$}\label{mainreg}
We establish the equivalences sketched in Section \ref{intro} pertaining to the uncountability of $\R$ and properties of regulated functions. 
\subsubsection{Volterra's early work}
In this section, we connect the uncountability of $\R$ to Volterra's early results from Section \ref{vintro}.  
In particular, we establish the following theorem where the final two items exhibit some nice robustness properties of $\NIN_{\alt}$ and Volterra's results, as promised in Section \ref{intro}.
\begin{thm}[$\ACAo+\QFAC^{0,1}+\FIN$]\label{flonk}
The following are equivalent.
\begin{enumerate}
\renewcommand{\theenumi}{\alph{enumi}}
\item The uncountability of $\R$ as in $\NIN_{\alt}$.
\item \emph{Volterra's theorem for regulated functions}: there do not exist two regulated functions defined on the unit interval for which the continuity points of one are the discontinuity points of the other, and vice versa.\label{volare1}
\item \emph{Volterra's corollary for regulated functions}: there is no regulated function that is continuous on $\Q\cap[0,1]$ and discontinuous on $[0,1]\setminus\Q$.\label{volare2}
\item Generalised Volterra's corollary \(Theorem \ref{dorki}\) for regulated functions and height countable $D$ \(or: countable $D$, or: strongly countable $D$\). \label{volare3}
\item For a sequence $(X_{n})_{n\in \N}$ of finite sets in $[0,1]$, the set $[0,1]\setminus \cup_{n\in \N}X_{n}$ is dense \(or: not height countable, or: not countable, or: not strongly countable\).\label{lopi}
\end{enumerate}
\end{thm}
\begin{proof}
First of all, Volterra's theorem implies Volterra's corollary (both restricted to regulated functions), as Thomae's function $T$ from \eqref{thomae} is readily defined using $\exists^{2}$, while the latter also shows that $T$ is regulated and continuous exactly on $\R\setminus \Q$. 

\smallskip

Secondly, we now derive $\NIN_{\alt}$ from Volterra's corollary as in item \eqref{volare2}.  
To this end, let $(X_{n})_{n\in \N}$ be a sequence of finite sets such that $[0,1]=\cup_{n\in \N}X_{n}$.   Now use $\mu^{2}$ to define the following function.  
\be\label{modi2}
g(x):=
\begin{cases}
0 & x\in \Q\\
\frac{1}{2^{n+1}} & x \in \R\setminus \Q \wedge x\in X_{n} \textup{ and $n$ is the least such number}
\end{cases}.
\ee
We have $0=g(0+)=g(0-)=g(x+)=g(x-)$ for any $x\in (0,1)$, i.e.\ $g$ is regulated.  
To establish this fact in our base theory, note that $\cup_{k\leq n}X_{k}$ is finite for any $n\in\N$ and can be enumerated, both thanks to $\FIN$. 
As a result, $g$ is continuous at any $x\in \Q\cap [0,1]$ and discontinuous at any $y\in [0,1]\setminus \Q$.  This contradicts Volterra's corollary (for regulated functions), and $\NIN_{\alt}$ follows. 

\smallskip

Thirdly, we derive Volterra's corollary (for regulated functions) from $\NIN_{\alt}$, by contraposition. To this end, let $f$ be regulated, continuous on $[0,1]\cap \Q$, and discontinuous on $[0,1]\setminus \Q$. 
Now consider the following set
\be\label{sameold}\textstyle
X_{n}:=\big\{x\in (0,1): |f(x+)- f(x)|>\frac1{2^{n}} \vee |f(x-)- f(x)|>\frac1{2^{n}}\big\}, 
\ee
where we note that e.g.\ the right limit $f(x+)$ for $x\in (0,1)$ equals $\lim_{k\di \infty}f(x+\frac{1}{2^{k}})$; the latter limit is arithmetical and hence $\mu^{2}$ readily obtains it. 
Hence, the set $X_{n}$ from \eqref{sameold} can be defined in $\ACAo$. 
To show that $X_{n}$ is finite, suppose not and apply $\QFAC^{0,1}$ to find a sequence of reals in $X_{n}$.  By the Bolzano-Weierstrass theorem from \cite{simpson2}*{III.2}, this sequence 
has a convergent sub-sequence, say with limit $c\in [0,1]$; then either $f(c-)$ or $f(c+)$ does not exist (using the usual epsilon-delta definition), a contradiction.
Hence, $X_{n}$ is finite and by the assumptions on $f$, we have $D_{f}=\cup_{n\in \N}X_{n}= [0,1]\setminus\Q$.  Then $[0,1]=D_{f}\cup \Q=\big(\cup_{n\in \N}X_{n})\cup \Q$ shows that the unit interval is a union over $\N$ of finite sets, i.e.\ $\neg\NIN_{\alt}$ follows.  One derives item \eqref{volare1} from $\NIN_{\alt}$ in the same way; indeed: $[0,1]=D_{f}\cup D_{g}$ in case item \eqref{volare1} is false for regulated $f, g:[0,1]\di \R$, showing that the unit interval is the union over $\N$ of finite sets, yielding $\neg\NIN_{\alt}$. 

\smallskip

Fourth, we only need to show that $\NIN_{\alt}$ implies item \eqref{volare3}, as $\Q$ is trivially (height) countable and dense.
Hence, let $f$ be regulated, continuous on $[0,1]\cap D$, and discontinuous on $[0,1]\setminus D$, where $D$ is height countable and dense. 
In particular, assume $D=\cup_{n\in \N}D_{n}$ where $D_{n}$ is finite for $n\in \N$.  Now consider $X_{n}$ as in \eqref{sameold} from the above and note that $[0,1]\setminus D=\cup_{n\in \N}X_{n}$.
Hence, $[0,1]=\cup_{n\in \N}Y_{n}$ where $Y_{n}=X_{n}\cup D_{n}$ is finite (as the components are), i.e.\ $\neg\NIN_{\alt}$ follows. 

\smallskip

Finally, we only need to show that $\NIN_{\alt}$ implies the final item \eqref{lopi}.   For the terms in brackets in the latter, this is trivial as (strongly) countable sets are height countable.
For the density claim, let $(X_{n})_{n\in \N}$ be a sequence of finite sets and suppose $x_{0}\in [0,1]$ and $N_{0}\in \N$ are such that $B(x_{0}, \frac{1}{2^{N_{0}}})\cap \big([0,1]\setminus \cup_{n\in \N}X_{n} \big)$ is empty.
Hence,  $B(x_{0}, \frac{1}{2^{N_{0}}})\subset  \cup_{n\in \N}X_{n}$ and define the finite sets $Y_{n}:= X_{n}\cap B(x_{0}, \frac{1}{2^{N_{0}}})$ using $\exists^{2}$.
This implies $B(x_{0}, \frac{1}{2^{N_{0}}})=  \cup_{n\in \N}Y_{n}$, which contradicts $\NIN_{\alt}$, modulo the rescaling discussed in Remark \ref{kimu}.
\end{proof}
The final item of the theorem essentially expresses the Baire category theorem restricted to the complement of finite sets, which are automatically open and dense.

\subsubsection{Continuity and Riemann integration}
We connect the uncountability of $\R$ to properties of regulated functions like continuity and Riemann integration.

\smallskip

First of all, we shall need the set of (dis)continuity points of regulated $f:[0,1]\di \R$, definable via $\exists^{2}$ as follows:
\[
C_{f}:=\{x\in (0,1): f(x)=f(x+)=f(x-)\} \textup{ and } D_{f}=[0,1]\setminus C_{f}.
\] 
These sets occupy a central spot in the study of regulated functions.  
We have the following theorem, where most items exhibit some kind of robustness.
\begin{thm}[$\ACAo+\QFAC^{0,1}+\FIN$]\label{duck}
The following are equivalent.  
\begin{enumerate}
\renewcommand{\theenumi}{\roman{enumi}}
\item The uncountability of $\R$ as in $\NIN_{\alt}$, \label{pon1}
\item For any regulated $f:[0,1]\di \R$, there is $x\in [0,1]$ where $f$ is continuous \(or: quasi-continuous, or: lower semi-continuous\). \label{pon2}
\item Any regulated $f:[0,1]\di \R$ is pointwise discontinuous, i.e.\ the set $C_{f}$ is dense in the unit interval. \label{pon3}
\item For regulated $f:[0,1]\di \R$, the set $C_{f}$ is not height countable \(or: not countable, or: not strongly countable, or: not enumerable\). \label{pon35}
\item For regulated $f:[0,1]\di [0,1]$ such that the Riemann integral $\int_{0}^{1}f(x)dx$ exists and is $0$, there is $x\in [0,1]$ with $f(x)=0$. \(Bourbaki, \cite{boereng}*{p.\ 61}\).\label{pon8}
\item For regulated $f:[0,1]\di [0,1]$ such that the Riemann integral $\int_{0}^{1}f(x)dx$ exists and equals $0$, the set $\{x \in [0,1]: f(x)=0\}$ is dense. \(\cite{boereng}*{p.\ 61}\). \label{pon9}
\item Blumberg's theorem \(\cite{bloemeken}\) restricted to regulated functions on $[0,1]$.\label{pon13}
\item Measure theoretic Blumberg's theorem \(\cite{bruinebloem}\): for regulated $f:[0,1]\di \R$, there is a dense and uncountable \(or: not strongly countable, or: not height countable\) subset $D\subset [0,1]$ such that $f_{\upharpoonright D}$ is pointwise discontinuous. \label{pon14}
\item For regulated $f:[0,1]\di (0, 1]$, there exist $N\in \N, x\in [0,1]$ such that $(\forall y\in B(x, \frac{1}{2^{N}}))(f(y)\geq \frac{1}{2^{N}})$. \label{pon15}
\item For regulated $f:[0,1]\di (0, 1]$, there exist a dense set $D$ such that $f_{\upharpoonright D}$ is locally bounded away from zero\footnote{In symbols: $(\forall x\in D)(\exists N\in \N)(\forall y\in \underline{B(x, \frac{1}{2^{N}})\cap D})(f(y)\geq \frac{1}{2^{N}})$, where we stress the underlined part as it implements the claimed restriction to $D$.} \(on $D$\).\label{pon16}
\item \textsf{\textup{(FTC)}} For regulated $f:[0,1]\di \R$ such that $F(x):=\lambda x.\int_{0}^{x}f(t)dt$ exists, there is $x_{0}\in (0,1)$ where $F(x)$ is differentiable with derivative $f(x_{0})$.\label{ponfar}
\item For any regulated $f:[0,1]\di \R$, there is a Darboux point. \label{pon17}
\item For any regulated $f:[0,1]\di \R$, its Darboux points are dense. \label{pon18}
\item For any regulated $f:[0,1]\di \R$ with only removable discontinuities, there is $x\in [0,1]$ which is \textbf{not} a strict\footnote{A point $x\in [0,1]$ is a strict local maximum of $f:[0,1]\di \R$ in case $(\exists N\in \N)( \forall y \in B(x, \frac{1}{2^{N}}))(y\ne x\di f(y)<f(x))$.} local maximum. \label{pon19}

\end{enumerate}
\end{thm}
\begin{proof}
%
%
%
First of all, we prove item \eqref{pon2} from $\NIN_{\alt}$; we may use Volterra's corollary as in Theorem \ref{flonk}.  Fix regulated $f:[0,1]\di \R$ and consider this case distinction:
\begin{itemize}
\item if there is $q\in \Q\cap [0,1]$ with $f(q+)= f(q-)=f(q)$, item \eqref{pon2} follows.
\item if there is no such rational, then Volterra's corollary guarantees there is $x\in [0,1]\setminus \Q$ such that $f$ is continuous at $x$.
\end{itemize}
In each case, there is a point of continuity for $f$, i.e.\ item \eqref{pon2} follows. 
To prove that item \eqref{pon2} implies $\NIN_{\alt}$, let $X:=\cup_{n\in \N}X_{n}$ be the union of finite sets $X_{n}\subset [0,1]$ and define $h$ as in \eqref{mopi}.
As for $g$ from \eqref{modi2} in  Theorem \ref{flonk}, $h$ is regulated.  
Item~\eqref{pon2} then provides a point of continuity $y\in [0,1]$ of $h$, which by definition must be such that $y\not \in X$.
The same holds for quasi- and lower semi-continuity.  

\smallskip

The implication \eqref{pon3}$\di$\eqref{pon2} is immediate (as the empty set is not dense in $[0,1]$).
To prove \eqref{pon1}$\di$\eqref{pon3}, let $f$ be regulated and such that $C_{f}$ is not dense.  To derive $\neg \NIN_{\alt}$, consider $X_{n}$ as in \eqref{sameold}, which is finite for all $n\in \N$ (as is proved using $\QFAC^{0,1}$ and the Bolzano-Weierstrass theorem).
Since $C_{f}$ is not dense in $[0,1]$, there is $y\in [0,1]$ and $N\in \N$ such that $B(y, \frac{1}{2^{N}})\cap C_{f}=\emptyset$.  
By definition, the set $D_{f}=\cup_{n\in \N}X_{n}$ collects all points where $f$ is discontinuous.  Hence, $[0,1]\setminus C_{f}=D_{f}$, yielding $B(y, \frac{1}{2^{N}})\subset D_{f}$.  Now define $Y_{n}=X_{n}\cap B(y, \frac{1}{2^{N}})$, which is finite since $X_{n}$ is finite.  
Hence, $\cup_{n\in \N}Y_{n}=B(y, \frac{1}{2^{N}})$, i.e.\ an interval can be expressed as the union over $\N$ of finite sets, which readily yields $\neg\NIN_{\alt}$ after rescaling as in Remark \ref{kimu}. 

\smallskip

Regarding item \eqref{pon35}, it suffices to derive the latter from $\NIN_{\alt}$, which is immediate as $D_{f}$ is height countable. 
Indeed, if $C_{f}$ is also height countable, then $[0,1]=C_{f}\cup D_{f}$ is height countable, contradicting $\NIN_{\alt}$.  
In case $A\subset [0,1]$ is countable, then any $Y:[0,1]\di \N$ injective on $A$ is also a height function, i.e.\ $A$ is also height countable.
The same holds for strongly countable and enumerable sets.

\smallskip

Regarding item \eqref{pon8} and \eqref{pon9}, the latter immediately follow from item \eqref{pon2} and \eqref{pon3}. Indeed, in case $f(x)>0$ for $x\in C_{f}$, then by continuity there are $k, N\in \N$ such that 
$f(y)>\frac{1}{2^{k}}$ for $y\in B(x, \frac{1}{2^{N+1}})$, implying $\int_{0}^{1}f(x)dx>\frac{1}{2^{k}2^{N}}>0$.   Now assume item \eqref{pon8}, let $(X_{n})_{n\in \N}$ be a sequence of finite sets, and let $h$ be as in \eqref{mopi}.  The latter is Riemann integrable with $\int_{0}^{1}h(x)dx=0$, which one shows via the usual `epsilon-delta' definition and $\FIN$.  Any $y\in [0,1]$ such that $h(y)=0$, also satisfies $y\not \in \cup_{n\in \N} X_{n}$, i.e.\ $ \NIN_{\alt}$ follows.

\smallskip

Next, we clearly have \eqref{pon3}$\di$\eqref{pon13} and \eqref{pon35}$\di \eqref{pon14}$ since $D=C_{f}$ is as required.  To prove \eqref{pon13}$\di\NIN_{\alt}$, let $(X_{n})_{n\in \N}$ be a sequence of finite sets in $[0,1]$ and consider the regulated function $h$ from \eqref{mopi}.  Let $D$ be the dense set provided by item~\eqref{pon13} and consider $y\in D$.  In case $h(y)\ne 0$, say $h(y)>\frac{1}{2^{k_{0}}}$, use $\FIN$ to enumerate $\cup_{n\leq k_{0}+1} X_{n}$.  
Hence, we can find $N\in \N$ such that for $z\in B(x, \frac{1}{2^{N}})$, $h(z)<\frac{1}{2^{k_{0}+1}}$, i.e.\ $h_{\upharpoonright D}$ is not continuous (on $D$).  
This contradiction implies that $h(y)=0$, meaning $y\in [0,1]\setminus\cup_{n\in \N}X_{n}$, i.e.\ $\NIN_{\alt}$ follows.  A similar proof works for item \eqref{pon14} by considering a point from the dense set of continuity points of $f_{\upharpoonright D}$. 

\smallskip

Next, \eqref{pon2} (resp.\ \eqref{pon3}) clearly implies \eqref{pon15} (resp.\ \eqref{pon16}).
To show that \eqref{pon15} and \eqref{pon16} implies $\NIN_{\alt}$, one proceeds as in the previous paragraphs. 
Similarly, \eqref{pon2} (resp.\ \eqref{pon3}) implies \eqref{pon17} (resp.\ \eqref{pon18}), as any continuity point is a Darboux point, by definition.  
To show that \eqref{pon17} implies $\NIN_{\alt}$, one considers the function $h$ as in \eqref{mopi} and notes that a Darboux point of $h$ is not in $\cup_{n\in \N}X_{n}$.  

\smallskip

For item \eqref{ponfar}, the usual epsilon-delta proof establishes that $F(x):=\int_{0}^{x}f(t)dt$ is continuous and that $F'(y)=f(y)$ in case $f$ is continuous at $y$, i.e.\ item \eqref{pon2} implies item \eqref{ponfar}.
As noted above, $h$ as in \eqref{mopi} satisfies $H(x):=\int_{0}^{x}h(x)dt=0$ for any $x\in [0,1]$, i.e.\ any $y\in (0,1)$ such that $H'(y)=0=h(y)$ is such that $y\not\in \cup_{n\in \N}X_{n}$, yielding $\NIN_{\alt}$ as required.

\smallskip

Finally, assume item \eqref{pon19}, let $(X_{n})_{n\in \N}$ be a sequence of finite sets, and note that $h$ as in \eqref{mopi} is regulated with only removable discontinuities.  
Now, the set $X=\cup_{n\in \N}X_{n}$ consists of the local strict maxima of $h$, i.e.\ item \eqref{pon19} yields $\NIN_{\alt}$.  
For the reversal, $\exists^{2}$ computes a functional $M$ such that $M(g, a, b)$ is a maximum of $g\in C([0,1])$ on $[a,b]\subset [0,1]$ (see \cite{kohlenbach2}*{\S3}), i.e.\ $(\forall y\in [a,b])(g(y)\leq g(M(g, a,b)))$.
Using the functional $M$, one readily shows that `$x$ is a strict local maximum of $g$' is decidable\footnote{If $g\in C([0,1])$, then $x\in [0,1]$ is a strict local maximum iff for some $\epsilon\in \Q^{+}$:
\begin{itemize}
\item $g(y) < g(x)$ whenever $|x-y] < \epsilon$ for any $q\in [0,1]\cap\Q $, and: 
\item $\sup_{y\in [a,b]}g(y) < g(x)$ whenever $x \not \in [a,b]$, $a,b\in \Q$ and $[a,b] \subset [x - \epsilon,x + \epsilon]$.
\end{itemize}
Note that $\mu^{2}$ readily yields $N\in \N$ such that $(\forall y\in B(x, \frac{1}{2^{N}}))( g(y)<g(x))$.\label{roofer}
} given $\exists^{2}$, for $g$ continuous on $[0,1]$.  
Now let $f:[0,1]\di \R$ be regulated and with only removable discontinuities.  
Use $\exists^{2}$ to define $\tilde{f}:[0,1]\di \R$ as follows: $\tilde{f}(x):= f(x+)$ for $x\in [0, 1)$ and $\tilde{f}(1)=f(1-)$.
By definition, $\tilde{f}$ is continuous on $[0,1]$, and $\exists^{2}$ computes a (continuous) modulus of continuity, which follows in the same way as for Baire space (see e.g.\ \cite{kohlenbach4}*{\S4}).
In case $f$ is discontinuous at $x\in [0,1]$, the latter point is a strict local maximum of $f$ if and only if $f(x)>f(x+)$ (or $f(x)>f(x-)$ in case $x=1$).  
Note that $\mu^{2}$ (together with a modulus of continuity for $\tilde{f}$) readily yields $N_{f, x}\in \N$ such that $(\forall y\in B(x, \frac{1}{2^{N_{f,x}}}))( f(y)<f(x))$, in case $x$ is a strict local maximum of $f$.
In case $f$ is continuous at $x\in [0,1]$, we can use $\exists^{2}$ to decide whether $x$ is a local strict maximum of $\tilde{f}$.  
By Footnote \ref{roofer}, $\mu^{2}$ again yields $N_{f, x}\in \N$ such that $(\forall y\in B(x, \frac{1}{2^{N}}))( \tilde{f}(y)<\tilde{f}(x)))$, in case $x$ is a strict local maximum of $\tilde{f}$.
Now consider the following set:
\[\textstyle
A_{n}:=\{x\in [0,1]:  \textup{$x$ is a strict local maximum of $\tilde{f}$ or $f$ with $n\geq N_{{f}, x}$}\}.
\]
Then $A_{n}$ has is finite as strict local maxima cannot be `too close'.  
Hence, $\NIN_{\alt}$ yields $y \in [0,1]\setminus \cup_{n\in \N}A_{n}$, which is not a local maximum of $f$ by definition, i.e.\ item \eqref{pon19} follows, and we are done. 
%
%
%
%
\end{proof}
We may view item \eqref{pon1} as an extremely basic version of the connectedness of $[0,1]$, as defined by Jordan in \cite{jordel3}*{p.\ 24-28}.
Similarly, item \eqref{ponfar} is an extremely basic version of the fundamental theorem of calculus (FTC) and item \eqref{pon2} is an extremely basic version of the Lebesgue criterion for Riemann integrability.  
Moreover, it seems {necessary} to formulate items \eqref{pon8}, \eqref{pon9}, and \eqref{ponfar} with the extra condition that the functions at hand be Riemann integrable.  
As an exercise, the reader should prove that $h:[0,1]\di \R$ as in \eqref{mopi} is \emph{effectively} Riemann integrable, i.e.\ there
is a functional that outputs the `$\delta>0$' on input the `$\eps>0$' as in the usual epsilon-delta definition of Riemann integrability. 

\smallskip

Moreover, the notion of `left and right limit' gives rise to the notion of `left and right derivative'; following item \eqref{ponfar}, the left (resp.\ right) limit of regulated $f:[0,1]\di \R$ equals the left (resp.\ right) derivative of $F$ at every $x_{0}\in (0,1)$ via a completely elementary proof (say in $\ACAo$).  
We could also formulate item \eqref{pon2} with \emph{approximate continuity} (see e.g.\ \cite{broeker}*{II.5}) in the conclusion, but this notion seems to involve a lot of measure theory.  
 
\subsubsection{Restrictions of regulated functions}\label{restri}
In this section, we show that the above equivalences for $\NIN_{\alt}$ remain valid if we impose certain natural restrictions.

\smallskip

First of all, the above results show that we have to be careful with intuitive statements like \emph{regulated functions are `close to continuous'}.
 Indeed, by Theorem~\ref{duck}, $\Z_{2}^{\omega}+\QFAC^{0,1}$ is consistent with the existence of regulated functions that are discontinuous everywhere.  
 Similarly, $\NIN$ follows from the statement \emph{a regulated function is Baire 1} by \cite{dagsamXIV}*{Theorem 2.34}.  Hence, $\Z_{2}^{\omega}+\QFAC^{0,1}$ cannot prove the latter basic fact and the restriction in item \eqref{Dib} in Theorem \ref{ducksss} is therefore non-trivial.  Similar results hold for items items \eqref{Dib3}-\eqref{Dib2} in Theorem \ref{ducksss} following \cite{dagsamXIV}*{\S2.8}. 
 \begin{thm}[$\ACAo+\QFAC^{0,1}+\FIN$]\label{ducksss}
The following are equivalent.  
\begin{enumerate}
\renewcommand{\theenumi}{\roman{enumi}}
\item The uncountability of $\R$ as in $\NIN_{\alt}$, 
\item For any regulated and \textbf{Baire 2} function $f:[0,1]\di \R$, there is $x\in [0,1]$ where $f$ is continuous.\label{Dib}
\item For any regulated and \textbf{Baire 1$^{*}$} function $f:[0,1]\di \R$, there is $x\in [0,1]$ where $f$ is continuous.\label{Dib3}
\item For any regulated and \textbf{usco} \(or: lsco\) function $f:[0,1]\di \R$, there is $x\in [0,1]$ where $f$ is continuous.\label{Dib2}
\end{enumerate}
\end{thm}
\begin{proof}
In light of Theorem \ref{duck}, we only need to show that items \eqref{Dib}-\eqref{Dib2} imply $\NIN_{\alt}$. 
The function $h:[0,1]\di \R$ from \eqref{mopi} is central in the proof of the former theorem.  
For item \eqref{Dib}, $h$ is Baire 2 as follows: define $h_{n}(x)$ as $h(x)$ in case $x\in \cup_{m\leq n}X_{n}$, and $0$ otherwise. 
By definition, $h$ is the pointwise limit of $(h_{n})_{n\in \N}$ and $\FIN$ allows us to enumerate $\cup_{m\leq n_{0}}X_{n}$ for fixed $n_{0}\in \N$.
Hence, for fixed $n_{0}\in \N$, $h_{n_{0}}$ has at most finitely many points of discontinuity.  
In particular, there is an obvious sequence of continuous function with pointwise limit $h_{n_{0}}$, i.e.\ the latter is Baire 1. 
For item \eqref{Dib2}, the function $h$ is also usco as follows: in case $h(x_{0})=0$ for $x_{0}\in [0,1]$, the definition of usco is trivially satisfied. 
In case $h(x_{0})=\frac{1}{2^{n+1}}$, then there is $N\in \N$ such that $(\forall y\in B(x, \frac{1}{2^{N}}))( y\not \in \cup_{m\leq n-1}X_{m}   )$ as in the proof of Theorem \ref{duck}.  
In this case, the definition of usco is also satisfied.  The function $\lambda x. (1-h(x))$ is lsco.  For item \eqref{Dib3}, define the closed sets $C_{n}:= \cup_{m\leq n}X_{m}$ and note that the restriction of $h$ to $C_{n}$ is continuous for each $n$, i.e.\ $h$ is also Baire 1$^{*}$.
\end{proof}
Secondly, we point out one subtlety in the previous proof: it is \emph{only} shown that $h:[0,1]\di \R$ from \eqref{mopi} is Baire 2.  In particular, 
we cannot\footnote{The results in \cite{dagsamXIV}*{\S2.6} establish that $\ACAo+\ATR_{0}$ plus extra induction can prove numerous theorems about $BV$-functions \emph{if} we assume the latter are also Baire 1.  This can be generalised from `$BV$' to `regulated' and from `Baire 1' to `Baire 2 given as an iterated limit of a double sequence of continuous functions'.  The technical details are however rather involved.} construct a double sequence of continuous functions with iterated limit equal to $h$.
As it happens, Baire himself notes in \cite{beren2}*{p.\ 69} that Baire 2 functions can be \emph{represented} by such double sequences. 
We \emph{could} generalise item \eqref{Dib} in Theorem \ref{ducksss} to any higher Baire class and beyond, i.e.\ the latter theorem constitutes robustness in the flesh.  

\smallskip

Moreover, going against our intuitions, we cannot replace `Baire 2' by `Baire 1' in Theorem \ref{ducksss} as the latter condition renders items \eqref{Dib}-\eqref{Dib2} provable in $\ACAo+\QFAC^{0,1}$ by the results in Section \ref{plif}.
In particular, while Baire $1^{*}$ and usco are subclasses of Baire 1, say in $\ZF$ or $\Z_{2}^{\Omega}$, these inclusions do not necessarily hold in weaker systems.  
For instance, it is consistent with $\Z_{2}^{\omega}+\QFAC^{0,1}$ that there are totally discontinuous usco and regulated functions (see Theorem \ref{duck}).
An interesting RM-question would be to calibrate the strength of some of the well-known inclusions, like \emph{a regulated function on the unit interval is Baire 1}.

\smallskip

Thirdly, the \emph{supremum principle} for regulated functions implies $\NIN$ by \cite{dagsamXIV}*{Theorem 2.32} where the former principle states the existence of $F:\Q^{2}\di \R$ such that $F(p, q)=\sup_{x\in [p,q]}f(x)$ for any $p, q\subset [0,1]\cap \Q$.  Indeed, using the well-known interval-halving technique, a supremum functional for $h:[0,1]\di \N$ as in \eqref{mopi} would allow us to enumerate the associated union $\cup_{n\in \N}X_{n}$, i.e.\ $\NIN_{\alt}$ readily follows.  
Perhaps surprisingly, 
the equivalences for $\NIN_{\alt}$ from the previous sections still go through if we restrict to regulated functions with a supremum functional.
Regarding item \eqref{super2}, the Heaviside function is regulated but not symmetrically continuous, where the latter notion goes back to Hausdorff (\cite{p62}).  
\begin{thm}[$\ACAo+\QFAC^{0,1}+\FIN$]\label{diender}
The following are equivalent.  
\begin{enumerate}
\renewcommand{\theenumi}{\roman{enumi}}
\item The uncountability of $\R$ as in $\NIN_{\alt}$.
\item For regulated $f:[0,1]\di \R$ with a supremum functional, there is $x\in [0,1]$ where $f$ is continuous.\label{super1}
\item For regulated and lsco $f:[0,1]\di \R$ with a supremum functional, there is $x\in [0,1]$ where $f$ is continuous.\label{super11}
\item For regulated and symmetrically continuous $f:[0,1]\di \R$, there is $x\in [0,1]$ where $f$ is continuous.\label{super2}
\end{enumerate}
\end{thm}
\begin{proof}
The first item implies the other items by Theorem \ref{duck}.  
Now assume the second item and let $Y:[0,1]\di \N$ be an injection and define $e(x):= \sum_{n=0}^{Y(x)+1}\frac{x^{n}}{n!} $.  
By definition, we have $e(x)<e^{x}$ for $x\in [0,1]$.  Using the second item of $\FIN$, we have $e(x+)=e(x-)=e^{x}$ for $x\in (0,1)$.  
Indeed, for small enough neighbourhoods $U$ of $x\in (0,1)$, $Y$ is arbitrarily large on $U\setminus \{x\}$, while $e^{x}$ is uniformly continuous on $[0,1]$.
Hence, $\lambda x.e(x)$ is regulated (and lsco) and $\sup_{x\in [p,q]}e(x)=e^{q}$ also follows.  
Since the former function is totally discontinuous, we obtain a contradiction. 
To show that $\lambda x.e(x)$ is also symmetrically continuous, note that the second item of $\FIN$ implies that $|e(x+h)-e(x-h)|$ is arbitrarily small for small enough $h\in \R$.  
\end{proof}
%
Finally, there are a number of equivalent definitions of `Baire 1' on the reals (\cites{leebaire,beren,koumer}), including the following ones by \cite{koumer}*{Theorem 2.3} and \cite{kura}*{\S34, VII}.   
\bdefi\label{frag}~
\begin{itemize}
\item Any $f:[0,1]\di \R$ is \emph{fragmented} if for any $\eps>0$ and closed $C\subset [0,1]$, there is non-empty relatively\footnote{For $A\subseteq B\subset \R$, we say that \emph{$A$ is relatively open \(in $B$\)} if for any $a\in A$, there is $N\in \N$ such that $B(x, \frac{1}{2^{N}})\cap B\subset A$.  Note that $B$ is always relatively open in itself.} open $O\subset C$ such that $\textup{\textsf{diam}}(f(O))<\eps$.
\item Any $f:[0,1]\di \R$ is \emph{$B$-measurable of class 1} if for every open $Y\subset \R$, the set $f^{-1}(Y)$ is $\F_{\sigma}$, i.e.\ a union over $\N$ of closed sets. 
\end{itemize}
\edefi
The \emph{diameter} of a set $X$ of reals is defined as usual, namely $\textup{\textsf{diam}}(X):=\sup_{x, y \in X}|x-y|$, where the latter supremum need not exist for Definition \ref{frag}.  
We have the following theorem, similar to Theorem \ref{ducksss}.
\begin{thm}[$\ACAo+\IND_{0}$]\label{clockal} The following are equivalent. 
\begin{enumerate}
 \renewcommand{\theenumi}{\alph{enumi}}
\item The uncountability of $\R$ as in $\NIN_{\alt}$.\label{fra1}
\item For fragmented and regulated $f:[0,1]\di \R$, there is a point $x\in [0,1]$ where $f$ is continuous. \label{fra2}
\item For $B$-measurable of class 1 and regulated $f:[0,1]\di \R$, there is $x\in [0,1]$ where $f$ is continuous.\label{fra3}
\end{enumerate}
\end{thm}
\begin{proof}
In light of Theorem \ref{duck}, it suffices to prove that $h$ from \eqref{mopi} is fragmented and $B$-measurable of class $1$. 
For the former notion, for fixed $k\in \N$, $\FIN$ can enumerate the (finitely many) $x\in [0,1]$ such that $h(x)\geq \frac{1}{2^{k}}$.  Any open set $O$ not 
including these points is such that $\textsf{diam}(h(O))<\frac{1}{2^{k}}$, showing that $h$ is fragmented.  

\smallskip

For the $B$-measurability (of first class), in case $(X_{n})_{n\in \N}$ is a sequence of finite sets such that $[0,1]=\cup_{n\in \N}X_{n}$, note that for any $Z\subset \R$, the set $h^{-1}(Z)$ is the union of those $X_{n}$ such that $\frac{1}{2^{n+1}}\in Z$, i.e.\ $\F_{\sigma}$ by definition.  
\end{proof}
We show in Section \ref{plif} that most of the above statements that are equivalent to $\NIN_{\alt}$, become provable in the much weaker system $\ACAo+\QFAC^{0,1}+\FIN$ if we additionally require the functions to be Baire 1 \emph{as in Definition} \ref{flung}.  

\subsection{Bounded variation and the uncountability of $\R$}\label{mainbv}
\subsubsection{Introduction}\label{schintro}
 In this section, we establish the equivalences sketched in Section \ref{intro} pertaining to the uncountability of $\R$ and properties of $BV$-functions. 
In particular, we study the following weakening of $\NIN_{\alt}$ involving the notion of height-width countability from Definition \ref{hoogzalieleven2}.
\begin{princ}[$\NIN_{\alt}'$]
The unit interval is not height-width countable.  
\end{princ}
Equivalences for $\NIN_{\alt}'$ will involve some (restrictions of) items from Theorems~\ref{flonk} and~\ref{duck}, but also a number of theorems from analysis that hold for $BV$-functions and not for regulated ones. 
For the latter, we need some additional definitions, found in Section \ref{defz}, while the equivalences are in Section \ref{REZ}

\smallskip

Finally, we first establish Theorem \ref{deville}, which is interesting because we are unable to derive $\NIN_{\alt}$ from the items listed therein.  
The exact definitions of $\HBU$ and $\WHBU$ are below, where the former expresses that the uncountable covering $\cup_{x\in [0,1]}B(x, \Psi(x))$ has a finite sub-covering, i.e.\ the \emph{Heine-Borel theorem} or \emph{Cousin lemma} (\cite{cousin1}). 
The principle $\WHBU$ is the combinatorial essence of the Vitali covering theorem, as studied in \cite{dagsamVI}.
\begin{princ}[$\HBU$;\cite{dagsamIII}]
For any $ \Psi:\R\di \R^{+}$, there are $  y_{0}, \dots, y_{k}\in [0,1]$ such that $ \cup_{i\leq k}B(y_{i}, \Psi(y_{i}))$ covers $[0,1]$.
\end{princ}
\begin{princ}[$\WHBU$;\cite{dagsamVI}] For any $\Psi:\R\di \R^{+} $ and $ \eps>_{\R}0$, there are $ y_{0}, \dots, y_{k}\in [0,1]$ such that $\cup_{i\leq k}B(y_{i}, \Psi(y_{i}))$ has measure at least $1-\eps$. 
\end{princ}
We note that $\WHBU$ can be formulated without using the Lebesgue measure, as discussed at length in e.g.\ \cite{dagsamVI} or \cite{simpson2}*{X.1}.
We conjecture that $\NIN_{\alt}$ is not provable in $\ACAo+\HBU$.
\begin{thm}[$\ACAo+\FIN$]\label{deville}
The following theorems imply $\NIN_{\alt}'$:
\begin{itemize}
\item  Jordan decomposition theorem for $[0,1]$.
\item The principle $\HBU$ restricted to $BV$-functions.  
\item The principle $\WHBU$ restricted to $BV$-functions.
\end{itemize}
\end{thm}
\begin{proof}
For the first item, let $(Y_{n})_{n\in \N}$ be a sequence of finite sets with width bound $g\in \N^{\N}$.  
The function $k:[0,1]\di \R$ from \eqref{mopi2} has bounded variation, with upper bound 1 by definition, for which we need $\FIN$. 
By \cite{dagsamXII}*{Lemma 7}, $\mu^{2}$ can enumerate the points of discontinuity of a monotone function, i.e.\ the Jordan decomposition
theorem provides a sequence $(x_{n})_{n\in \N}$ that enumerates the points of discontinuity of a $BV$-function.  Using the usual diagonal argument (see e.g.\ \cite{simpson2}*{II.4.9}), we can find a
point not in this sequence, yielding $\NIN_{\alt}'$.

\smallskip

For the remaining items, let $(Y_{n})_{n\in \N}$ again be a sequence of finite sets with width bound $g\in \N^{\N}$.  
Suppose $[0,1]=\cup_{n\in \N}Y_{n}$ and define $\Psi:[0,1]\di \R^{+}$ as follows $\Psi(x):= \frac{1}{2^{n+5}}\frac{1}{g(n)+1}$ where $x\in Y_{n}$ and $n$ is the least such number. 
For $x_{0}, \dots, x_{k}\in [0,1]$, the measure of $\cup_{i\leq k} B(x_{i}, \Psi(x_{i}))$ is at most $1/2$ by construction, contradicting $\HBU$ and $\WHBU$.  
Using $\FIN$, one readily shows that $\Psi$ is in $BV$.  
%
%
%
%
\end{proof}
The principle $\HBU$ is studied in \cite{basket2, basket} for $\Psi$ represented by e.g. second-order Borel codes.  
This `coded' version is provable in $\ATR_{0}$ extended with some induction.  By contrast and Theorem \ref{deville}, $\HBU$ restricted to $BV$-functions, which are definitely Borel,
implies $\NIN_{\alt}'$, which in turn is not provable in $\Z_{2}^{\omega}$.  Thus, the use of codes fundamentally changes the logical strength of $\HBU$. 
A similar argument can be made for the Jordan decomposition theorem, studied for second-order codes in \cites{nieyo}

\subsubsection{Definitions}\label{defz}
We introduce some extra definitions needed for the RM-study of $BV$-functions as in Section \ref{REZ}.

\smallskip
First of all, we shall study \emph{unordered sums}, which are a device for bestowing meaning upon `uncountable sums' $\sum_{x\in I}f(x)$ for \emph{any} index set $I$ and $f:I\di \R$.  
A central result is that if $\sum_{x\in I}f(x)$ somehow exists, it must be a `normal' series of the form $\sum_{i\in \N}f(y_{i})$, i.e.\ $f(x)=0$ for all but countably many $x\in [0,1]$; Tao mentions this theorem in \cite{taomes}*{p.~xii}. 

\smallskip

By way of motivation, there is considerable historical and conceptual interest in this topic: Kelley notes in \cite{ooskelly}*{p.\ 64} that E.H.\ Moore's study of unordered sums in \cite{moorelimit2} led to the concept of \emph{net} with his student H.L.\ Smith (\cite{moorsmidje}).
Unordered sums can be found in (self-proclaimed) basic or applied textbooks (\cites{hunterapp,sohrab}) and can be used to develop measure theory (\cite{ooskelly}*{p.\ 79}).  
Moreover, Tukey shows in \cite{tukey1} that topology can be developed using \emph{phalanxes}, which are nets with the same index sets as unordered sums.  

\smallskip

Now, unordered sums are just a special kind of \emph{net} and $a:[0,1]\di \R$ is therefore written $(a_{x})_{x\in [0,1]} $ in this context to suggest the connection to nets.  
The associated notation $\sum_{x\in [0,1]}a_{x}$ is purely symbolic.   
We only need the following notions in the below. 
Let $\fin(\R)$ be the set of all finite sequences of reals without repetitions.  
\bdefi\label{kaukie} Let $a:[0,1]\di \R$ be any mapping, also denoted $(a_{x})_{x\in [0,1]}$.
\begin{itemize}
\item We say that $\sum_{x\in [0,1]} a_{x}$ is \emph{Cauchy} if there is $\Phi:\R\di \fin(\R)$ such that for $\eps>0$ and all $J\in \fin({\R})$ with $J\cap\Phi(\eps)=\emptyset$, we have $|\sum_{x\in J}a_{x}|<\eps$.
\item We say that $\sum_{x\in [0,1]}a_{x} $ is \emph{bounded} if there is $N_{0}\in \N$ such that for any $J\in \fin(\R)$, $N_{0}>|\sum_{x\in J}a_{x}|$. 
\end{itemize}
\edefi
Note that in the first item, $\Phi$ is called a \emph{Cauchy modulus}.  
For simplicity, we focus on \emph{positive unordered sums}, i.e.\ $(a_{x})_{x\in [0,1]}$ such that $a_{x}\geq 0$ for $x\in [0,1]$.

\smallskip 

Secondly, there are many spaces between the regulated and $BV$-functions, as discussed in Remark \ref{essenti}.  
We shall study one particular construct, called \emph{Waterman variation}, defined as follows. 
\bdefi\label{kwar}
A decreasing sequence of positive reals $\Lambda=(\lambda_{n})_{n\in \N}$ is a \emph{Waterman sequence} if $\lim_{n\di \infty}\lambda_{n}=0$ and $\sum_{n=0}^{\infty}\lambda_{n}=\infty$.
\edefi
\bdefi[Waterman variation]\label{Warwar}
The function $f:[a,b]\di \R$ has \emph{bounded Waterman variation with sequence $\Lambda=(\lambda_{n})_{n\in \N}$} on $[a,b]$ 
if there is $k_{0}\in \N$ such that $k_{0}\geq \sum_{i=0}^{n} \lambda_{i} |f(x_{i})-f(x_{i+1})|$ 
for any finite collection of pairwise non-overlapping intervals $(x_{i}, x_{i+1})\subset [a,b]$.\label{donpw}
\edefi
Note that Definition \ref{Warwar} is \emph{equivalent} to the `official' definition of Waterman variation by \cite{voordedorst}*{Prop.\ 2.18}.
In case $f:[0,1]\di \R$ has bounded Waterman variation (with sequence $\Lambda$ as in Definition \ref{kwar}), we write `$f\in \Lambda BV$'.

\smallskip

Thirdly,  we make use of the usual definitions of Fourier coefficients and series.   
\bdefi\label{popolop}
The \emph{Fourier series} $S(f)(x)$ of $f:[-\pi, \pi]\di \R$ at $x\in [-\pi, \pi]$ is: 
\be\label{daria}\textstyle
\frac{a_{0}}{2} +  \sum_{n=1}^{\infty} ( a_{k} \cdot \cos(nx)+b_{k} \cdot \sin(nx) ),
\ee
with \emph{Fourier coefficients} $ a_{n} := \frac{1}{\pi}\int_{-\pi}^{\pi}   f (t)\cos(nt)dt$ and $ b_{n} :=\frac1\pi \int_{-\pi}^{\pi}   f (t) \sin(nt)dt$.
\edefi
\noindent
By the proof of Theorem \ref{nirvana}, the Fourier coefficients and series of $k:[0,1]\di \R$ as in \eqref{mopi} all exist.  
Our study of the Fourier series $S(f)$ as in \eqref{daria} will always assume (at least) that the Fourier coefficients of $f$ exist. 
Functions of bounded variation are of course Riemann integrable and similarly have Fourier coefficients, but only in sufficiently strong systems that seem to dwarf $\NIN_{\alt}'$.

\subsubsection{Basic equivalences}\label{REZ}
We establish a number of equivalences for $\NIN_{\alt}'$ and basic properties of $BV$-functions, some similar to those in Theorems \ref{flonk} and \ref{duck}, some new.  
We note that $\QFAC^{0,1}$ is no longer needed in the base theory. 
\begin{thm}[$\ACAo+\FIN$]\label{flunk2}
The following are equivalent.
\begin{enumerate}
\renewcommand{\theenumi}{\roman{enumi}}
\item The uncountability of $\R$ as in $\NIN_{\alt}'$. \label{bv3}
\item For a positive unordered sum $\sum_{x\in [0,1]}a_{x}$ with upper bound \(or Cauchy modulus; Def.\ \ref{kaukie}\), there is $y\in [0,1]$ such that $a_{y}=0$.\label{bv7}
\item For a positive unordered sum $\sum_{x\in [0,1]}a_{x}$ with upper bound \(or Cauchy modulus; Def.\ \ref{kaukie}\), the set  $\{ x\in [0,1]: a_{x}=0\}$ is dense \(or: not height countable, or: not countable, or: not strongly countable).\label{bv77}
\item[(iv)-(xix)] Any of items \eqref{volare1} to \eqref{volare3} from Theorem \ref{flonk} or items \eqref{pon2} to \eqref{pon19} from Theorem \ref{duck} restricted to $BV$-functions.\label{bv0}
\setcounter{enumi}{+19}
\item For a Waterman sequence $\Lambda$ and $f:[0,1]\di \R$ in $\Lambda BV$, there is $y\in  [0,1]$ where $f$ is continuous.\label{bv8}
\item[(xxi)-(xxxvi)] Any of items \eqref{volare1} to \eqref{volare3} from Theorem \ref{flonk} or items \eqref{pon2} to \eqref{pon19} from Theorem \ref{duck} restricted to $\Lambda BV$ for any fixed Waterman sequence $\Lambda$.\label{bv0s}
\end{enumerate}
\end{thm}
\begin{proof}
The equivalences involving the restrictions from regulated to $BV$ functions follow from the proofs of Theorems~\ref{flonk} and \ref{duck}.  Indeed, for $f\in BV$ with variation bounded by $1$, the set $X_{n}$ from \eqref{sameold} has size bounded by $2^{n}$ since each element of $X_{n}$ contributes at least $1/2^{n}$ to the variation.  
Moreover, rather than $h:[0,1]\di \R$ as in \eqref{mopi}, we use $k:[0,1]\di \R$ as in \eqref{mopi2}
which has variation bounded by $1$ if $(Y_{n})_{n\in \N}$ is a sequence of sets with width function $g\in \N^{\N}$.  The properties of $k:[0,1]\di \R$ are readily proved using $\FIN$; in particular, the width function $g$ obviates the use of $\QFAC^{0,1}$ as in the proofs of Theorem \ref{flonk} and \ref{duck}.  
Hence, the equivalences between $\NIN_{\alt}'$ and items (iv)-(xix) has been established. 

\smallskip

The equivalence between items \eqref{bv3} and \eqref{bv7} is as follows: assume the latter and let $(X_{n})_{n\in \N}$ and $g:\N\di \N$ be as in item \eqref{bv3}. 
Define $(a_{x})_{x\in [0,1]}$ as follows:
\[
a_{x}:=  
\begin{cases}
0  & x\not \in \cup_{n\in \N} X_{n} \\
\frac{1}{2^{n}}\frac{1}{g(n)+1} & x\in X_{n} \textup{ and $n$ is the least such natural} 
\end{cases}.
\]
Clearly, this unordered sum is Cauchy and has upper bound $1$; if $y\in [0,1]$ is such that $a_{y}=0$, then $y\not \in \cup_{n\in\N}X_{n}$, as required for item \eqref{bv3}.
Now assume the latter and let $(a_{x})_{x\in [0,1]}$ be an unordered sum that is Cauchy.  Now consider the following set:
\be\label{lort}
X_{n}:=\{x\in [0,1]: a_{x}>1/2^{n}\}.
\ee
Apply the Cauchy property of $(a_{x})_{x\in [0,1]}$ for $\eps=1$, yielding an upper bound $N_{0}\in \N$ for $\sum_{x\in K}a_{x}$ for any $K\in \fin{(\R)}$.
Hence, the finite set $X_{n}$ in \eqref{lort} has size at most $2^{n}N_{0}$.  In this way, we have a width function for $(X_{n})_{n\in \N}$; 
any $y\in [0,1]\setminus \cup_{n\in \N}X_{n}$ is such that $a_{y}=0$, as required for item \eqref{bv7}.
Item \eqref{bv77} now follows in the same way as for item \eqref{lopi} in the proof of Theorem \ref{flonk}. 

\smallskip

For item \eqref{bv8}, assume the latter and note that \eqref{kik} establishes that $f\in BV$ implies $f\in \Lambda BV$ for \emph{any} Waterman sequence $\Lambda=(\lambda_{n})_{n\in \N}$:
\be\label{kik}\textstyle
 \sum_{i=0}^{n} \lambda_{i} |f(x_{i})-f(x_{i+1})|\leq \sum_{i=0}^{n} \lambda_{1} |f(x_{i})-f(x_{i+1})|\leq \lambda_{1} V_{0}^{1}(f),
\ee
as Waterman sequences are decreasing by definition.
Hence, item~\eqref{bv3} readily follows from item \eqref{bv8}.  Now assume item \eqref{bv3} and recall that for $BV$-functions with variation bounded by $1$, the set $X_{n}$ from \eqref{sameold} can have at most $2^{n}$ elements, as each element of $X_{n}$ contributes at least $\frac{1}{2^{n}}$ to the total variation. Functions with Waterman variation bounded by $1$ similarly come with explicit upper bounds on the set $X_{n}$, namely $|X_{n}|\leq K_{n}$ where $K_{n}$ is the least $k\in \N$ such that $2^{n}<\sum_{m=0}^{k}\lambda_{m}$.  
Hence, item \eqref{bv8} follows and we are done. 
%
%
%
%
%
\end{proof}
Following the proof of Theorem \ref{ducksss}, we observe that we may restrict to $BV$-function that are Baire 2 or Baire 1$^{*}$ in the previous theorem.
With the gift of hindsight we can even obtain the following corollary.
We recall that the space of regulated functions is the union over all Waterman sequences $\Lambda=(\lambda_{n})_{n\in \N}$ of the spaces $\Lambda BV$, as established in \cite{voordedorst}*{Prop.\ 2.24}.
\begin{cor}[Some generalisations]~
\begin{itemize}
\item We may replace `$f\in \Lambda BV$' in items \textup{(xx)-(xxxvi)} of the theorem by `regulated function $f:[0,1]\di \R$ with $\Lambda=(\lambda_{n})_{n\in \N}$ such that $f\in \Lambda BV$'.
\item Assuming $\ACAo+\QFAC^{0,1}+\FIN$, the higher item implies the lower one.  
\begin{itemize}
\item For any regulated $f:[0,1]\di \R$, there is a Waterman sequence $\Lambda=(\lambda_{n})_{n\in \N}$ such that $f\in \Lambda BV$.
\item The uniform finite union theorem.
\end{itemize}
\end{itemize}
\end{cor}
\begin{proof}
For the first item, the last part of the proof of the theorem provides the required upper bound function for applying $\NIN_{\alt}'$. 

\smallskip

For the second item, let $(X_{n})_{n\in \N}$ be a sequence of finite sets in $[0,1]$ and consider the regulated function $h:[0,1]\di \R$ as in \eqref{mopi}.
Suppose $h$ is in $\Lambda BV$ with $\Lambda=(\lambda_{n})_{n\in \N}$ and with upper bound $N_{0}\in \N$ on the Waterman variation.  
Then $A_{n+2}$ as in \eqref{lagel2} is $\cup_{i\leq n}X_{i}$ for all $n\in \N$ and $|A_n| \leq K_{n}$ where $K_{n}$ is the least $k\in \N$ such that $2^{n}N_{0}<\sum_{m=0}^{k}\lambda_{m}$.
The uniform finite union theorem thus follows.
\end{proof}
An equivalence is possible in the second item, but the technical details are considerable.  
Our above results suggest that the principles equivalent to $\NIN_{\alt}'$ also have a certain robustness since we can replace `one point' properties like item \eqref{bv7} in Theorem \ref{flunk2}, by e.g.\ `density' versions like item \eqref{bv77} in Theorem \ref{flunk2}.  Nonetheless, we believe we cannot replace `density' by `full measure'.  
In particular, we conjecture that `measure theoretic' statements like
\begin{itemize}
\item {a $BV$-function is continuous (or differentiable) almost$^{\ref{bruhair}}$ everywhere},
\item a height-width countable set $A\subset [0,1]$ 
has measure\footnote{The definition of `$A\subset [0,1]$ has measure zero set' can be written down without using the Lebesgue measure, just like in second-order RM (see \cite{simpson2}*{X.1}).\label{bruhair}} zero,
\end{itemize}
are \emph{strictly} stronger than $\NIN_{\alt}'$.  We do not have a proof of this claim.

\smallskip

Finally, the \emph{variation function} $\lambda x. V_{a}^{x}(f)$ is defined in the obvious way, namely based on \eqref{tomb}. 
This function shares pointwise properties like continuity and differentiability with $f:[a,b]\di \R$.  
For instance, the following equivalence for any $x\in [0,1)$ is obtained in \cite{hugsandkisses}*{Cor.\ 1.1}:
\be\label{EW}
\textup{\emph{$f$ is right-continuous at $x$ if and only if $\lambda x. V_{a}^{x}(f)$ is right-continuous at $x$}}.
\ee
Here, `right-continuous at $y\in [0,1)$' means $g(y)=g(y+)$.  Now, although the variation function may not exist for $BV$-functions, say in $\RCAo$, the right-hand side of \eqref{EW} 
makes sense using the well-known `virtual' or `comparative' meaning of suprema from second-order RM (see \cite{simpson2}*{X.1}).
Perhaps surprisingly, working over $\ACAo+\FIN+\neg \NIN_{\alt}'$, the function $k:[0,1]\di \R$ from \eqref{mopi2} satisfies:
\begin{center}
\emph{the function $\lambda x.V_{0}^{x}(k)$ is right-continuous for $x\in [0,1)$}, 
\end{center}
which is to be interpreted in the aforementioned virtual sense.  Thus, one readily proves that the following are equivalent, where $E(f, x)$ is \eqref{EW}.
\begin{itemize}
\item The uncountability of $\R$ as in $\NIN_{\alt}'$.
\item For $f:[0,1]\di \R$ in $BV$, there is $y\in [0,1]$ where $f$ is continuous.\label{ba2}
\item For $f:[0,1]\di \R$ in $BV$, there is $y\in [0,1]$ where $E(f, y)$ holds. 
\item For $f:[0,1]\di \R$ in $BV$ such that $\lambda x.V_{0}^{x}(f)$ is right-continuous on $[0,1)$, there is $y\in [0,1]$ where $f$ is right-continuous.
\end{itemize}
To be absolutely clear, we think this topic should \textbf{not} be pursued further: mistakes are (too) easily made when dealing with `virtual' objects like $\lambda x.V_{0}^{x}(f)$. 

%

\subsubsection{Advanced equivalences: Fourier series}
We obtain an equivalence for $\NIN_{\alt}'$ and properties of the Fourier series of $BV$-functions.  Since the forward direction is rather involved,
we have reserved a separate section for this result. Moreover, Theorem \ref{nirvana} is not at all straightforward: Jordan proves the convergence of Fourier series for $BV$-functions using the Jordan decomposition theorem, 
and the same for e.g.\ \cite{ziggy}*{p.\ 57-58}.  However, the latter theorem seems much stronger\footnote{The system $\FIVE^{\omega}$ plus the Jordan decomposition theorem can prove $\SIX$ (\cite{dagsamXI}).  By Theorem \ref{deville}, $\WHBU$ implies $\NIN_{\alt}'$, where the former seems weak in light of  \cite{dagsamVI}*{\S4}.} than $\NIN_{\alt}'$.  
\begin{thm}[$\ACAo+\FIN$]\label{nirvana} 
The following are equivalent to $\NIN_{\alt}'$.
\begin{itemize}
\item For $f:[-\pi,\pi]\di \R$ in $BV$ such that the Fourier coefficients exist, there is $x_{0}\in (-\pi,\pi)$ where the Fourier series $S(f)(x_{0})$ equals $f(x_{0})$.  
\item For $f:[-\pi,\pi]\di \R$ in $BV$ such that the Fourier coefficients exist, the set of $x\in (-\pi, \pi)$ where the Fourier series $S(f)(x)$ equals $f(x)$, is dense \(or: not height countable, or: not countable, or: not strongly countable\).  
\end{itemize}
\end{thm}
\begin{proof}
Assume the first item and consider $k:[0,1]\di \R$ from \eqref{mopi2}.  This function is Riemann integrable with $\int_{0}^{1}k(x)dx=0$, which one proves in the same way as for $h:[0,1]\di \R$ from \eqref{mopi} in the proof of Theorem~\ref{duck}.  Similarly (and with a suitable rescaling), the Fourier coefficients of the Fourier series of $k$ are zero.  Hence, any $x_{0}$ where the Fourier series of $k$ converges to $k(x_{0})$ must be such that $k(x_{0})=0$, as required for $\NIN_{\alt}'$ since then $x_{0}\not \in \cup_{n\in \N}X_{n}$.    

\smallskip

Secondly, by items (iv)-(xix) in Theorem \ref{flunk2}, $\NIN_{\alt}'$ implies that for a $BV$-function, the set $C_{f}$ is dense  \(or: not height countable, or: not countable, or: not strongly countable).  Hence, the second item from Theorem \ref{nirvana} is immediate \emph{if} 
we can show that $S(f)(x)$ from Definition \ref{popolop} equals $\frac{f(x+)-f(x-)}{2}$ for $f$ in $BV$ and any $x$ in the domain.  Waterman provides an elementary and \emph{almost} self-contained  
proof of this convergence fact in \cite{meerwater}, avoiding the Jordan decomposition theorem and only citing \cite{ziggy}*{Vol I, p.\ 55, (7.1)}.  The proof of the latter is straightforward trigonometry and Waterman's argument is readily formalised in $\ACAo$. 
Similarly, there are `textbook' proofs that $S(f)(x)$ equals $\frac{f(x+)-f(x-)}{2}$ for $f$ in $BV$ and $x$ in the domain that avoid the Jordan decomposition theorem.  
Such proofs generally seem to proceed as follows (see e.g.\ \cite{ziggy, watsons, easypeasyfourieranalysie}).
\begin{itemize}
\item By \emph{Fej\'er's theorem} (\cite{ziggy}*{p.\ 89} or \cite{watsons}*{p.\ 170}), Fourier series of $BV$-functions are convergent in the C\'esaro mean to $\frac{f(x+)-f(x-)}{2}$.  
\item For $BV$-functions, Fourier coefficients are $O(\frac{1}{n})$ (\cite{watsons}*{p.\ 172},\cite{ziggy}*{p.\ 48}).  
\item By \emph{Hardy's theorem} (\cite{ziggy}*{p.\ 78} or \cite{watsons}*{p.\ 156}), if a series converges in the C\'esaro mean, it also converges in case the terms are $O(\frac{1}{n})$. 
\item By \emph{C\'esaro's method of summation} (see \cite{watsons}*{p.\ 155}), if a series converges in the C\'esaro mean to a limit $s$ and the series also converges, then the series converges to the limit $s$.  
\end{itemize}
Each of these results has an elementary (sometimes tedious and lengthy) proof that readily formalises in $\ACAo$.  As an example, the proofs of the second item in \cite{watsons}*{p.~172} and \cite{voordedorst}*{p.~288} make use of the Jordan decomposition theorem, while the proofs in \cite{tafel} and \cite{ziggy}*{p.\ 48} do not \emph{and} are completely elementary. 
Finally, the perhaps `most elementary' proof based on the above items is found in \cite{easypeasyfourieranalysie}. 
\end{proof}
Regarding the conditional nature of the items in Theorem \ref{nirvana}, the Fourier coefficients of $BV$-functions of course exist by the Lebesgue criterion for the Riemann integral.  
However, that $BV$-functions have a point of continuity is already non-trivial by items (iv)-(xix) in Theorem \ref{flunk2}.  Moreover, the Darboux formulation of the Riemann integral
involves suprema of $BV$-functions, which are hard to compute by the second cluster theorem in \cite{dagsamXIII}.   

\smallskip

Finally, we could generalise the items from Theorems \ref{flunk2} and \ref{nirvana} to other notions of `generalised' bounded variation.   
The latter notions yield (many) intermediate spaces between $BV$ and regulated as follows. 
\begin{rem}[Between bounded variation and regulated]\label{essenti}\rm 
The following spaces are intermediate between $BV$ and regulated; all details may be found in \cite{voordedorst}.  

\smallskip

Wiener spaces from mathematical physics (\cite{wiener1}) are based on \emph{$p$-variation}, which amounts to replacing `$ |f(x_{i})-f(x_{i+1})|$' by `$ |f(x_{i})-f(x_{i+1})|^{p}$' in the definition of variation \eqref{tomb}. 
Young (\cite{youngboung}) generalises this to \emph{$\phi$-variation} which instead involves $\phi( |f(x_{i})-f(x_{i+1})|)$ for so-called Young functions $\phi$, yielding the Wiener-Young spaces.  
Perhaps a simpler construct is the Waterman variation (\cite{waterdragen}), which involves $ \lambda_{i}|f(x_{i})-f(x_{i+1})|$ and where $(\lambda_{n})_{n\in \N}$ is a sequence of reals with nice properties; in contrast to $BV$, any continuous function is included in the Waterman space (\cite{voordedorst}*{Prop.\ 2.23}).  Combining ideas from the above, the \emph{Schramm variation} involves $\phi_{i}( |f(x_{i})-f(x_{i+1})|)$ for a sequence $(\phi_{n})_{n\in \N}$ of well-behaved `gauge' functions (\cite{schrammer}).  
As to generality, the union (resp.\ intersection) of all Schramm spaces yields the space of regulated (resp.\ $BV$) functions, while all other aforementioned spaces are Schramm spaces (\cite{voordedorst}*{Prop.\ 2.43 and 2.46}).
In contrast to $BV$ and the Jordan decomposition theorem, these generalised notions of variation have no known `nice' decomposition theorem.  The notion of \emph{Korenblum variation} (\cite{koren}) does have such a theorem (see \cite{voordedorst}*{Prop.\ 2.68}) and involves a distortion function acting on the \emph{partition}, not on the function values (see \cite{voordedorst}*{Def.\ 2.60}).  
\end{rem}
It is no exaggeration to say that there are \emph{many} natural spaces between the regulated and $BV$-functions, all of which yield equivalences in Theorem~\ref{flunk2}.

\subsection{When more is less in Reverse Mathematics}\label{plif}
An important -if not central- aspect of analysis is the relationship between its many function classes.  
It goes without saying that these relationships need not hold in weak logical systems.   
For instance, the well-known inclusion \emph{regulated implies Baire 1} is not provable in $\Z_{2}^{\omega}+\QFAC^{0,1}$ by \cite{dagsamXIV}*{Theorem 2.34}.  

\smallskip

In this section, we establish a kind of dual to the previous negative result: we show that most of the above statements that are equivalent to $\NIN_{\alt}$ or $\NIN_{\alt'}$, become provable in the much weaker system $\ACAo+\QFAC^{0,1}+\FIN$ if we additionally require the functions to be Baire 1.  
To this end,  we need the following theorem, where a \emph{jump continuity} is any $x\in (0,1)$ such that $f(x+)\ne f(x-)$.
\begin{thm}[$\ACAo$]\label{falm}
If $f:[0,1]\di \R$ is regulated, there is a sequence of reals containing all jump discontinuities of $f$.
\end{thm}
\begin{proof}
This is immediate from \cite{dagsamXIV}*{Theorem 2.16}.
\end{proof}
\noindent
The following theorem should be contrasted with Theorems~\ref{ducksss} and \ref{flunk2}.  
\begin{thm}[$\ACAo+\FIN$]\label{falnt3}
For a Baire 1 function $f:[0,1]\di \R$ in $BV$, the points of continuity of $f$ are dense. 
\end{thm}
\begin{proof}
Let $f:[0,1]\di \R$ be Baire $1$ and in $BV$ with variation bounded by $1$.
This function is regulated by Theorem \ref{flima}.
Use Theorem \ref{falm} to enumerate the jump discontinuities of $f$ as $(y_{n})_{n\in \N}$. 
Let $(f_{n})_{n\in \N}$ be a sequence of continuous functions with pointwise limit $f$ on $[0,1]$ and consider the following formula:
\be\label{xdx}\textstyle
\varphi(n_{0}, k, x)\equiv(\forall n, m\geq n_{0})(\forall q\in B(x, \frac{1}{2^{m}})\cap \Q)(|f_{n}(x)-f(q)|\geq \frac{1}{2^{k}}+\frac{1}{2^{n_{0}}}).  
\ee
For fixed $k\in \N$, $\varphi(n_{0}, k, x)$ holds for large enough $n_{0}\in \N$ in case $f$ has a removable discontinuity at $x\in (0,1)$ such that $|f(x)-f(x+)|>\frac1{2^{k}}$.   For fixed $n_{0},k\in \N$, there can only be $2^{k}$ many pairwise distinct $x\in [0,1]$ such that $\varphi(n_{0}, k, x)$, as each
such real contributes at least $\frac1{2^{k}}$ to the total variation of $f$.  

\smallskip

Next, the formula $\varphi(n_{0}, k, x)$ is equivalent to (second-order) $\Pi_{1}^{0}$ as $f$ only occurs with rational input and $f_{n}$ can be replaced uniformly by a sequence of codes $\Phi_{n}$.  Moreover, in case $x=_{\R} y$, then trivially $\varphi(n_{0}, k, x)\asa \varphi(n_{0}, k, y)$, i.e.\ we have the extensionality property required for \cite{simpson2}*{II.5.7}.  By the latter there is an RM-code of a closed set $C_{n_{0},k}$  such that $x\in C_{n_{0},k}\asa \varphi(n_{0}, k, x)$ for all $x\in \R$ and $n_{0}, k\in \N$.  
Since $C_{n_{0}, k}$ is finite, $O_{n_{0}, k}:= [0,1]\setminus \big(C_{n_{0}, k}\cup \{y_{0}, \dots, y_{k}\} \big)$ is open and dense.  
By the Baire category theorem for RM-codes (\cite{simpson2}*{II.5.8}), the intersection $\cap_{n_{0}, k\in \N} O_{n_{0}, k} $ is dense in $[0,1]$. 
By definition, this intersection does not contain any points of discontinuity of $f$, and we are done.
\end{proof}
\noindent
The following corollary should be contrasted with Theorem \ref{nirvana}.  
\vspace{1mm}
\begin{cor}[$\ACAo+\FIN$]~
\begin{itemize}
\item For a positive unordered sum $\sum_{x\in [0,1]}a_{x}$ with upper bound and where $(a_{x})_{x\in[0,1]}$ is Baire 1, the set $\{y\in [0,1]:a_{y}=0\}$ is dense.
\item For Baire 1 function $f:[-\pi,\pi]\di \R$ in $BV$ such that the Fourier coefficients exist, the set  $\{x\in (-\pi, \pi): S(f)(x)=f(x)\}$, is dense.
\end{itemize}
\end{cor}
\begin{proof}
Let $(a_{x})_{x\in [0,1]}$ be an unordered sum with upper bound $1$.  Now consider
\be\label{port}
X_{n}:=\{x\in [0,1]: a_{x}>1/2^{n}\},
\ee
which can have at most $2^{n}$ elements.   As in the proof of the theorem, the Baire 1 approximation of $(a_{x})_{x\in [0,1]}$ allows 
us to show that $a_{x}>1/2^{n}$ is (implied by) a suitable (second-order) $\Sigma_{1}^{0}$-formula.  
One then uses the (second-order) Baire category theorem to show that $y\in [0,1]$ such that $a_{y}=0$ are dense.  

\smallskip

For the second item, following the proof of Theorem \ref{nirvana}, $f(x)=S(f)(x)$ holds in case $f$ is continuous at $x\in [0,1]$, where the latter is provided by the theorem. 
\end{proof}
Next, the first item in Theorem \ref{klam} follow from \cite{dagsamXIV}*{Theorem 2.26}, but the latter is proved using $\ACAo+\ATR_{0}$.
\begin{thm}[$\ACAo+\FIN+\QFAC^{0,1}$]\label{klam}~
\begin{itemize}
\item For regulated Baire 1 $f:[0,1]\di \R$, the points of continuity of $f$ are dense. 
\item \emph{Volterra's corollary}: there is no regulated and Baire 1 function that is continuous on $\Q\cap[0,1]$ and discontinuous on $[0,1]\setminus\Q$.
\item For a Riemann integrable and regulated $f:[0,1]\di [0,1]$ in Baire 1 with $\int_{0}^{1}f(x)dx=0$, the set $\{x\in [0,1]:f(x)=0\}$ is dense.
\item Blumberg's theorem \(\cite{bloemeken}\) restricted to regulated Baire 1 functions on $[0,1]$.
\end{itemize}
\end{thm}
\begin{proof}
For the first item, the proof of Theorem \ref{falnt3} can be modified follows:  for regulated $f$ and fixed $k\in \N$, there can be at most \emph{finitely many} $x\in [0,1]$ such that $|f(x)-f(x+)|>\frac1{2^{k}}$.  
One proves this fact by contradiction (as in the proof of Theorem \ref{flonk}), where $\QFAC^{0,1}$ provides a sequence $(x_{n})_{n\in \N}$ of reals in $ [0,1]$ such that $|f(x_{n})-f(x_{n}+)|>\frac1{2^{k}}$ for all $n\in \N$.  
This sequence has a convergent sub-sequence, say with limit $y\in [0,1]$; one readily verifies that $f(y+)$ or $f(y-)$ does not exist.  The rest of the proof is now identical to that of Theorem \ref{falnt3}. 

\smallskip

For the second item, the proof of Theorem \ref{falnt3} is readily adapted as follows: let $(q_{n})_{n\in \N}$ be an enumeration of the rationals in $[0,1]$ and define $O_{n_{0}, k}'$ as $O_{n_{0},k}\setminus \{q_{0}, \dots, q_{k}\}$.  Then Theorem \ref{falnt3} must yield an irrational point of continuity, a contradiction.

\smallskip

For the third item, note that $f(x)=0$ must hold in case $f$ is continuous at $x\in [0,1]$, where the (dense set of the) latter is provided by the first item.

\smallskip

For the fourth item, this immediately follows from the first item. 
\end{proof}
We could obtain similar results for most items of Theorem \ref{duck} and related theorems.  
We could also replace `Baire 1' by `effectively Baire 2' in Theorem \ref{ducksss} where the latter means that the function is given 
as the pointwise iterated limit of a double sequence of continuous functions; however, this would require us 
to go up to at least $\ATR_{0}$, as suggested by the results in \cite{dagsamXIV}*{\S2.6}.

\begin{ack}\rm
We thank Anil Nerode and Chris Impens for their valuable advice.
Our research was supported by the \emph{Deutsche Forschungsgemeinschaft} via the DFG grant SA3418/1-1 and the \emph{Klaus Tschira Boost Fund} via the grant Projekt KT43.
Some results were obtained during the stimulating \emph{Workshop on Reverse Mathematics and Philosophy}
in June 2022 in Paris, sponsored by the NSF FRG grant \emph{Collaborative Research: Computability Theoretic Aspects of Combinatorics}.
%
We express our gratitude towards the aforementioned institutions.    
\end{ack}

\bye

\appendix

\section{Old stuff}
\subsection{Preliminaries and definitions}
We briefly introduce \emph{Reverse Mathematics} and \emph{higher-order computability theory} in Section \ref{prelim}.
We introduce some essential definitions and axioms in Section \ref{lll}-\ref{deffer}.  A full introduction may be found in e.g.\ \cite{dagsamX}*{\S2}.
\subsubsection{Reverse Mathematics and higher-order computability theory}\label{prelimk}
Reverse Mathematics (RM hereafter) is a program in the foundations of mathematics initiated around 1975 by Friedman (\cites{fried,fried2}) and developed extensively by Simpson (\cite{simpson2}).  
The aim of RM is to identify the minimal axioms needed to prove theorems of ordinary, i.e.\ non-set theoretical, mathematics. 

\smallskip

We refer to \cite{stillebron} for a basic introduction to RM and to \cite{simpson2, simpson1} for an overview of RM.  We expect basic familiarity with RM, in particular Kohlenbach's \emph{higher-order} RM (\cite{kohlenbach2}) essential to this paper, including the base theory $\RCAo$.   An extensive introduction can be found in e.g.\ \cites{dagsamIII, dagsamV, dagsamX}.  

\smallskip

Next, some of our main results will be proved using techniques from computability theory.
Thus, we first make our notion of `computability' precise as follows.  
\begin{enumerate}
\item[(I)] We adopt $\ZFC$, i.e.\ Zermelo-Fraenkel set theory with the Axiom of Choice, as the official metatheory for all results, unless explicitly stated otherwise.
\item[(II)] We adopt Kleene's notion of \emph{higher-order computation} as given by his nine clauses S1-S9 (see \cite{longmann}*{Ch.\ 5} or \cite{kleeneS1S9}) as our official notion of `computable'.
\end{enumerate}
We refer to \cite{longmann} for a thorough overview of higher-order computability theory.

\smallskip

Finally, it has been suggested to us that Kleene's S9 clause is somewhat mysterious and ad hoc.  
We refer to \cite{dagsamXII} for an equivalent-but-more-elegant $\lambda$-calculus formulation of S1-S9-computability based on fixed point operators.  
This new framework also accommodates partial functionals, a necessity as the central objects of study in \cite{dagsamXII} are \emph{fundamentally partial} in nature. 

%
%

\subsubsection{Some comprehension functionals}\label{lllk}
In second-order RM, the logical hardness of a theorem is measured via what fragment of the comprehension axiom is needed for a proof.  
For this reason, we introduce some axioms and functionals related to \emph{higher-order comprehension} in this section.
We are mostly dealing with \emph{conventional} comprehension here, i.e.\ only parameters over $\N$ and $\N^{\N}$ are allowed in formula classes like $\Pi_{k}^{1}$ and $\Sigma_{k}^{1}$.

\smallskip

First of all, the following functional is clearly discontinuous at $f=11\dots$; in fact, $(\exists^{2})$ is equivalent to the existence of $F:\R\di\R$ such that $F(x)=1$ if $x>_{\R}0$, and $0$ otherwise (\cite{kohlenbach2}*{\S3}).  This fact shall be repeated often.  
\be\label{muk}\tag{$\exists^{2}$}
(\exists \varphi^{2}\leq_{2}1)(\forall f^{1})\big[(\exists n)(f(n)=0) \asa \varphi(f)=0    \big]. 
\ee
Related to $(\exists^{2})$, the functional $\mu^{2}$ in $(\mu^{2})$ is also called \emph{Feferman's $\mu$} (\cite{avi2}).
\begin{align}\label{mu}\tag{$\mu^{2}$}
(\exists \mu^{2})(\forall f^{1})\big[ (\exists n)(f(n)=0) \di [f(\mu(f))=0&\wedge (\forall i<\mu(f))(f(i)\ne 0) ]\\
& \wedge [ (\forall n)(f(n)\ne0)\di   \mu(f)=0]    \big], \notag
\end{align}
We have $(\exists^{2})\asa (\mu^{2})$ over $\RCAo$ and $\ACAo\equiv\RCAo+(\exists^{2})$ proves the same sentences as $\ACA_{0}$ by \cite{hunterphd}*{Theorem~2.5}. 

\smallskip

Secondly, the functional $\SS^{2}$ in $(\SS^{2})$ is called \emph{the Suslin functional} (\cite{kohlenbach2}).
\be\tag{$\SS^{2}$}
(\exists\SS^{2}\leq_{2}1)(\forall f^{1})\big[  (\exists g^{1})(\forall n^{0})(f(\overline{g}n)=0)\asa \SS(f)=0  \big], 
\ee
The system $\FIVE^{\omega}\equiv \RCAo+(\SS^{2})$ proves the same $\Pi_{3}^{1}$-sentences as $\FIVE$ by \cite{yamayamaharehare}*{Theorem 2.2}.   
By definition, the Suslin functional $\SS^{2}$ can decide whether a $\Sigma_{1}^{1}$-formula as in the left-hand side of $(\SS^{2})$ is true or false.   We similarly define the functional $\SS_{k}^{2}$ which decides the truth or falsity of $\Sigma_{k}^{1}$-formulas from $\L_{2}$; we also define 
the system $\SIXK$ as $\RCAo+(\SS_{k}^{2})$, where  $(\SS_{k}^{2})$ expresses that $\SS_{k}^{2}$ exists.  
We note that the operators $\nu_{n}$ from \cite{boekskeopendoen}*{p.\ 129} are essentially $\SS_{n}^{2}$ strengthened to return a witness (if existant) to the $\Sigma_{n}^{1}$-formula at hand.  

\smallskip

\noindent
Thirdly, full second-order arithmetic $\Z_{2}$ is readily derived from $\cup_{k}\SIXK$, or from:
\be\tag{$\exists^{3}$}
(\exists E^{3}\leq_{3}1)(\forall Y^{2})\big[  (\exists f^{1})(Y(f)=0)\asa E(Y)=0  \big], 
\ee
and we therefore define $\Z_{2}^{\Omega}\equiv \RCAo+(\exists^{3})$ and $\Z_{2}^\omega\equiv \cup_{k}\SIXK$, which are conservative over $\Z_{2}$ by \cite{hunterphd}*{Cor.\ 2.6}. 
Despite this close connection, $\Z_{2}^{\omega}$ and $\Z_{2}^{\Omega}$ can behave quite differently, as discussed in e.g.\ \cite{dagsamIII}*{\S2.2}.   
The functional from $(\exists^{3})$ is also called `$\exists^{3}$', and we use the same convention for other functionals.

\subsubsection{Some basic definitions}\label{cdefs}
We introduce some definitions needed in the below, all stemming from mainstream mathematics.

\smallskip

First of all, we shall study the following notions of weak continuity.  
\bdefi\label{flung} For $f:[0,1]\di \R$, we have the following definitions:
\begin{itemize}
\item $f$ is \emph{upper semi-continuous} at $x_{0}\in [0,1]$ if $f(x_{0})\geq_{\R}\lim\sup_{x\di x_{0}} f(x)$,
\item $f$ is \emph{lower semi-continuous} at $x_{0}\in [0,1]$ if $f(x_{0})\leq_{\R}\lim\inf_{x\di x_{0}} f(x)$,
\item $f$ is \emph{quasi-continuous} (resp.\ \emph{cliquish}) at $x_{0}\in [0, 1]$ if for $ \epsilon > 0$ and an open neighbourhood $U$ of $x_{0}$, 
there is a non-empty open ${ G\subset U}$ with $(\forall x\in G) (|f(x_{0})-f(x)|<\eps)$ (resp.\ $(\forall x, y\in G) (|f(x)-f(y)|<\eps)$).
\end{itemize}
\edefi
Secondly, we also study \emph{unordered sums}, which are a device for bestowing meaning upon `uncountable sums' $\sum_{x\in I}f(x)$ for any index set $I$ and $f:I\di \R$.  
A central result is that if $\sum_{x\in I}f(x)$ somehow exists, it must be a `normal' series of the form $\sum_{i\in \N}f(y_{i})$, i.e.\ $f(x)=0$ for all but countably many $x\in [0,1]$; Tao mentions this theorem in \cite{taomes}*{p.~xii}. 

\smallskip

By way of motivation, there is considerable historical and conceptual interest in this topic: Kelley notes in \cite{ooskelly}*{p.\ 64} that E.H.\ Moore's study of unordered sums in \cite{moorelimit2} led to the concept of \emph{net} with his student H.L.\ Smith (\cite{moorsmidje}).
Unordered sums can be found in (self-proclaimed) basic or applied textbooks (\cites{hunterapp,sohrab}) and can be used to develop measure theory (\cite{ooskelly}*{p.\ 79}).  
Moreover, Tukey shows in \cite{tukey1} that topology can be developed using \emph{phalanxes}, which are nets with the same index sets as unordered sums.  

\smallskip

Now, unordered sums are just a special kind of \emph{net} and $a:[0,1]\di \R$ is therefore written $(a_{x})_{x\in [0,1]} $ in this context to suggest the connection to nets.  
We only need the following notions in the below. 
Let $\fin(\R)$ be the set of all finite sequences of reals without repetitions.  
\bdefi\label{kaukiek} Let $(a_{x})_{x\in [0,1]}$ be an unordered sum.
\begin{itemize}
\item We say that $(a_{x})_{x\in [0,1]} $ is \emph{Cauchy} if there is $\Phi:\R\di \fin(\R)$ such that for $\eps>0$ and all $J\in \fin({\R})$ with $J\cap\Phi(\eps)=\emptyset$, we have $|\sum_{x\in J}a_{x}|<\eps$.
\item We say that $(a_{x})_{x\in [0,1]} $ is \emph{bounded} if there is $N_{0}\in \N$ such that for any $J\in \fin(\R)$, $N_{0}>|\sum_{x\in J}a_{x}|$. 
\item We say that $(a_{x})_{x\in [0,1]} $ is \emph{convergent to the limit $a\in \R$} if there is $\Phi:\R\di \fin(\R)$ such that for $\eps>0, I\in \fin({\R})$, we have $|\sum_{x\in \Phi(\eps)}a_{x}|<\eps$.
\end{itemize}
\edefi
Note that in the first item, $\Phi$ is a \emph{Cauchy modulus}, while in the third item, it is a \emph{modulus of convergence}.
For simplicity, we focus on \emph{positive unordered sums}, i.e.\ $(a_{x})_{x\in [0,1]}$ such that $a_{x}\geq 0$ for $x\in [0,1]$.

\smallskip

Thirdly, we also need the following definition.  
\bdefi[Darboux property] Let $f:[0,1]\di \R$ and $x_{0}\in [0,1]$ be given. 
\begin{itemize}
\item A real $y\in \R$ is a left (resp.\ right) \emph{cluster value} of $f$ is there is $(x_{n})_{n\in \N}$ such that $y=\lim_{n\di \infty} f(x_{n})$ and $x=\lim_{n\di \infty}f(x_{n})$ and $(\forall n\in \N)(x_{n}\leq x)$ (resp.\ $(\forall n\in \N)(x_{n}\geq x)$).  
\item A point $x\in [0,1]$ is a \emph{Darboux point} of $f:[0,1]\di \R$ if for any $\delta>0$ and any left (resp.\ right) cluster value $y$ of $f$ and $z\in \R$ strictly between $y$ and $f(x)$, there is $w\in (x-\delta, x)$ (resp.\ $w\in ( x, x+\delta)$) such that $f(w)=y$.   
\end{itemize}
\edefi
By definition, a point of continuity is also a Darboux point, but not vice versa.

\smallskip

Fourth, we introduce our particular notion of `finite set' and provide some motivation in Remark \ref{diunk} right below.   On a historical note, the study of various definitions of finite set (in set theory) was the topic of Mostowski's dissertation, as suggested by Tarski (\cite{moserover}*{p.\ 18-19}). 
\begin{defi}[Finite]\label{deaddk}\rm
Any $X\subset \R$ is \emph{finite} if there is $N\in \N$ such that for any finite sequence $(x_{0}, \dots, x_{N})$ of distinct reals, there is $i\leq N$ such that $x_{i}\not \in X$.
\edefi
Note that the previous definition is not circular as `finite sequences or reals' are just objects of type $1$, modulo coding using $\exists^{2}$.
We now motivate Definition \ref{deadd}.
\begin{rem}[Finite sets by any other name]\label{diunkk}\rm
First of all, working in set theory, the various definitions\footnote{In $\ZF$, a set $A$ is `finite' if there is a bijection to $\{0, 1, \dots, n\}$ for some $n\in \N$; a set $A$ is `Dedekind finite' if any injective mapping from $A$ to $ A$ is also surjective.\label{krukk}} of `finite set' are not equivalent over $\ZF$, while countable choice suffices to establish the equivalence (\cite{jechp}).  Hence, it should not be a surprise that studying finite sets in weak systems requires one to choose a specific definition.  


\smallskip

Secondly, consider the following set where $f$ is a function of bounded variation:
\be\label{lagel2k}\textstyle
A_{n}:=\big\{x\in (0,1): |f(x+)- f(x)|>\frac1{2^{n}} \vee |f(x-)- f(x)|>\frac1{2^{n}}\big\} 
\ee
This set is finite as each element of $A_{n}$ contributes at least $\frac{1}{2^{n}}$ to the total variation of $f$.  
Finite as $A_{n}$ may be, we are unable to exhibit an injection (let alone bijection) from $A_{n}$ to $\{0,1, \dots, k\}$ for some $k\in \N$, say over $\Z_{2}^{\omega}+\QFAC^{0,1}$.  
A similar argument goes through in case $f$ is regular. 
By contrast, $A_{n}$ is trivially finite in the sense of Definition \ref{deadd} in $\ACAo$.  
Similarly, consider a set $A\subset \R$ without limit points.  
While $A\cap [-n, n]$ is finite (Definition \ref{deadd}) for any $n\in \N$ in $\ACAo+\QFAC^{0,1}$,  we are unable to exhibit an injection from $A\cap [-n,n]$ to $\{0,1, \dots, k\}$ for some $k\in \N$, say over $\Z_{2}^{\omega}+\QFAC^{0,1}$.

\smallskip

%
In conclusion, \emph{if} one wants to work in a weak logical system, \emph{then} (certain) finite sets that `appear in the wild' are best studied via Definition \ref{deadd}, and not the definition from Footnote~\ref{krukk} involving bijections or injections.  
\end{rem}
The conclusion of Remark \ref{diunk} of course holds for countable sets as well: studying the latter is best done as in Definition \ref{hoogzalieleven}, as the latter behaves better in weak systems, in contrast to Definition \ref{standard}, which is more `set theoretical' in nature. 
\bdefi[Enumerable sets of reals]\label{enik}
A set $A\subset \R$ is \emph{enumerable} if there exists a sequence $(x_{n})_{n\in \N}$ such that $(\forall x\in \R)(x\in A\di (\exists n\in \N)(x=_{\R}x_{n}))$.  
\edefi
This definition reflects the RM-notion of `countable set' from \cite{simpson2}*{V.4.2}.  
We note that given $\mu^{2}$ from Section \ref{lll}, we may replace the final implication in Definition \ref{eni} by an equivalence. 
\bdefi[Countable subset of $\R$]\label{standardk}~
A set $A\subset \R$ is \emph{countable} if there exists $Y:\R\di \N$ such that $(\forall x, y\in A)(Y(x)=_{0}Y(y)\di x=_{\R}y)$. 
If $Y:\R\di \N$ is also \emph{surjective}, i.e.\ $(\forall n\in \N)(\exists x\in A)(Y(x)=n)$, we call $A$ \emph{strongly countable}.
\edefi
The first part of Definition \ref{standard} is from Kunen's set theory textbook (\cite{kunen}*{p.~63}) and the second part is taken from Hrbacek-Jech's set theory textbook \cite{hrbacekjech} (where the term `countable' is used instead of `strongly countable').  
For the rest of this paper, `strongly countable' and `countable' shall exclusively refer to Definition \ref{standard}, \emph{except when explicitly stated otherwise}. 
Remark \ref{diunk} suggest the following definition.
\begin{defi}\label{hoogzalielevenk}
A set $A\subset \R$ is \emph{height countable} if there is a \emph{height} $H:\R\di \N$ for $A$, i.e.\ for all $n\in \N$, $A_{n}:= \{ x\in A: H(x)<n\}$ is finite \(Definition \ref{deadd}\). 
\end{defi}
We note that the notion of `height' is mentioned in e.g.\ \cite{demol}*{p.\ 33} and \cite{vadsiger}.  A set is therefore height countable iff it is the union over $\N$ of finite sets.  

\subsubsection{Some advanced definitions: bounded variation and around}\label{deffer}
We formulate the definitions of bounded variation and regulated functions, as well as some background. 

\smallskip

Firstly, the notion of \emph{bounded variation} (often abbreviated $BV$ below) was first explicitly\footnote{Lakatos in \cite{laktose}*{p.\ 148} claims that Jordan did not invent or introduce the notion of bounded variation in \cite{jordel}, but rather discovered it in Dirichlet's 1829 paper \cite{didi3}.} introduced by Jordan around 1881 (\cite{jordel}) yielding a generalisation of Dirichlet's convergence theorems for Fourier series.  
Indeed, Dirichlet's convergence results are restricted to functions that are continuous except at a finite number of points, while $BV$-functions can have infinitely many points of discontinuity, as already studied by Jordan, namely in \cite{jordel}*{p.\ 230}.
Nowadays, the \emph{total variation} of a function $f:[a, b]\di \R$ is defined as follows:
\be\label{tombk}\textstyle
V_{a}^{b}(f):=\sup_{a\leq x_{0}< \dots< x_{n}\leq b}\sum_{i=0}^{n} |f(x_{i})-f(x_{i+1})|.
\ee
If this quantity exists and is finite, one says that $f$ has bounded variation on $[a,b]$.
Now, the notion of bounded variation is defined in \cite{nieyo} \emph{without} mentioning the supremum in \eqref{tomb}; this approach can also be found in \cites{kreupel, briva, brima}.  
Hence, we shall distinguish between the two notions in Definition \ref{varvar}.  As it happens, Jordan seems to use item \eqref{donp} of Definition \ref{varvar} in \cite{jordel}*{p.\ 228-229}.
This definition suggests a two-fold variation for any result on functions of bounded variation, namely depending on whether the supremum \eqref{tomb} is given, or only an upper bound on the latter.  
\bdefi[Variations on variation]\label{varvar}
\begin{enumerate}  
\renewcommand{\theenumi}{\alph{enumi}}
\item The function $f:[a,b]\di \R$ \emph{has bounded variation} on $[a,b]$ if there is $k_{0}\in \N$ such that $k_{0}\geq \sum_{i=0}^{n} |f(x_{i})-f(x_{i+1})|$ 
for any partition $x_{0}=a <x_{1}< \dots< x_{n-1}<x_{n}=b  $.\label{donp}
\item The function $f:[a,b]\di \R$ \emph{has {a} variation} on $[a,b]$ if the supremum in \eqref{tomb} exists and is finite.\label{donp2}
\end{enumerate}
\edefi
Secondly, the fundamental theorem about $BV$-functions is formulated as follows.
\begin{thm}[Jordan decomposition theorem, \cite{jordel}*{p.\ 229}]\label{drd}
A $BV$-function $f : [0, 1] \di \R$ is the difference of  two non-decreasing functions $g, h:[0,1]\di \R$.
\end{thm}
Theorem \ref{drd} has been studied via second-order representations in \cites{groeneberg, kreupel, nieyo, verzengend}.
The same holds for constructive analysis by \cites{briva, varijo,brima, baathetniet}, involving different (but related) constructive enrichments.  
Now, $\ACA_{0}$ suffices to derive Theorem \ref{drd} for various kinds of second-order \emph{representations} of $BV$-functions in \cite{kreupel, nieyo}.  
By contrast, our results in \cite{dagsamXI} imply that $\Z_{2}^{\omega}+\QFAC^{0,1}$ cannot prove the third-order version of Theorem \ref{drd}.  

\smallskip

Thirdly, Jordan proves in \cite{jordel3}*{\S105} that $BV$-functions are exactly those for which the notion of `length of the graph of the function' makes sense.  In particular, $f\in BV$ if and only if the `length of the graph of $f$', defined as follows:
\be\label{puhe}\textstyle
L(f, [0,1]):=\sup_{0=t_{0}<t_{1}<\dots <t_{m}=1} \sum_{i=0}^{m-1} \sqrt{(t_{i}-t_{i+1})^{2}+(f(t_{i})-f(t_{i+1}))^{2}  }
\ee
exists and is finite by \cite{voordedorst}*{Thm.\ 3.28.(c)}.  In case the supremum in \eqref{puhe} exists (and is finite), $f$ is also called \emph{rectifiable}.  
Rectifiable curves predate $BV$-functions: in \cite{scheeffer}*{\S1-2}, it is claimed that \eqref{puhe} is essentially equivalent to Duhamel's 1866 approach from \cite{duhamel}*{Ch.\ VI}.  Around 1833, Dirksen, the PhD supervisor of Jacobi and Heine, already provides a definition of arc length that is (very) similar to \eqref{puhe} (see \cite{dirksen}*{\S2, p.\ 128}), but with some conceptual problems as discussed in \cite{coolitman}*{\S3}.

\smallskip

Fourth, a function is \emph{regulated} (called `regular' in \cite{voordedorst}) if for every $x_{0}$ in the domain, the `left' and `right' limit $f(x_{0}-)=\lim_{x\di x_{0}-}f(x)$ and $f(x_{0}+)=\lim_{x\di x_{0}+}f(x)$ exist.  
Scheeffer studies discontinuous regulated functions in \cite{scheeffer} (without using the term `regulated'), while Bourbaki develops Riemann integration based on regulated functions in \cite{boerbakies}.  
Now, $BV$-functions are regulated (see Theorem \ref{flima}), while Weierstrass' `monster' function is a natural example of a regulated function not in $BV$.  
An interesting observation about regular functions and continuity is as follows.
\begin{rem}[Continuity and the Axiom of Choice]\label{atleast}\rm
As discussed in \cite{kohlenbach2}*{\S3}, the \emph{local} equivalence for functions on Baire space between sequential and `epsilon-delta' continuity can be proved in $\RCAo+\QFAC^{0,1}$, but not in $\ZF$.  
By \cite{dagsamXI}*{Theorem 3.32}, this equivalence for \emph{regulated} functions is provable in $\ACAo$.
\end{rem}
%
Finally, the {Jordan decomposition theorem} as in Theorem \ref{drd} shows that a $BV$-function can be `decomposed' as the difference of monotone functions. 
This is however not the only result of its kind: Sierpi\'{n}ski e.g.\ establishes in \cite{voordesier} that for regulated $f:[0,1]\di \R$, there are $g, h$ such that $f=g\circ h$ with $g$ continuous and $h$ strictly increasing on their respective domains.

\section{Computability theory}
\subsection{Introduction: computational properties of the uncountability of $\R$}
The uncountability of $\R$ can be studied in numerous guises in higher-order computability theory.  
For instance, the following notions were introduced in \cite{dagsamX, dagsamXII}, where it is also shown that many extremely basic operations compute these realisers.  
\bdefi[Realisers for the uncountability of $\R$]\label{kefi}~
\begin{itemize}
\item A \emph{Cantor functional/realiser} takes as input $A\subset [0,1]$ and $Y:[0,1]\di \N$ such that $Y$ is injective on $A$, and outputs $x\not \in A$.  
\item A \emph{\textbf{weak} Cantor realiser} takes as input $A\subset [0,1]$ and $Y:[0,1]\di \N$ such that $Y$ is \textbf{bijective} on $A$, and outputs $x\not \in A$.  
\item A $\NIN$-\emph{realiser} takes as input $Y:[0,1]\di \N$ and outputs $x,y\in [0,1]$ with $x\ne y \wedge Y(x)=Y(y)$.  
\end{itemize}
\edefi
Clearly, the realisers in Definition \ref{kefi} are based on the `set theoretic' definition of countability.  As discussed in Remark \ref{diunk}, the notion of weak countability (Definition \ref{hoogzalieleven}) seems to be have better, suggesting 
the following stronger notions. 
\begin{defi}[A strong Cantor realiser]
A \emph{strong Cantor realiser} is any functional that on input a sequence $(X_{n})_{n\in \N}$ of finite sets in $[0,1]$, outputs $y\in [0,1]\setminus \cup_{n\in \N}X_{n}$. 
\end{defi}
The following definition is motivated by the observation that \eqref{lagel2} in Remark~\ref{diunk} comes with an explicit upper bound (namely $2^{n}V_{0}^{1}(f)$), while for regulated functions, all we can say is that \eqref{lagel2} is finite (Def.\ \ref{deadd}).
\begin{defi}[An intermediate Cantor realiser]
An \emph{intermediate Cantor realiser} is any functional that on input a sequence $(X_{n})_{n\in \N}$ of finite sets in $[0,1]$ and $g:\N\di \N$ such that $(\forall n\in \N)( |X_{n}|\leq g(n))$, outputs $y\in [0,1]\setminus \cup_{n\in \N}X_{n}$. 
\end{defi}
The reader will observe that an intemediate Cantor realiser computes a normal one.  Any Cantor realiser expresses the extremely basic idea that removing a countable set from an uncountable one, there is at least one element left.  

\smallskip

We show in Section \ref{birf} that many operations on regulated functions are computationally equivalent to strong Cantor realisers.  
We show in Section \ref{birf2} that many operations on $BV$-functions are computationally equivalent to intermediate Cantor realisers.  
We discuss related functionals and formulate some conjectures.  

\subsection{Strong Cantor realisers}\label{birf}
In this section, we show that many operations on regulated functions are computationally equivalent to strong Cantor realisers.  

\smallskip

First of all, various natural functionals compute a strong Cantor realiser.
\begin{thm}\label{kifkif}
The following functionals compute a strong Cantor realiser:
\begin{itemize}
\item the structure functional $\Omega$ from \cite{dagsamXIII} that enumerates finite sets,
\item the non-monotone induction functional \(see \cite{dagsamVII, dagcomp20}\),
\item a realiser for the Baire category theorem \(see \cite{dagsamVII}*{\S6}\).
\end{itemize}
\end{thm}
\begin{proof}
Let $(X_{n})_{n\in \N}$ be a sequence of finite sets in $[0,1]$.  
Since $\Omega$ can enumerate finite sets, the sequence $\Omega(X_{0})*\Omega(X_{1})*\dots$ enumerates $\cup_{n\in \N}X_{n}$.  
One readily computes $y\in [0,1]$ not in this sequence (see e.g.\ \cite{simpson2}*{II.4.9} or \cite{grayk}), yielding a strong Cantor realiser.  

\smallskip

The second item computes the third item by \cite{dagsamVII}*{Theorem 6.5}.  Regarding the third item, $O_{n}:= [0,1]\setminus \cup_{k\leq n}X_{k}$ is (trivially) dense and open.  
Clearly any $y \in \cap_{n\in \N}O_{n}$ by definition satisfies $y\in [0,1]\setminus \cup_{n\in \N}X_{n}$, i.e.\ a realiser for the Baire category theorem computes a strong Cantor realiser. 
\end{proof}
%

Secondly, we have the following theorem. 
Regarding item \eqref{full6}, the property `$x$ is a local strict maximum of $f$' does not seem to be decidable given $\exists^{2}$ for (general) regulated $f$, as opposed to `$f$ is continuous at $x$', which is equivalent to $f(x+)\ne f(x-)$.
We also note that items \eqref{full2} and \eqref{full25} are based on Volterra's Theorem \ref{VOL} and Corollary \ref{VOLcor}.
\begin{thm}\label{flunk}
The following are computationally equivalent modulo $\exists^{2}$.
\begin{enumerate}
\renewcommand{\theenumi}{\alph{enumi}}
\item A functional that on input regulated $f:[0,1]\di \R$, outputs $a, b\in  [0,1]$ such that $\{ x\in [0,1]:f(a)\leq f(x)\leq f(b)\}$ is infinite.\label{full0}
\item A functional that on input regulated $f:[0,1]\di \R$, outputs $y\in  [0,1]$ where $f$ is continuous \(or {quasi}- or lower semi-continuous\).\label{full1}
\item A `density' functional that on input regulated $f:[0,1]\di \R$, $x\in [0,1]$, and $k\in \N$, outputs $y\in  [0,1]\cap B(x, \frac{1}{2^{k}})$ such that $f$ is continuous at $y$. \label{full1.5}
\item A functional that on input regulated $f:[0,1]\di \R$, outputs either $q\in \Q\cap [0,1]$ where $f$ is discontinuous, or $x\in [0,1]\setminus \Q$ where $f$ is continuous. \label{full2}
\item A functional that on input regulated $f,g:[0,1]\di \R$, outputs a real $x\in [0,1]$ such that $f$ and $g$ are both continuous or discontinuous at $x$. \label{full25}
\item A strong Cantor realiser.\label{full3}
\item A functional that on input regulated $f:[0,1]\di [0,1]$ with Riemann integral $\int_{0}^{1}f(x)dx=0$, outputs $x\in [0,1]$ with $f(x)=0$. \label{full4}
\item A functional that on input regulated $f:[0,1]\di \R$, outputs $x_{0}\in (0,1)$ where $F(x):=\int_{0}^{x}f(t)dt$ is differentiable with derivative $f(x_{0})$.\label{full5}
\item A functional that on input regulated $f:[0,1]\di \R$ with only removable discontinuities, outputs $x\in [0,1]$ which is \textbf{not} a strict\footnote{A point $x\in [0,1]$ is a strict local maximum of $f:[0,1]\di \R$ in case $(\exists N\in \N)( \forall y \in B(x, \frac{1}{2^{N}}))(y\ne x\di f(y)<f(x))$.} local maximum. \label{full6}
\end{enumerate}
\end{thm}
\begin{proof}
First of all, regarding item \eqref{full0}, consider the regulated function $h$ from \eqref{modi3}. 
In case $h(a)>0$, then $\{ x\in [0,1]:h(a)\leq h(x)\leq h(b)\}$ is finite by the definition of $h$.  In case $h(a)=0$, \eqref{modi3} implies $a\in [0,1]\setminus \cup_{n\in \N}X_{n}$, as required for a strong Cantor realiser. 
Now consider item \eqref{full1} and note that if $f:[0,1]\di \R$ is continuous at $y\in [0,1]$, we can use $\mu^{2}$ to find $N\in \N$ such that $(\forall q\in B(y, \frac{1}{2^{N}})\cap \Q\cap [0,1]  )(|f(y)-f(q)|<\frac{1}{2}) $, which readily 
yields $a, b\in [0,1]$ as required by item \eqref{full0}.
Below, we prove that item \eqref{full3} implies item \eqref{full1}.

\smallskip

Secondly, The proof of the theorem, in particular that $\QRQ_{\reg}$ implies item \eqref{on6} of the theorem, yields the implication \eqref{full2}$\di$\eqref{full1} in the corollary. 
Similarly, the proof of the theorem, in particular that item \eqref{on6} of the theorem implies $\NIN_{\alt}$, yields the implication \eqref{full1}$\di$\eqref{full3} in the corollary. 
The same proof goes through for continuity replaced by quasi-continuity or regulated. 

\smallskip

Thirdly, item \eqref{full1.5} immediately yields item \eqref{full1}, while the reversal readily follows by rescaling. 

\smallskip

Fourth, assume item \eqref{full3} from the corollary is given and fix regulated $f:[0,1]\di \R$.   Note that $\mu^{2}$ can find $q\in [0,1]\cap \Q$ with $f(q+)\ne f(q-)$, if such rational exists.   
In case $f$ is continuous at all $q\in [0,1]\cap \Q$, consider the finite set $X_{n}$ from \eqref{sameold} and let $(q_{m})_{m\in \N}$ be an enumeration of $\Q\cap [0,1]$. 
Now use the strong Cantor realiser to find $y\in [0,1]\setminus \big(   \cup_{n\in \N} Y_{n}  \big)$, where $Y_{n}=X_{n}\cup \{q_{n}\}$.
Then $f$ is continuous at $y \in [0,1]\setminus \Q$ and item \eqref{full2} from the corollary follows. 

\smallskip

Fifth, assume item \eqref{full4}, let $(X_{n})_{n\in \N}$ be a sequence of finite sets in $[0,1]$, and let $h$ be as in \eqref{modi3}. 
Since $h$ is regulated, it is Riemann integrable, and of course has integral $0$.  Any $y\in [0,1]$ such that $h(y)=0$ is by definition not in $\cup_{n\in \N}X_{n}$, yielding a strong Cantor realiser. 
Item \eqref{full1} readily yields item \eqref{full4} as follows: \emph{if} regulated $f:[0,1]\di [0,1]$ has $\int_{0}^{1}f(x)dx=0$ and is continuous at $y\in [0,1]$, \emph{then} we must have $f(y)=0$;
indeed, in case $f(y)>0$, the continuity of $f$ at $y$ implies that $f(z)>\frac{1}{2^{k}}$ for $z\in B(y,\frac{1}{2^{k}} )$ for some $k\in \N$, implying $\int_{0}^{1}f(x)dx > 0$.

\smallskip

Sixth, the equivalence between items \eqref{full1} and \eqref{full5} is immediate in light of the (first part of the) fundamental theorem of calculus. 

\smallskip

Seventh, consider item \eqref{full6} and note that $h$ as in \eqref{modi3} is regulated with only removable discontinuities.  
Now, the set $X=\cup_{n\in \N}X_{n}$ consists of the local strict maxima of $h$, i.e.\ item \eqref{full6} yields a strong Cantor realiser. 
For the reversal, $\exists^{2}$ computes a functional $M$ such that $M(g, a, b)$ is a maximum of $g\in C([0,1])$ on $[a,b]\subset [0,1]$ (see \cite{kohlenbach2}*{\S3}), i.e.\ $(\forall y\in [a,b])(g(y)\leq g(M(g, a,b)))$.
Using the functional $M$, one readily shows that `$x$ is a strict local maximum of $g$' is decidable\footnote{If $g\in C([0,1])$, then $x\in [0,1]$ is a local maximum iff for some (rational) $\epsilon > 0$ we have that 
\begin{itemize}
\item $g(y) < g(x)$ whenever $|x-y] < \epsilon$ for any $q\in [0,1]\cap\Q $, and: 
\item $\sup_{y\in [a,b]}g(y) < g(x)$ whenever $x \not \in [a,b]$, $a,b\in \Q$ and $[a,b] \subset [x - \epsilon,x + \epsilon]$.
\end{itemize}.
Note that $\mu^{2}$ readily yields $N\in \N$ such that $(\forall y\in B(x, \frac{1}{2^{N}}))( g(y)<g(x))$.\label{roofer}
} given $\exists^{2}$, for $g$ continuous on $[0,1]$.  
Now let $f:[0,1]\di \R$ be regulated and with only removable discontinuities.  
Use $\exists^{2}$ to define $\tilde{f}:[0,1]\di \R$ as follows: $\tilde{f}(x):= f(x+)$ for $x\in [0, 1)$ and $\tilde{f}(1)=f(1-)$.
By definition, $\tilde{f}$ is continuous on $[0,1]$, and $\exists^{2}$ computes a (continuous) modulus of continuity, which follows in the same way as for Baire space (see e.g.\ \cite{kohlenbach4}*{\S4}).
In case $f$ is discontinuous at $x\in [0,1]$, the latter point is a strict local maximum of $f$ if and only if $f(x)>f(x+)$ (or $f(x)>f(x-)$ in case $x=1$).  
Note that $\mu^{2}$ (together with a modulus of continuity for $\tilde{f}$) readily yields $N_{f, x}\in \N$ such that $(\forall y\in B(x, \frac{1}{2^{N_{f,x}}}))( f(y)<f(x))$, in case $x$ is a strict local maximum of $f$.
In case $f$ is continuous at $x\in [0,1]$, we can use $\exists^{2}$ to decide whether $x$ is a local strict maximum of $\tilde{f}$.  
By Footnote \ref{roofer}, $\mu^{2}$ again yields $N_{f, x}\in \N$ such that $(\forall y\in B(x, \frac{1}{2^{N}}))( \tilde{f}(y)<\tilde{f}(x)))$, in case $x$ is a strict local maximum of $\tilde{f}$.
Now consider the following set:
\[\textstyle
A_{n}:=\{x\in [0,1]:  \textup{$x$ is a local strict maximum of $\tilde{f}$ or $f$ with $n\geq N_{{f}, x}\in \N$}\}.
\]
Then $A_{n}$ has at most $2^{n+2}$ elements as strict local maxima cannot be `too close'.  
Hence, a strong Cantor realiser yields $y \in [0,1]\setminus \cup_{n\in \N}A_{n}$, which is not a local maximum of $f$, i.e.\ item \eqref{full6} follows. 

\smallskip

Finally, consider item \eqref{full25} and note that Thomae's function from \eqref{thomae} is regulated as the left and right limits are $0$.  Hence, we obtain item \eqref{full2}, as Thomae's function is continuous on $\R\setminus \Q$ and discontinuous on $\Q$.  To obtain item \eqref{full25} from a strong Cantor realiser, let $f, g:[0,1]$ be regulated and consider the finite set $X_{n}$ as in \eqref{sameold}.
Let $Y_{n}$ be the same set for $g$. Then use a strong Cantor realiser to obtain $y\in [0,1]\setminus \big(\cup_{n\in \N}(X_{n}\cup Z_{n})\big)$.  By definition, $f$ and $g$ are continuous at $y$.
%
\end{proof}
Regarding item \eqref{full6}, one readily shows that the following functional is computationally equivalent to $\Omega$ from Theorem \ref{kifkif}, also for the restriction to $BV$-functions.  
Note that the set of strict local maxima is countable for \emph{any} $\R\di \R$-function by \cite{saks}*{p.\ 261, Theorem (1.1)}.
\begin{itemize}
\item A functional that on input regulated $f:[0,1]\di \R$ with only removable discontinuities, outputs an enumeration of all strict local maxima. 
\end{itemize}
In conclusion, we have established a number of equivalences for strong Cantor realisers, while many more variations are possible. 

\subsection{Intermediate Cantor realisers}\label{birf2}

\section{Reverse Mathematics}

\subsection{Some results}
We establish the results sketched in Section \ref{vintro}

\smallskip

Add the commented out part from Gauld's paper.

\smallskip

While a $BV$-function is differentiable a.e.\ (though a regular function may be nowhere differentiable), we do not seem to get equivalences like for continuity as in Theorem \ref{duck}.

\smallskip

\smallskip

The following may be omitted.
\begin{rem}\rm
Why can we expect regulated functions and properties of finite sets to be connected?  This connection can be explained by \cite{gaffe2}*{p.\ 63}: the authors define the `saltus' function $f^{*}$ as $f^{*}(x):= f(x+)-f(x-)$ for regulated $f:[a,b]\di \R$.  
Then $\{ x\in [a,b] : |f^{*}(x)|>\eps  \}$ is finite for $\eps>0$ \emph{and} this is the essential property of the saltus function of a regulated function, by the following theorem
\begin{center} 
\emph{If $h:[a,b]\di \R$ is such that $\{x \in [a,b]: |h(x)| > 1/2^{n}\}$ is finite for all $n\in \N$ and $h(a) = h(b) = 0$, then there is regulated $f$ such that $f^{*}= h$ \(\cite{gaffe2}*{Theorem 2}\).}
\end{center}
\end{rem}

\appendix

\bye
\appendix

\section{Background}\label{prelimA}
We introduce \emph{Reverse Mathematics} in Section \ref{prelim1}, as well as Kohlenbach's generalisation to \emph{higher-order arithmetic}, and the associated base theory $\RCAo$.  
We introduce higher-order \emph{computability theory}, following Kleene's computation schemes S1-S9, in Section \ref{HCT}. 

\subsection{Reverse Mathematics}\label{prelim1}
We provide an introduction to Reverse Mathematics (Section \ref{introrm}) and introduce Kohlenbach's base theory of \emph{higher-order} Reverse Mathematics (Section \ref{rmbt}).
Some essential axioms, functionals, and notations may be found in Sections \ref{kkk} and \ref{lll}.
\subsubsection{Introduction}\label{introrm}
Reverse Mathematics (RM hereafter) is a program in the foundations of mathematics initiated around 1975 by Friedman (\cites{fried,fried2}) and developed extensively by Simpson (\cite{simpson2}).  
The aim of RM is to identify the minimal axioms needed to prove theorems of ordinary, i.e.\ non-set theoretical, mathematics. 

\smallskip

We refer to \cite{stillebron} for a basic introduction to RM and to \cite{simpson2, simpson1} for an overview of RM.  We expect basic familiarity with RM, but do sketch some aspects of Kohlenbach's \emph{higher-order} RM (\cite{kohlenbach2}) essential to this paper, including the base theory $\RCAo$ (Definition \ref{kase}).  

\smallskip

First of all, in contrast to `classical' RM based on \emph{second-order arithmetic} $\Z_{2}$, higher-order RM uses $\L_{\omega}$, the richer language of \emph{higher-order arithmetic}.  
Indeed, while the former is restricted to natural numbers and sets of natural numbers, higher-order arithmetic can accommodate sets of sets of natural numbers, sets of sets of sets of natural numbers, et cetera.  
To formalise this idea, we introduce the collection of \emph{all finite types} $\mathbf{T}$, defined by the two clauses:
\begin{center}
(i) $0\in \mathbf{T}$   and   (ii)  If $\sigma, \tau\in \mathbf{T}$ then $( \sigma \di \tau) \in \mathbf{T}$,
\end{center}
where $0$ is the type of natural numbers, and $\sigma\di \tau$ is the type of mappings from objects of type $\sigma$ to objects of type $\tau$.
In this way, $1\equiv 0\di 0$ is the type of functions from numbers to numbers, and  $n+1\equiv n\di 0$.  Viewing sets as given by characteristic functions, we note that $\Z_{2}$ only includes objects of type $0$ and $1$.    

\smallskip

Secondly, the language $\L_{\omega}$ includes variables $x^{\rho}, y^{\rho}, z^{\rho},\dots$ of any finite type $\rho\in \mathbf{T}$.  Types may be omitted when they can be inferred from context.  
The constants of $\L_{\omega}$ include the type $0$ objects $0, 1$ and $ <_{0}, +_{0}, \times_{0},=_{0}$  which are intended to have their usual meaning as operations on $\N$.
Equality at higher types is defined in terms of `$=_{0}$' as follows: for any objects $x^{\tau}, y^{\tau}$, we have
\be\label{aparth}
[x=_{\tau}y] \equiv (\forall z_{1}^{\tau_{1}}\dots z_{k}^{\tau_{k}})[xz_{1}\dots z_{k}=_{0}yz_{1}\dots z_{k}],
\ee
if the type $\tau$ is composed as $\tau\equiv(\tau_{1}\di \dots\di \tau_{k}\di 0)$.  
Furthermore, $\L_{\omega}$ also includes the \emph{recursor constant} $\mathbf{R}_{\sigma}$ for any $\sigma\in \mathbf{T}$, which allows for iteration on type $\sigma$-objects as in the special case \eqref{special}.  Formulas and terms are defined as usual.  
One obtains the sub-language $\L_{n+2}$ by restricting the above type formation rule to produce only type $n+1$ objects (and related types of similar complexity).        

\subsubsection{The base theory of higher-order Reverse Mathematics}\label{rmbt}
We introduce Kohlenbach's base theory $\RCAo$, first introduced in \cite{kohlenbach2}*{\S2}.
\bdefi\label{kase} 
The base theory $\RCAo$ consists of the following axioms.
\begin{enumerate}
 \renewcommand{\theenumi}{\alph{enumi}}
\item  Basic axioms expressing that $0, 1, <_{0}, +_{0}, \times_{0}$ form an ordered semi-ring with equality $=_{0}$.
\item Basic axioms defining the well-known $\Pi$ and $\Sigma$ combinators (aka $K$ and $S$ in \cite{avi2}), which allow for the definition of \emph{$\lambda$-abstraction}. 
\item The defining axiom of the recursor constant $\mathbf{R}_{0}$: for $m^{0}$ and $f^{1}$: 
\be\label{special}
\mathbf{R}_{0}(f, m, 0):= m \textup{ and } \mathbf{R}_{0}(f, m, n+1):= f(n, \mathbf{R}_{0}(f, m, n)).
\ee
\item The \emph{axiom of extensionality}: for all $\rho, \tau\in \mathbf{T}$, we have:
\be\label{EXT}\tag{$\textsf{\textup{E}}_{\rho, \tau}$}  
(\forall  x^{\rho},y^{\rho}, \varphi^{\rho\di \tau}) \big[x=_{\rho} y \di \varphi(x)=_{\tau}\varphi(y)   \big].
\ee 
\item The induction axiom for quantifier-free formulas of $\L_{\omega}$.
\item $\QFAC^{1,0}$: the quantifier-free Axiom of Choice as in Definition \ref{QFAC}.
\end{enumerate}
\edefi
\noindent
Note that variables (of any finite type) are allowed in quantifier-free formulas of the language $\L_{\omega}$: only quantifiers are banned.
Recursion as in \eqref{special} is called \emph{primitive recursion}; the class of functionals obtained from $\mathbf{R}_{\rho}$ for all $\rho \in \mathbf{T}$ is called \emph{G\"odel's system $T$} of all (higher-order) primitive recursive functionals. 
\bdefi\label{QFAC} The axiom $\QFAC$ consists of the following for all $\sigma, \tau \in \textbf{T}$:
\be\tag{$\QFAC^{\sigma,\tau}$}
(\forall x^{\sigma})(\exists y^{\tau})A(x, y)\di (\exists Y^{\sigma\di \tau})(\forall x^{\sigma})A(x, Y(x)),
\ee
for any quantifier-free formula $A$ in the language of $\L_{\omega}$.
\edefi
As discussed in \cite{kohlenbach2}*{\S2}, $\RCAo$ and $\RCA_{0}$ prove the same sentences `up to language' as the latter is set-based and the former function-based.   
This conservation results is obtained via the so-called $\ECF$-interpretation, which we now discuss. 
\begin{rem}[The $\ECF$-interpretation]\label{ECF}\rm
The (rather) technical definition of $\ECF$ may be found in \cite{troelstra1}*{p.\ 138, \S2.6}.
Intuitively, the $\ECF$-interpretation $[A]_{\ECF}$ of a formula $A\in \L_{\omega}$ is just $A$ with all variables 
of type two and higher replaced by type one variables ranging over so-called `associates' or `RM-codes' (see \cite{kohlenbach4}*{\S4}); the latter are (countable) representations of continuous functionals.  
The $\ECF$-interpretation connects $\RCAo$ and $\RCA_{0}$ (see \cite{kohlenbach2}*{Prop.\ 3.1}) in that if $\RCAo$ proves $A$, then $\RCA_{0}$ proves $[A]_{\ECF}$, again `up to language', as $\RCA_{0}$ is 
formulated using sets, and $[A]_{\ECF}$ is formulated using types, i.e.\ using type zero and one objects.  
\end{rem}
In light of the widespread use of codes in RM and the common practise of identifying codes with the objects being coded, it is no exaggeration to refer to $\ECF$ as the \emph{canonical} embedding of higher-order into second-order arithmetic. 

%

\subsubsection{Notations and the like}\label{kkk}
We introduce the usual notations for common mathematical notions, like real numbers, as also introduced in \cite{kohlenbach2}.  
\begin{defi}[Real numbers and related notions in $\RCAo$]\label{keepintireal}\rm~
\begin{enumerate}
 \renewcommand{\theenumi}{\alph{enumi}}
\item Natural numbers correspond to type zero objects, and we use `$n^{0}$' and `$n\in \N$' interchangeably.  Rational numbers are defined as signed quotients of natural numbers, and `$q\in \Q$' and `$<_{\Q}$' have their usual meaning.    
\item Real numbers are coded by fast-converging Cauchy sequences $q_{(\cdot)}:\N\di \Q$, i.e.\  such that $(\forall n^{0}, i^{0})(|q_{n}-q_{n+i}|<_{\Q} \frac{1}{2^{n}})$.  
We use Kohlenbach's `hat function' from \cite{kohlenbach2}*{p.\ 289} to guarantee that every $q^{1}$ defines a real number.  
\item We write `$x\in \R$' to express that $x^{1}:=(q^{1}_{(\cdot)})$ represents a real as in the previous item and write $[x](k):=q_{k}$ for the $k$-th approximation of $x$.    
\item Two reals $x, y$ represented by $q_{(\cdot)}$ and $r_{(\cdot)}$ are \emph{equal}, denoted $x=_{\R}y$, if $(\forall n^{0})(|q_{n}-r_{n}|\leq {2^{-n+1}})$. Inequality `$<_{\R}$' is defined similarly.  
We sometimes omit the subscript `$\R$' if it is clear from context.           
\item Functions $F:\R\di \R$ are represented by $\Phi^{1\di 1}$ mapping equal reals to equal reals, i.e.\ extensionality as in $(\forall x , y\in \R)(x=_{\R}y\di \Phi(x)=_{\R}\Phi(y))$.\label{EXTEN}
\item The relation `$x\leq_{\tau}y$' is defined as in \eqref{aparth} but with `$\leq_{0}$' instead of `$=_{0}$'.  Binary sequences are denoted `$f^{1}, g^{1}\leq_{1}1$', but also `$f,g\in C$' or `$f, g\in 2^{\N}$'.  Elements of Baire space are given by $f^{1}, g^{1}$, but also denoted `$f, g\in \N^{\N}$'.
\item For a binary sequence $f^{1}$, the associated real in $[0,1]$ is $\r(f):=\sum_{n=0}^{\infty}\frac{f(n)}{2^{n+1}}$.\label{detrippe}
\item Sets of type $\rho$ objects $X^{\rho\di 0}, Y^{\rho\di 0}, \dots$ are given by their characteristic functions $F^{\rho\di 0}_{X}\leq_{\rho\di 0}1$, i.e.\ we write `$x\in X$' for $ F_{X}(x)=_{0}1$. \label{koer} 
\end{enumerate}
\end{defi}
For completeness, we list the following notational convention for finite sequences.  
\begin{nota}[Finite sequences]\label{skim}\rm
The type for `finite sequences of objects of type $\rho$' is denoted $\rho^{*}$, which we shall only use for $\rho=0,1$.  
Since the usual coding of pairs of numbers goes through in $\RCAo$, we shall not always distinguish between $0$ and $0^{*}$. 
Similarly, we assume a fixed coding for finite sequences of type $1$ and shall make use of the type `$1^{*}$'.  
In general, we do not always distinguish between `$s^{\rho}$' and `$\langle s^{\rho}\rangle$', where the former is `the object $s$ of type $\rho$', and the latter is `the sequence of type $\rho^{*}$ with only element $s^{\rho}$'.  The empty sequence for the type $\rho^{*}$ is denoted by `$\langle \rangle_{\rho}$', usually with the typing omitted.  

\smallskip

Furthermore, we denote by `$|s|=n$' the length of the finite sequence $s^{\rho^{*}}=\langle s_{0}^{\rho},s_{1}^{\rho},\dots,s_{n-1}^{\rho}\rangle$, where $|\langle\rangle|=0$, i.e.\ the empty sequence has length zero.
For sequences $s^{\rho^{*}}, t^{\rho^{*}}$, we denote by `$s*t$' the concatenation of $s$ and $t$, i.e.\ $(s*t)(i)=s(i)$ for $i<|s|$ and $(s*t)(j)=t(|s|-j)$ for $|s|\leq j< |s|+|t|$. For a sequence $s^{\rho^{*}}$, we define $\overline{s}N:=\langle s(0), s(1), \dots,  s(N-1)\rangle $ for $N^{0}<|s|$.  
For a sequence $\alpha^{0\di \rho}$, we also write $\overline{\alpha}N=\langle \alpha(0), \alpha(1),\dots, \alpha(N-1)\rangle$ for \emph{any} $N^{0}$.  By way of shorthand, 
$(\forall q^{\rho}\in Q^{\rho^{*}})A(q)$ abbreviates $(\forall i^{0}<|Q|)A(Q(i))$, which is (equivalent to) quantifier-free if $A$ is.   
\end{nota}
\subsubsection{Some axioms and functionals}\label{lll}
As noted in Section \ref{intro}, the logical hardness of a theorem is measured via what fragment of the comprehension axiom is needed for a proof.  
For this reason, we introduce some axioms and functionals related to \emph{higher-order comprehension} in this section.

\smallskip

First of all, the following functional is clearly discontinuous at $f=11\dots$; in fact, $(\exists^{2})$ is equivalent to the existence of $F:\R\di\R$ such that $F(x)=1$ if $x>_{\R}0$, and $0$ otherwise (\cite{kohlenbach2}*{\S3}).  This fact shall be repeated often.  
\be\label{muk}\tag{$\exists^{2}$}
(\exists \varphi^{2}\leq_{2}1)(\forall f^{1})\big[(\exists n)(f(n)=0) \asa \varphi(f)=0    \big]. 
\ee
Related to $(\exists^{2})$, the functional $\mu^{2}$ in $(\mu^{2})$ is also called \emph{Feferman's $\mu$} (\cite{avi2}).
\begin{align}\label{mu}\tag{$\mu^{2}$}
(\exists \mu^{2})(\forall f^{1})\big[ (\exists n)(f(n)=0) \di [f(\mu(f))=0&\wedge (\forall i<\mu(f))(f(i)\ne 0) ]\\
& \wedge [ (\forall n)(f(n)\ne0)\di   \mu(f)=0]    \big], \notag
\end{align}
We have $(\exists^{2})\asa (\mu^{2})$ over $\RCAo$ and $\ACAo\equiv\RCAo+(\exists^{2})$ proves the same sentences as $\ACA_{0}$ by \cite{hunterphd}*{Theorem~2.5}. 

\smallskip

Secondly, the functional $\SS^{2}$ in $(\SS^{2})$ is called \emph{the Suslin functional} (\cite{kohlenbach2}).
\be\tag{$\SS^{2}$}
(\exists\SS^{2}\leq_{2}1)(\forall f^{1})\big[  (\exists g^{1})(\forall n^{0})(f(\overline{g}n)=0)\asa \SS(f)=0  \big], 
\ee
The system $\FIVE^{\omega}\equiv \RCAo+(\SS^{2})$ proves the same $\Pi_{3}^{1}$-sentences as $\FIVE$ by \cite{yamayamaharehare}*{Theorem 2.2}.   
By definition, the Suslin functional $\SS^{2}$ can decide whether a $\Sigma_{1}^{1}$-formula as in the left-hand side of $(\SS^{2})$ is true or false.   We similarly define the functional $\SS_{k}^{2}$ which decides the truth or falsity of $\Sigma_{k}^{1}$-formulas from $\L_{2}$; we also define 
the system $\SIXK$ as $\RCAo+(\SS_{k}^{2})$, where  $(\SS_{k}^{2})$ expresses that $\SS_{k}^{2}$ exists.  
We note that the operators $\nu_{n}$ from \cite{boekskeopendoen}*{p.\ 129} are essentially $\SS_{n}^{2}$ strengthened to return a witness (if existant) to the $\Sigma_{n}^{1}$-formula at hand.  

\smallskip

\noindent
Thirdly, full second-order arithmetic $\Z_{2}$ is readily derived from $\cup_{k}\SIXK$, or from:
\be\tag{$\exists^{3}$}
(\exists E^{3}\leq_{3}1)(\forall Y^{2})\big[  (\exists f^{1})(Y(f)=0)\asa E(Y)=0  \big], 
\ee
and we therefore define $\Z_{2}^{\Omega}\equiv \RCAo+(\exists^{3})$ and $\Z_{2}^\omega\equiv \cup_{k}\SIXK$, which are conservative over $\Z_{2}$ by \cite{hunterphd}*{Cor.\ 2.6}. 
Despite this close connection, $\Z_{2}^{\omega}$ and $\Z_{2}^{\Omega}$ can behave quite differently, as discussed in e.g.\ \cite{dagsamIII}*{\S2.2}.   
The functional from $(\exists^{3})$ is also called `$\exists^{3}$', and we use the same convention for other functionals.  

\subsection{Higher-order computability theory}\label{HCT}

\appendix

\subsection{Reverse Mathematics}\label{prelim1}
We discuss Reverse Mathematics (Section \ref{introrm}) and introduce -in full detail- Kohlenbach's base theory of \emph{higher-order} Reverse Mathematics (Section \ref{rmbt}).
Some essential axioms, functionals, and notations may be found in Sections \ref{kkk} and \ref{lll}.
\subsubsection{Introduction}\label{introrm}
Reverse Mathematics (RM hereafter) is a program in the foundations of mathematics initiated around 1975 by Friedman (\cites{fried,fried2}) and developed extensively by Simpson (\cite{simpson2}).  
The aim of RM is to identify the minimal axioms needed to prove theorems of ordinary, i.e.\ non-set theoretical, mathematics. 

\smallskip

We refer to \cite{stillebron} for a basic introduction to RM and to \cite{simpson2, simpson1} for an overview of RM.  We expect basic familiarity with RM, but do sketch some aspects of Kohlenbach's \emph{higher-order} RM (\cite{kohlenbach2}) essential to this paper, including the base theory $\RCAo$ (Definition \ref{kase}).  

\smallskip

First of all, in contrast to `classical' RM based on \emph{second-order arithmetic} $\Z_{2}$, higher-order RM uses $\L_{\omega}$, the richer language of \emph{higher-order arithmetic}.  
Indeed, while the former is restricted to natural numbers and sets of natural numbers, higher-order arithmetic can accommodate sets of sets of natural numbers, sets of sets of sets of natural numbers, et cetera.  
To formalise this idea, we introduce the collection of \emph{all finite types} $\mathbf{T}$, defined by the two clauses:
\begin{center}
(i) $0\in \mathbf{T}$   and   (ii)  If $\sigma, \tau\in \mathbf{T}$ then $( \sigma \di \tau) \in \mathbf{T}$,
\end{center}
where $0$ is the type of natural numbers, and $\sigma\di \tau$ is the type of mappings from objects of type $\sigma$ to objects of type $\tau$.
In this way, $1\equiv 0\di 0$ is the type of functions from numbers to numbers, and  $n+1\equiv n\di 0$.  Viewing sets as given by characteristic functions, we note that $\Z_{2}$ only includes objects of type $0$ and $1$.    

\smallskip

Secondly, the language $\L_{\omega}$ includes variables $x^{\rho}, y^{\rho}, z^{\rho},\dots$ of any finite type $\rho\in \mathbf{T}$.  Types may be omitted when they can be inferred from context.  
The constants of $\L_{\omega}$ include the type $0$ objects $0, 1$ and $ <_{0}, +_{0}, \times_{0},=_{0}$  which are intended to have their usual meaning as operations on $\N$.
Equality at higher types is defined in terms of `$=_{0}$' as follows: for any objects $x^{\tau}, y^{\tau}$, we have
\be\label{aparth}
[x=_{\tau}y] \equiv (\forall z_{1}^{\tau_{1}}\dots z_{k}^{\tau_{k}})[xz_{1}\dots z_{k}=_{0}yz_{1}\dots z_{k}],
\ee
if the type $\tau$ is composed as $\tau\equiv(\tau_{1}\di \dots\di \tau_{k}\di 0)$.  
Furthermore, $\L_{\omega}$ also includes the \emph{recursor constant} $\mathbf{R}_{\sigma}$ for any $\sigma\in \mathbf{T}$, which allows for iteration on type $\sigma$-objects as in the special case \eqref{special}.  Formulas and terms are defined as usual.  
One obtains the sub-language $\L_{n+2}$ by restricting the above type formation rule to produce only type $n+1$ objects (and related types of similar complexity).        

\subsubsection{The base theory of higher-order Reverse Mathematics}\label{rmbt}
We introduce Kohlenbach's base theory $\RCAo$, first introduced in \cite{kohlenbach2}*{\S2}.
\bdefi\label{kase} 
The base theory $\RCAo$ consists of the following axioms.
\begin{enumerate}
 \renewcommand{\theenumi}{\alph{enumi}}
\item  Basic axioms expressing that $0, 1, <_{0}, +_{0}, \times_{0}$ form an ordered semi-ring with equality $=_{0}$.
\item Basic axioms defining the well-known $\Pi$ and $\Sigma$ combinators (aka $K$ and $S$ in \cite{avi2}), which allow for the definition of \emph{$\lambda$-abstraction}. 
\item The defining axiom of the recursor constant $\mathbf{R}_{0}$: for $m^{0}$ and $f^{1}$: 
\be\label{special}
\mathbf{R}_{0}(f, m, 0):= m \textup{ and } \mathbf{R}_{0}(f, m, n+1):= f(n, \mathbf{R}_{0}(f, m, n)).
\ee
\item The \emph{axiom of extensionality}: for all $\rho, \tau\in \mathbf{T}$, we have:
\be\label{EXT}\tag{$\textsf{\textup{E}}_{\rho, \tau}$}  
(\forall  x^{\rho},y^{\rho}, \varphi^{\rho\di \tau}) \big[x=_{\rho} y \di \varphi(x)=_{\tau}\varphi(y)   \big].
\ee 
\item The induction axiom for quantifier-free formulas of $\L_{\omega}$.
\item $\QFAC^{1,0}$: the quantifier-free Axiom of Choice as in Definition \ref{QFAC}.
\end{enumerate}
\edefi
\noindent
Note that variables (of any finite type) are allowed in quantifier-free formulas of the language $\L_{\omega}$: only quantifiers are banned.
Recursion as in \eqref{special} is called \emph{primitive recursion}; the class of functionals obtained from $\mathbf{R}_{\rho}$ for all $\rho \in \mathbf{T}$ is called \emph{G\"odel's system $T$} of all (higher-order) primitive recursive functionals. 
\bdefi\label{QFAC} The axiom $\QFAC$ consists of the following for all $\sigma, \tau \in \textbf{T}$:
\be\tag{$\QFAC^{\sigma,\tau}$}
(\forall x^{\sigma})(\exists y^{\tau})A(x, y)\di (\exists Y^{\sigma\di \tau})(\forall x^{\sigma})A(x, Y(x)),
\ee
for any quantifier-free formula $A$ in the language of $\L_{\omega}$.
\edefi
As discussed in \cite{kohlenbach2}*{\S2}, $\RCAo$ and $\RCA_{0}$ prove the same sentences `up to language' as the latter is set-based and the former function-based.   
This conservation results is obtained via the so-called $\ECF$-interpretation, which we now discuss. 
\begin{rem}[The $\ECF$-interpretation]\label{ECF}\rm
The (rather) technical definition of $\ECF$ may be found in \cite{troelstra1}*{p.\ 138, \S2.6}.
Intuitively, the $\ECF$-interpretation $[A]_{\ECF}$ of a formula $A\in \L_{\omega}$ is just $A$ with all variables 
of type two and higher replaced by type one variables ranging over so-called `associates' or `RM-codes' (see \cite{kohlenbach4}*{\S4}); the latter are (countable) representations of continuous functionals.  
The $\ECF$-interpretation connects $\RCAo$ and $\RCA_{0}$ (see \cite{kohlenbach2}*{Prop.\ 3.1}) in that if $\RCAo$ proves $A$, then $\RCA_{0}$ proves $[A]_{\ECF}$, again `up to language', as $\RCA_{0}$ is 
formulated using sets, and $[A]_{\ECF}$ is formulated using types, i.e.\ using type zero and one objects.  
\end{rem}
In light of the widespread use of codes in RM and the common practise of identifying codes with the objects being coded, it is no exaggeration to refer to $\ECF$ as the \emph{canonical} embedding of higher-order into second-order arithmetic. 

\smallskip

Finally as noted above, Theorem \ref{cant2} is provable in the base theory. 
\begin{thm}[\cite{simpson2}*{II.4.9}] The following is provable in $\RCA_{0}$.
For any sequence of real numbers $(x_{n})_{n\in \N}$, there is a real $y$ different from $x_{n}$ for all $n\in \N$.
\end{thm}

\subsubsection{Notations and the like}\label{kkk}
We introduce the usual notations for common mathematical notions, like real numbers, as also introduced in \cite{kohlenbach2}.  
\begin{defi}[Real numbers and related notions in $\RCAo$]\label{keepintireal}\rm~
\begin{enumerate}
 \renewcommand{\theenumi}{\alph{enumi}}
\item Natural numbers correspond to type zero objects, and we use `$n^{0}$' and `$n\in \N$' interchangeably.  Rational numbers are defined as signed quotients of natural numbers, and `$q\in \Q$' and `$<_{\Q}$' have their usual meaning.    
\item Real numbers are coded by fast-converging Cauchy sequences $q_{(\cdot)}:\N\di \Q$, i.e.\  such that $(\forall n^{0}, i^{0})(|q_{n}-q_{n+i}|<_{\Q} \frac{1}{2^{n}})$.  
We use Kohlenbach's `hat function' from \cite{kohlenbach2}*{p.\ 289} to guarantee that every $q^{1}$ defines a real number.  
\item We write `$x\in \R$' to express that $x^{1}:=(q^{1}_{(\cdot)})$ represents a real as in the previous item and write $[x](k):=q_{k}$ for the $k$-th approximation of $x$.    
\item Two reals $x, y$ represented by $q_{(\cdot)}$ and $r_{(\cdot)}$ are \emph{equal}, denoted $x=_{\R}y$, if $(\forall n^{0})(|q_{n}-r_{n}|\leq {2^{-n+1}})$. Inequality `$<_{\R}$' is defined similarly.  
We sometimes omit the subscript `$\R$' if it is clear from context.           
\item Functions $F:\R\di \R$ are represented by $\Phi^{1\di 1}$ mapping equal reals to equal reals, i.e.\ extensionality as in $(\forall x , y\in \R)(x=_{\R}y\di \Phi(x)=_{\R}\Phi(y))$.\label{EXTEN}
\item The relation `$x\leq_{\tau}y$' is defined as in \eqref{aparth} but with `$\leq_{0}$' instead of `$=_{0}$'.  Binary sequences are denoted `$f^{1}, g^{1}\leq_{1}1$', but also `$f,g\in C$' or `$f, g\in 2^{\N}$'.  Elements of Baire space are given by $f^{1}, g^{1}$, but also denoted `$f, g\in \N^{\N}$'.
\item For a binary sequence $f^{1}$, the associated real in $[0,1]$ is $\r(f):=\sum_{n=0}^{\infty}\frac{f(n)}{2^{n+1}}$.\label{detrippe}
\item Sets of type $\rho$ objects $X^{\rho\di 0}, Y^{\rho\di 0}, \dots$ are given by their characteristic functions $F^{\rho\di 0}_{X}\leq_{\rho\di 0}1$, i.e.\ we write `$x\in X$' for $ F_{X}(x)=_{0}1$. \label{koer} 
\end{enumerate}
\end{defi}
For completeness, we list the following notational convention for finite sequences.  
\begin{nota}[Finite sequences]\label{skim}\rm
The type for `finite sequences of objects of type $\rho$' is denoted $\rho^{*}$, which we shall only use for $\rho=0,1$.  
Since the usual coding of pairs of numbers goes through in $\RCAo$, we shall not always distinguish between $0$ and $0^{*}$. 
Similarly, we assume a fixed coding for finite sequences of type $1$ and shall make use of the type `$1^{*}$'.  
In general, we do not always distinguish between `$s^{\rho}$' and `$\langle s^{\rho}\rangle$', where the former is `the object $s$ of type $\rho$', and the latter is `the sequence of type $\rho^{*}$ with only element $s^{\rho}$'.  The empty sequence for the type $\rho^{*}$ is denoted by `$\langle \rangle_{\rho}$', usually with the typing omitted.  

\smallskip

Furthermore, we denote by `$|s|=n$' the length of the finite sequence $s^{\rho^{*}}=\langle s_{0}^{\rho},s_{1}^{\rho},\dots,s_{n-1}^{\rho}\rangle$, where $|\langle\rangle|=0$, i.e.\ the empty sequence has length zero.
For sequences $s^{\rho^{*}}, t^{\rho^{*}}$, we denote by `$s*t$' the concatenation of $s$ and $t$, i.e.\ $(s*t)(i)=s(i)$ for $i<|s|$ and $(s*t)(j)=t(|s|-j)$ for $|s|\leq j< |s|+|t|$. For a sequence $s^{\rho^{*}}$, we define $\overline{s}N:=\langle s(0), s(1), \dots,  s(N-1)\rangle $ for $N^{0}<|s|$.  
For a sequence $\alpha^{0\di \rho}$, we also write $\overline{\alpha}N=\langle \alpha(0), \alpha(1),\dots, \alpha(N-1)\rangle$ for \emph{any} $N^{0}$.  By way of shorthand, 
$(\forall q^{\rho}\in Q^{\rho^{*}})A(q)$ abbreviates $(\forall i^{0}<|Q|)A(Q(i))$, which is (equivalent to) quantifier-free if $A$ is.   
\end{nota}
\subsubsection{Some axioms and functionals}\label{lll}
As noted in Section \ref{intro}, the logical hardness of a theorem is measured via what fragment of the comprehension axiom is needed for a proof.  
For this reason, we introduce some axioms and functionals related to \emph{higher-order comprehension} in this section.
We are mostly dealing with \emph{conventional} comprehension here, i.e.\ only parameters over $\N$ and $\N^{\N}$ are allowed in formula classes like $\Pi_{k}^{1}$ and $\Sigma_{k}^{1}$.  

\smallskip

First of all, the following functional is clearly discontinuous at $f=11\dots$; in fact, $(\exists^{2})$ is equivalent to the existence of $F:\R\di\R$ such that $F(x)=1$ if $x>_{\R}0$, and $0$ otherwise (\cite{kohlenbach2}*{\S3}).  This fact shall be repeated often.  
\be\label{muk}\tag{$\exists^{2}$}
(\exists \varphi^{2}\leq_{2}1)(\forall f^{1})\big[(\exists n)(f(n)=0) \asa \varphi(f)=0    \big]. 
\ee
Related to $(\exists^{2})$, the functional $\mu^{2}$ in $(\mu^{2})$ is also called \emph{Feferman's $\mu$} (\cite{avi2}).
\begin{align}\label{mu}\tag{$\mu^{2}$}
(\exists \mu^{2})(\forall f^{1})\big[ (\exists n)(f(n)=0) \di [f(\mu(f))=0&\wedge (\forall i<\mu(f))(f(i)\ne 0) ]\\
& \wedge [ (\forall n)(f(n)\ne0)\di   \mu(f)=0]    \big], \notag
\end{align}
We have $(\exists^{2})\asa (\mu^{2})$ over $\RCAo$ and $\ACAo\equiv\RCAo+(\exists^{2})$ proves the same sentences as $\ACA_{0}$ by \cite{hunterphd}*{Theorem~2.5}. 

\smallskip

Secondly, the functional $\SS^{2}$ in $(\SS^{2})$ is called \emph{the Suslin functional} (\cite{kohlenbach2}).
\be\tag{$\SS^{2}$}
(\exists\SS^{2}\leq_{2}1)(\forall f^{1})\big[  (\exists g^{1})(\forall n^{0})(f(\overline{g}n)=0)\asa \SS(f)=0  \big], 
\ee
The system $\FIVE^{\omega}\equiv \RCAo+(\SS^{2})$ proves the same $\Pi_{3}^{1}$-sentences as $\FIVE$ by \cite{yamayamaharehare}*{Theorem 2.2}.   
By definition, the Suslin functional $\SS^{2}$ can decide whether a $\Sigma_{1}^{1}$-formula as in the left-hand side of $(\SS^{2})$ is true or false.   We similarly define the functional $\SS_{k}^{2}$ which decides the truth or falsity of $\Sigma_{k}^{1}$-formulas from $\L_{2}$; we also define 
the system $\SIXK$ as $\RCAo+(\SS_{k}^{2})$, where  $(\SS_{k}^{2})$ expresses that $\SS_{k}^{2}$ exists.  
We note that the operators $\nu_{n}$ from \cite{boekskeopendoen}*{p.\ 129} are essentially $\SS_{n}^{2}$ strengthened to return a witness (if existant) to the $\Sigma_{n}^{1}$-formula at hand.  

\smallskip

\noindent
Thirdly, full second-order arithmetic $\Z_{2}$ is readily derived from $\cup_{k}\SIXK$, or from:
\be\tag{$\exists^{3}$}
(\exists E^{3}\leq_{3}1)(\forall Y^{2})\big[  (\exists f^{1})(Y(f)=0)\asa E(Y)=0  \big], 
\ee
and we therefore define $\Z_{2}^{\Omega}\equiv \RCAo+(\exists^{3})$ and $\Z_{2}^\omega\equiv \cup_{k}\SIXK$, which are conservative over $\Z_{2}$ by \cite{hunterphd}*{Cor.\ 2.6}. 
Despite this close connection, $\Z_{2}^{\omega}$ and $\Z_{2}^{\Omega}$ can behave quite differently, as discussed in e.g.\ \cite{dagsamIII}*{\S2.2}.   
The functional from $(\exists^{3})$ is also called `$\exists^{3}$', and we use the same convention for other functionals.  
\end{comment}
\bye